\documentclass[10pt]{amsart}
\usepackage{url}
\usepackage{hyperref,amssymb}
\hypersetup{colorlinks=true,linkcolor=blue,
urlcolor=cyan,citecolor=magenta} 
\usepackage[hyperpageref]{backref}
\usepackage{frcursive,calrsfs,mathrsfs}
\usepackage{xcolor}
\newtheorem{theorem}{Theorem}[section]
\newtheorem{lemma}[theorem]{Lemma}
\newtheorem{corollary}[theorem]{Corollary}
\newtheorem{conjecture}[theorem]{Conjecture}
\newtheorem{proposition}[theorem]{Proposition}
\newtheorem{remark}[theorem]{Remark}
\newtheorem{remarks}[theorem]{Remarks}

\newtheorem{examples}[theorem]{Examples}
\newtheorem{definition}[theorem]{Definition}
\newtheorem{definitions}[theorem]{Definitions}
\newtheorem{notations}[theorem]{Notations}
\numberwithin{equation}{section}
\newcommand{\Aa}{{\mathbb{A}}}
\newcommand{\Q}{{\mathbb{Q}}}
\newcommand{\Z}{{\mathbb{Z}}}
\newcommand{\F}{{\mathbb{F}}}
\newcommand{\C}{{\mathbb{C}}}
\newcommand{\R}{{\mathbb{R}}}
\newcommand{\ds}{\displaystyle}
\newcommand{\sst}{\scriptstyle}
\newcommand{\ov}{\overline}
\newcommand{\wt}{\widetilde}
\newcommand{\wh}{\widehat}
\newcommand{\ft}{\footnotesize}
\newcommand{\ns}{\normalsize}
\newcommand{\cl}{c\hskip-1pt{{\ell}}}
\newcommand{\BI}{\hbox{\bf$\textbf{I}$}}
\newcommand{\BH}{\hbox{\bf$\textbf{H}$}}
\newcommand{\CH}{{\mathcal H}}
\newcommand{\BE}{\hbox{\bf$\textbf{E}$}}
\newcommand{\CE}{{\mathcal E}}
\newcommand{\BF}{\hbox{\bf$\textbf{F}$}}
\newcommand{\CF}{{\mathcal F}}
\newcommand{\BC}{\hbox{\bf$\textbf{C}$}}
\newcommand{\BL}{\hbox{\bf$\textbf{L}$}}
\newcommand{\BM}{\hbox{\bf$\textbf{M}$}}
\newcommand{\CM}{{\mathcal M}}
\newcommand{\BA}{\hbox{\bf$\textbf{A}$}}
\newcommand{\CA}{{\mathcal A}}
\newcommand{\CR}{{\mathcal R}}
\newcommand{\CW}{{\mathcal W}}
\newcommand{\CU}{{\mathcal U}}
\newcommand{\CK}{{\mathcal K}}
\newcommand{\CT}{{\mathcal T}}
\newcommand{\CX}{{\mathcal X}}
\newcommand{\G}{{\mathcal G}}
\newcommand{\BB}{\hbox{\bf$\textbf{B}$}}
\newcommand{\BS}{\hbox{\bf$\textbf{S}$}}
\newcommand{\BZ}{\hbox{\bf$\textbf{Z}$}}
\newcommand{\order}{\raise0.8pt \hbox{${\scriptstyle \#}$}}
\newcommand{\oorder}{\raise0.8pt \hbox{\tiny${\scriptstyle \#}$}}
\newcommand{\lien}{\mathrel{\mkern-4mu}}
\newcommand{\too}{\relbar\lien\rightarrow}
\newcommand{\tooo}{\relbar\lien\relbar\lien\too}

\newcommand{\plus}{\ds\mathop{\raise 0.5pt \hbox{$\bigoplus$}}\limits}
\newcommand{\prd}{\ds\mathop{\raise 1.0pt \hbox{$\prod$}}\limits}
\newcommand{\sm}{\ds\mathop{\raise 1.0pt \hbox{$\sum$}}\limits}
\newcommand{\ffrac}[2]{\hbox{\ft $\displaystyle\frac{#1}{#2}$}}

\newcommand{\Gal}{{\rm Gal}}
\newcommand{\id}{{\rm id}}
\newcommand{\Norm}{\hbox{\bf$\textbf{N}$}}
\newcommand{\J}{\hbox{\bf$\textbf{J}$}}
\newcommand{\g}{\hbox{\bf$\textbf{g}$}}
\newcommand{\Nu}{\hbox{\Large $\nu\!$}}
\newcommand{\Hom}{{\rm H}}

\newcommand{\tor}{{\rm tor}}
\newcommand{\Ker}{{\rm Ker}}
\newcommand{\pr}{{\rm pr}}
\newcommand{\ta}{{\rm ta}}
\newcommand{\ram}{{\rm ram}}
\newcommand{\ab}{{\rm ab}}
\newcommand{\alg}{{\rm alg}}
\newcommand{\ar}{{\rm ar}}
\newcommand{\nr}{{\rm nr}}
\newcommand{\an}{{\rm an}}
\newcommand{\val}{{\rm val}}
\newcommand{\bp}{{\rm bp}}
\newcommand{\cyc}{{\rm cyc}}
\newcommand{\Kappa}{{\hbox{\large $\kappa$}}}
\newcommand{\BKappa}{\hbox{\Large $\kappa$}}

\author[Georges Gras]{Georges Gras}

\address{Villa la Gardette, 4, chemin Ch\^ateau Gagni\`ere, 
38520, Le Bourg d'Oisans \rm {\url{http://orcid.org/0000-0002-1318-4414} } }
\email{g.mn.gras@wanadoo.fr}



\begin{document}

\title[Notion of abelian Arithmetic $\varphi$-object]
{Notion of abelian Arithmetic $\varphi$-object \\
for the study of $p$-class groups \\
and $p$-ramified torsion groups}

\date{July 3, 2023}

\begin{abstract} 
We revisit, in an elementary way, the \textit {classical statement} of various ``Main Conjectures'' 
for $p$-class groups $\CH_K$ and $p$-ramified torsion groups $\CT_K$ of abelian 
fields $K$, in the non semi-simple case $p \mid [K : \Q]$. The classical ``algebraic'' 
definition of the $p$-adic isotopic components, $\CH^\alg_{K,\varphi}$, used in the literature, 
is inappropriate with respect to analytical formulas. For that reason we have introduced, 
in the 1970's, an ``arithmetic'' definition, $\CH^\ar_{K,\varphi}$, in perfect correspondence 
with all analytical formulas and giving a natural ``Main Conjecture'', still unproved for real 
fields in the non semi-simple case. The two notions coincide for relative 
class groups $\CH_K^-$ and groups $\CT_K$ since, in $p$-extensions, 
transfer maps are injective for these groups but not necessarily for real class groups. 
Numerical evidence of the gap between 
the two notions is given (Examples \ref{Ex1}, \ref{Ex2}) and PARI calculations
corroborate that the true Real Main Conjecture for $K$ writes on the form
$\order \CH^\ar_{K,\varphi} = \order (\CE_K / \wh \CE_K \, \CF_{\!K})_{\varphi}$, 
in terms of units $\CE_K$, $\wh \CE_K$ (units of the strict subfields) and $\CF_K$ 
(Leopoldt's cyclotomic units). A recent approach, conjecturing the capitulation of $\CH_K$ 
in some auxiliary cyclotomic extensions $K(\mu_\ell^{})$, proves the difficult real case.
\end{abstract}

\maketitle

\vspace{-0.6cm}
\tableofcontents

\section*{Foreword and preliminary remarks} 
This survey provides improvements, new results, numerical illustrations (with PARI
programs) and some history, regarding our original articles \cite{Gra, Gra0}. 
These two papers were written, in French, with illegible fonts due to the use 
of "typits" on typewriters and hand written characters, for mathematical symbols~! 
So they were hardly accessible and only \cite{Gra1} is cited in replacement of them.
This survey also mention, in Subsection \ref{reminders}, 
pioneering references, as well as some significant Leopldt's papers on cyclotomy 
\cite{15Leo, 16Leo}, written in german in the 1950/1960's. 

\smallskip
As the Referee pointed us, one must avoid any confusion with the \textit {Iwasawa Main Conjecture},
dealing, for instance, with cyclotomic $\Z_p$-extensions; so, the Conjectures for the case of finite 
abelian extensions (giving the most precise relations with analytic information) will be called ``Finite 
Abelian Main Conjectures'' in this paper (Finite AMC for short). This may be legitimate since 
beyond the Iwasawa Main Conjecture (or Mazur--Wiles' Main Theorem and generalizations) 
our purposes and conjectures deal always with \textbf{finite abelian extensions}, a context which, 
of course, must apply to the finite layers of the cyclotomic $\Z_p$-extension. Moreover, Thaine's 
technique and our new philosophy, using capitulation of classes in auxiliary cyclotomic 
extensions $K(\mu_\ell^{})$, strengthen the interest of the finite cases.
 
\smallskip
The Finite AMC (giving analytic expressions 
of annihilators and orders of $p$-adic isotopic components of class groups) 
that we revisit here, were first stated (\textit {especially in the non semi-simple case}) 
in our papers mentioned above (but not in \cite{Gra1}, as erroneously stated by 
some authors), and were given at the meeting ``Journ\'ees arithm\'etiques de Caen'' (1976) 
as it is correctly recalled for instance in \cite{Sol, Rib1}. This gives the occasion to mention 
that \cite{Gra1} (only recalling the statements of the conjectues in the semi-simple case) is 
especially devoted to a method using formal series, giving non trivial congruences when 
$p$-adic $\BL$-functions have a trivial zero; for instance we proved the following 
complement of Ankeny--Artin--Chowla--Kudo congruences (cf. \cite{AAC,Kud} and
\cite[Theorem 5.37]{Was}):

\begin{proposition}\label{AACK}
Let $f \equiv 0 \pmod 3$ be the conductor of a real quadratic field $K$; 
we consider the case $f/3 \equiv -1 \pmod 3$ (``special case'' when $3$ splits in the
mirror field $K' := \Q(\sqrt{-f/3})$). Let $\varepsilon = t + u \sqrt f$, $t, u >0$, be the 
fundamental unit of $K$ and let $h$ and $h'$ be the class numbers of $K$ and $K'$, 
respectively. Then $h \!\cdot\! t \!\cdot\! u +h' \equiv 0 \pmod 3$.
\end{proposition}

A program (Appendix \ref{P1}) only checks the congruence. 
But this ana\-lytic result, which seems unknown, is perhaps off topic for our purpose. 

\smallskip
The Finite AMC has been proven in the semi-simple case, then in the non semi-simple 
one for \textit {imaginary relative class groups} and mainly in the framework of Iwasawa's theory 
(a large overview on the precise proofs and classical references are given in Washington's 
book \cite[Chapters 6, 8, 13, 15]{Was}).

\smallskip
The \textit {non semi-simple case of even $p$-adic characters $\varphi$}
(real case), was less 
understood because of a problematic definition of $p$-adic isotopic components and
cyclotomic units; but at the time, we proposed another more natural conjectural 
context, still unproved, for which the definition of ``Arithmetic $\varphi$-objects'' has become 
essential since the distinction between ``Algebraic'' definitions (classical framework) and 
``Arithmetic'' definitions is crucial regarding analytic formulas (we shall give more comments 
in Remarks \ref{e0}).

\smallskip
Let $\G := \Gal(\Q^\ab/\Q)$ be the Galois group of the maximal abelian extension 
$\Q^\ab$ of $\Q$ and denote by $K$ a subfield of finite degree of $\Q^\ab$.
In fact, since abelian arithmetic deals with cyclic fields ``$K=K_\chi$'' indexed by 
rational characters, there is no restriction to take cyclic $K$'s in any result or comment.
The present article is divided into the following three parts, after an Introduction giving 
a brief description about the story (rather prehistory) that led to the numerous approaches 
giving, under some assumptions, proofs of a ``Main Theorem'':

\smallskip
(i) An algebraic part giving a systematic study of families $(\BM_K)_K$
of $\Z[\G]$-modules and of the $\Z_p[\G]$-modules 
$\CM_K := \BM_K \otimes \Z_p$, including the non semi-simple case 
$p \mid [K : \Q]$. This study leads to the definition of sub-modules 
$\CM^\alg_\varphi$ (algebraic) and $\CM^\ar_\varphi$ (arithmetic),  
indexed by the set of irreducible $p$-adic characters $\varphi$ of $\G$. 
The difference between $\CM^\alg_\varphi$ (used in all the literature) and 
$\CM^\ar_\varphi$ is that the first one relates to algebraic norms 
$\Nu_{k/k'} \in \Z[\Gal(k/k')]$ for their properties in relative sub-extensions 
of $K/\Q$, while the second one uses arithmetic norms $\Norm_{k/k'}$, 
the gap being given by the relation:
\[\Nu_{k/k'} = \J_{k/k'} \circ \Norm_{k/k'},\] 
where the transfer maps $\J_{k/k'}$ are often non injective in 
$p$-extensions (see \S\,\ref{tensor} for examples justifying 
Definition \ref{defI4e} for the statement of the Finite AMC
and \S\,\ref{mainresults} for the main properties). Moreover, 
the ``arithmetic'' point of view is naturally related to the formula:
\[\order \CM_K = \prd_{\varphi \in \Phi_K} \order \CM^\ar_\varphi \ \,
\hbox{(Theorems \ref{theoI5} and \ref{theoI2bis})}, \]
where the $\order \CM^\ar_\varphi$'s have (conjecturally) analytic 
expressions, contrary to the $\order \CM^\alg_\varphi$'s which do not 
always fulfill this relation.

\smallskip
(ii) An arithmetic part where we apply the above results to   
$p$-class groups $\CH_K$, $K$ real or imaginary, then to torsion groups 
$\CT_K$ of the Galois group of the maximal $p$-ramified abelian pro-$p$-extension 
of $K$ real. 

\smallskip
For rational characters $\chi$ and $p$-adic characters
$\varphi \mid \chi$, we define the ``Class Invariants'' $m^\alg_\varphi {\sst(\CH)}$ 
(algebraic), $m^\ar_\varphi {\sst(\CH)}$, $m^\ar_\varphi {\sst(\CT)}$ (arithmetic) 
then, in \S\,\ref{analytics}, the corresponding ``Analytic Invariants'' 
$m^\an_\varphi {\sst(\CH)}$, $m^\an_\varphi {\sst(\CT)}$ suggested by 
the analytic formulas of the arithmetic $\chi$-components deduced from Leopoldt's
Theorem \ref{chiformula} (cf. Theorems \ref{theoII2}, \ref{theotorsion2}, \ref{theoIII2}) 
and we develop the problem of their comparison. 
We conjecture a new annihilation theorem for $\CH^\ar_\varphi$
in the real non semi-simple case (Conjecture~\ref{annihilationthm}).

\smallskip
In \S\,\ref{mysterious}, we shed new light on the proof of the Finite AMC
in the real semi-simple case for $K$, in the spirit of Thaine's theorem described in
Washington's book, and we give numerical illustrations. It becomes clear that 
\textit{the knowledge of the sole cyclotomic unit $\eta_K^{}$ of $K$} 
contains, by means of very elementary arithmetic, all the information on annihilation 
and orders of the $\varphi$-components of its $p$-class group.
A new observation is that Thaine's method uses auxiliary cyclotomic extensions
$K(\mu_\ell^{})$ with \textit {$\ell$ totally split in $K$}, while our approach in 
\cite{Gra13,Gra14,Gra15} uses the same auxiliary extensions, but with \textit {$\ell$ totally 
inert in $K$}.

\smallskip
(iii) An illustration, in the semi-simple case, is given with cyclic cubic fields for $p\equiv 1 \pmod 3$,
as well as a PARI program computing the above invariants, which was not possible in the 1970's. 
Since the submission of this paper, more computations have been done and confirm the theoretical 
claims. Since numerical experiments have some importance and take much place, we 
report in the Appendix, PARI programs, tables and explanations for their use; the
programs may be copied and pasted from any pdf-file (e.g., \url{https://arxiv.org/pdf/2112.02865.pdf}).

\section{Introduction and brief historical survey}

\subsection{Main bibliographic reminders}\label{reminders}
It is difficult to give here the full story of such a subject, from 
Bernoulli, Kummer, Herbrand classical context, the initiating work of 
Iwasawa, Leopoldt, Greenberg, on the conjecture, then the deep results obtained by 
Ribet, Mazur, Wiles, Thaine, Rubin, Kolyvagin, Solomon, Greither, Coates, Sinnott, 
and others, on cyclotomy and $p$-adic $\BL$-functions. 
Several papers also give the Iwasawa formulation of the Main Theorem 
(see e.g., \cite{Gree1, Gree2}), in terms of $p$-adic $\BL$-functions, 
a generalizable feature to many fields. The fundamental difference, regarding 
finite $p$-extensions, is that, in Iwasawa's theory, capitulation 
kernels are hidden in statements using pseudo-isomorphisms, whence only giving 
results for the projective limit of the $p$-class groups in the $\Z_p$-extensions and, in 
general, no precise information is available in the finite layers (it's quite clear
in a numerical setting that any possible structure occurs in the first layers, up
to the algebraic regularity predicted by Iwasawa's theory; see for instance the
numerical computations given in Kraft--Schoof--Pagani
\cite{KS1995,Paga2022}). A clear result
about capitulation kernels is given in Grandet--Jaulent \cite[Th\'eor\`eme, p. 214]{GJ}

\smallskip
Let's give less known contributions of the beginnings:

\smallskip
We refer, for a very nice story of pioneering works, to Ribet \cite{Rib1, Rib2}, for 
detailed proofs of Iwasawa Main Conjecture to Washington \cite[Chap. 15]{Was}
following techniques initiated by Thaine then Kolyvagin, Ribet (exposed by Lang \cite{Lg}).
A Bourbaki Seminar, by Perrin-Riou \cite{PR}, gives a significant lecture (with
an impressive bibliography) on the works of Kolyvagin, Rubin and others about the 
Main Conjectures for number fields and elliptic curves.

\smallskip
The story is also given in the famous Mazur--Wiles paper, where the attribution 
of the various statements of the conjecture (in the semi-simple case) is  
accurately discussed (see \cite[\S\,1 and \S\,10\,(i, ii)]{MW} for more
comments on the works of Iwasawa, Leopoldt, Greenberg and us), even
if some references are missing.

\smallskip
Finally, proofs of our conjecture for the relative $p$-class groups $\CH^-$ and the
real torsion groups $\CT$ of the Galois groups of the maximal abelian $p$-ramified 
pro-$p$-extensions were given (Solomon for $\CH^-$ and $p \ne 2$ 
\cite[Theorem II.1]{Sol}, Greither for $\CH^-$, $\CT$ with $p \geq 2$ 
and $\CH^+$, but in a semi-simple context \cite[Theorems A, B, C, 4.14, 
Corollary 4.15]{Grei}). Let's mention the proof by Rubin \cite{Rub}, 
from the Kolyvagin Euler systems \cite{Kol} used in above proofs. 

\smallskip
Many complementary works about the order or the annihilation of 
the $\CH_\varphi$'s, for irreducible $p$-adic characters $\varphi$, 
were published before or after the decisive proofs (e.g., 
\cite{Gra1, Gil1, Gra4, Or1, Or2, GK0, BN, All1, BM, GK1, All2, Gra8, GK2,
Jau6, Jau7, Jau8}).
Mention a result of Oriat using reflection theorem \cite[Th\'eor\`eme, p.\,333]{Or2}. 

\smallskip
In the same way, it is hopeless to outline all generalizations 
giving ``Main Conjectures'' in other contexts than the absolute 
abelian case (e.g., \cite{Dar,MR, CoLi1, DK, CoLi2, BBDS, BDSS,Vig}), 
using essentially the technique of Kolyvagin's Euler systems; 
an expository book may be \cite{CS} for recent works, but excluding the story of the 
origins of the Main Conjecture as explained in Solomon--Greither papers \cite{Sol, Grei}, 
Washington's book \cite{Was} and Ribet's Lectures \cite{Rib1, Rib2}.

\smallskip
In another direction, we refer to enlargements of the algebraic/arith\-metic aspects of
$p$-adic characters in the area of metabelian Galois groups by Jaulent, with applications 
to class groups and units (see for instance \cite[Théor\`eme 1 and consequences]{Jau1}, 
\cite{Jau2, Jau3} in a class field theory context, then \cite{Lec, SchStu} in a 
geometric or Galois cohomology context). 

\smallskip
Due to the huge number of articles dealing with the concept of ``Main Conjecture'', 
many recent (or not) articles may have escaped our notice.

\subsection{Introduction of Arithmetic \texorpdfstring{$\varphi$}{Lg}-objects}
Nevertheless, all these works deal with an \textit {algebraic definition of the
$\varphi$-class groups} $\CH^\alg_\varphi$, from the $p$-class group $\CH_K$
(for irreducible $p$-adic characters $\varphi$); 
that is to say, when $G_K := \Gal(K/\Q)$ is cyclic, of order $g$ (i.e., $K=K_\chi$ 
is the fixed field by the kernel of a rational character $\chi$ as we have explained):
\[\CH^\alg_\varphi := \CH_K \otimes^{}_{\Z_p[G_K]} \Z_p[\mu_g^{}],\ \,
\hbox{for all $\varphi \mid \chi$}, \]
with the $\Z_p[\mu_g^{}]$-action $\sigma \in G_K \mapsto \psi (\sigma)$
($\psi \mid \varphi$ of degree $1$ and order $g$). 

\smallskip
Put $K=K' K_0$, where $[K_0 : \Q]$ is prime to $p$ and $[K' : \Q]$ a $p$-power.

We then prove (Theorem \ref{theoI2}\,(ii)) that, from the expression: 
\[\CH^\alg_\chi = \big\{ x \in \CH_K, \  \Nu_{K/k}(x) = 1,
\  \forall \, k \varsubsetneqq K \}\]
(Theorem \ref{theoI1}, where $\Nu_{K/k}$ is the algebraic norm), one gets:
\[\CH^\alg_\varphi = \big( \{x \in \CH_K, \ \Nu_{K/k}(x) = 1,\, 
\forall \, k \varsubsetneqq K \} \big)_{\varphi_0^{}}, \]
(where $\varphi = \varphi_0^{} \varphi_p$, $\varphi_0^{}$ of prime to $p$ order,
$\varphi_p$ of $p$-power order and where $(\ \ )_{\varphi_0^{}}$ denotes
a ${\varphi_0^{}}$-component obtained with the corresponding semi-simple 
idempotent), contrary to our definition:
\[\CH^\ar_\varphi := \big( \{x \in \CH_K, \ \Norm_{K/k}(x) = 1,\, \forall \, 
k \varsubsetneqq K \} \big)_{\varphi_0^{}}, \]
where $\Norm_{K/k}$ is the arithmetic norm. 

\smallskip
See \S\,\ref{maindef} for equivalent characterizations of $\CH^\alg_\varphi $ and 
$\CH^\ar_\varphi$ using local cyclotomic polynomials $P_\varphi$,
then for a summary of the main properties and results of the paper.

\smallskip
In the non semi-simple case $p \mid g$, the distinction between algebraic 
and arithmetic $\varphi$-components is not done in the literature. 
This does not matter for relative $p$-class groups $\CH_K^-$ and torsion 
$p$-groups $\CT_K$ since we will prove that the two notions coincide 
(Theorems \ref{theoII1}, \ref{theotorsion}); so the case of these invariants 
is definitely solved, contrary to that of $\varphi$-components of
$p$-class groups of real fields $K$ in the non semi-simple 
case deduced from the ``$\chi$-formulas'' given in Theorem \ref{theoIII2}
and the important relation that we talked about: 
\[\order \CH_K = \prd_{\varphi \in \Phi_K} \order \CH^\ar_\varphi
\ \, \hbox{(Theorems \ref{theoI5}, \ref{theoI2bis})}.\]
We compare the two definitions $\CH^\alg$, $\CH^\ar$ in \S\,\ref{tensor} 
and Appendix \ref{ex12}, with numerical illustrations showing the gap 
between them and involving capitulation phenomenon of $p$-classes
in $p$-extensions (see the detailed Examples \ref{Ex1}, \ref{Ex2}).

\subsection{Relation between the modules \texorpdfstring{$\CH$}{Lg} 
and \texorpdfstring{$\CT$}{Lg}}
If one considers, in the real case, the $\Z_p[\G]$-modules $\CT_K$, one gets,
for them, an easier annihilation theorem from the $p$-adic Mellin transform 
of Stickelberger elements (see \S\,\ref{annTphi}). Moreover, the norm maps 
$\Norm_{k/k'}$ are surjective and the transfer maps $\J_{k/k'}$ are injective 
under Leopoldt's conjecture \cite[Th\'eor\`eme I.1]{Gra41}, \cite{Jau3, Ngtor, Jau4} 
(collected in \cite[Theorem IV.2.1]{Gra6}); so this family behaves as that of relative 
class groups, which allows an obvious statement of the Finite AMC 
and then its proof with similar techniques, as done for instance in \cite{Grei}.

\smallskip
The order of the $p$-group $\CT_K$ is closely related to the $p$-adic 
$\BL$-functions ``at $s=1$'' \cite{2Coa} and a particularity of $\CT_K$ is its 
interpretation by means of the three $\Z_p[\G]$-modules $\CH^{\,\cyc}_K$, 
${\mathcal R}_K$ and ${\mathcal W}_K$; see \cite[Lemma III.4.2.4]{Gra6}
leading to the exact sequence \eqref{hrw} and the formula
$\order\CT_K = \order \CH^{\,\cyc} _K \!\! \cdot \order{\mathcal R}_K 
\cdot \order{\mathcal W}_K$, where ${\mathcal W}_K$ 
is an easy canonical invariant depending on local $p$-roots of unity, 
${\mathcal R}_K$ is the normalized $p$-adic regulator \cite[Lemma 3.1]{Gra7}
and $\CH^{\,\cyc}_K$ a subgroup of $\CH_K$ (equal to $\CH_K$, except ``the part''
corresponding to the maximal unramified extension contained in the cyclotomic 
$\Z_p$-extension of $K$, which simply depends on ramification of $p$ in $K$).

\smallskip
The order of the group ${\mathcal R}_K$ is (up to an obvious factor) the classical 
$p$-adic regulator which intervenes in the $p$-adic analytic formulas due to
the pioneering works of Kubota--Leopoldt on $p$-adic 
$\BL$-functions, then that of Amice--Fresnel--Barsky (e.g., \cite{Fre}), Coates,
Ribet and many other; see a survey in \cite{Gra3} and a lecture in \cite{Rib0}
where is used the beginnings of the concept of $p$-adic pseudo-measures 
of Mazur, developed by Serre \cite{CRAS}).
See in \cite{Gra66,Gra9} more complete studies and conjectures about 
$\CR_K$ and $\CT_K$.

\smallskip
At this time was stated the Iwasawa formalism of the Main Conjecture by
Greenberg \cite{Gree1, Gree2} after Iwasawa \cite{Iwa}.

\subsection{Main unsolved problem today}\label{remconcl}
Let $K/\Q$ be a real cyclic extension with a non trivial maximal $p$-sub-extension 
(non semi-simple case). Let $\BE_K$ (resp. $\BF_K$) be the group of units 
(resp. of Leopoldt's cyclotomic units) then $\CE_K = \BE_K \otimes \Z_p$ 
and $\CF_K = \BF_K \otimes \Z_p$; let $\wh \CE_K$ be the subgroup 
of $\CE_K$ generated by the $\CE_k$'s for all $k \varsubsetneqq K$.

\smallskip
It would remain to prove our conjecture \cite[\S\,III]{Gra0} for the $p$-adic characters 
$\varphi$ of $K$ saying that (see also Remarks \ref{e0} and \ref{remgreither}): 
\[\order \CH^\ar_\varphi = w_\varphi \cdot \order (\CE_K/\wh \CE_K \cdot \CF_K)_\varphi,
\ \ \hbox{$w_\varphi \in \{1, p\}$}, \]
where:
\[\CH^\ar_\varphi := \big\{ x \in \CH_{K},\ \, x^{P_\varphi (\sigma)} = 1 
\ \ \&\ \  \Norm_{K/k}(x) = 1,\ \forall \,k \varsubsetneqq K \big\} \]
and:
\[(\CE_K/\wh \CE_K \cdot \CF_K)_\varphi := \{\wt \varepsilon \in \CE_K/\wh \CE_K \cdot \CF_K,
\ \,   \wt \varepsilon^{P_\varphi (\sigma)} = 1 \}, \] 
where $P_\varphi$ is the local cyclotomic polynomial attached to $\varphi$
and $\sigma$ a generator of $\Gal(K/\Q)$.
For the $\varphi$-component $(\CE_K / \wh \CE_K \cdot \CF_K)_\varphi$, the two notions 
(arithmetic and algebraic) coincide, but the $\varphi$-class group must be 
defined in the arithmetic sense. One proves, Theorem \ref{theoI2}, that 
$(\CE_K/\wh \CE_K \cdot \CF_K)_\varphi = (\CE_K/\wh \CE_K \cdot \CF_K)_{\varphi_0^{}}$, 
$\varphi = \varphi_0^{} \varphi_p$; indeed, $(\CE_K/\wh \CE_K \cdot \CF_K)$ is a $\chi$-object 
for $\chi$ above $\varphi$ since it is annihilated by all the relative norms.

\section{Abelian extensions}
The idea of definition of the $\varphi$-objects owes a lot to the work of Leopoldt
\cite{15Leo, 16Leo} and their writing, in french, by Oriat in \cite{17Or, 18Or}.
Some outdated notations in these papers and ours are modified,
after changing $\ell$ into $p$ (e.g., $\Omega_p \mapsto \ov \Q_p$,
$\wh \Omega_p \mapsto \C_p$, $\Gamma \mapsto \Z_p$).

\subsection{Characters}
Let $\Q^\ab$ be the maximal abelian extension of $\Q$ contained in an 
algebraic closure $\ov \Q$ of $\Q$; let $\Q_p$ be the $p$-adic field and 
$\ov \Q_p$ an algebraic closure of $\Q_p$ containing $\ov \Q$. We put 
$\G := \Gal(\Q^\ab/\Q)$):

\begin{notations}\label{notations}
Let $\Psi$ be the set of irreducible characters of $\G$, of degree~$1$ and finite order, 
with values in $\ov \Q_p$. We define the sets of irreducible $p$-adic characters 
$\Phi$, for a prime $p \geq 2$, the set $\CX$ of irreducible rational characters
and the sets of irreducible characters $\Psi_K$, $\Phi_K$, $\CX_K$, 
of $K \subset \Q^\ab$.

\smallskip
The notation $\psi \mid \varphi \mid \chi$ (for $\psi \in \Psi$, $\varphi \in \Phi$,
$\chi \in \CX$) means that $\varphi$ is a term of $\chi$ and $\psi$ a term of $\varphi$.

\smallskip
Let $s_\infty \in \G$ be the complex conjugation and $\psi \in \Psi_K$;
if $\psi(s_\infty) = 1$ (resp. $\psi(s_\infty) = -1$), we say that $\psi$ is even (resp. odd)
and we denote by $\Psi_K^+$ (resp. $\Psi_K^-$) the corresponding subsets 
of characters. Since $\Psi_K^{\pm}$ is stable by any conjugation, this
defines $\Phi_K^{\pm}$, $\CX_K^{\pm}$.

\smallskip
Let $\chi \in \CX$; we denote by
$g^{}_\chi$, $K_\chi$, $G_\chi =: \langle \sigma_\chi \rangle$, $f_\chi$, 
$\Q(\mu_{g^{}_\chi})$, the order of any $\psi \mid \chi$, the subfield of $K$ 
fixed by $\Ker(\chi) := \Ker(\psi)$, $\Gal(K_\chi/\Q)$, the conductor of 
$K_\chi$, the field of values of the characters, respectively.
\end{notations}

\smallskip
The set $\CX$ has the following easy property considered as the ``Main theorem'' for
rational components (e.g., \cite[Chap. I, \S\,1, 1]{15Leo}):

\begin{theorem} \label{chiformula}
Let $K/\Q$ be a finite abelian extension and let 
$(A_\chi)_{\chi \in \CX_K}$, $(A'_\chi)_{\chi \in \CX_K}$
be two families of positive numbers, indexed by the set $\CX_K$ 
of irreducible rational characters of $K$. If for all 
subfields $k$ of $K$, one has
$\prod_{\chi \in \CX_k} A'_\chi = \prod_{\chi \in \CX_k} A_\chi$,
then $A'_\chi = A_\chi$ for all $\chi \in \CX_K$.
\end{theorem}

The interest of this property is that analytic formulas (giving for instance orders $A_K$ 
of some finite $p$-adic invariants $\CA_K$ of abelian fields $K$) may be \textit{canonically} 
decomposed under identities $A_K = \prod_{\chi \in \CX_K} A_\chi$, to be compared 
with algebraic relations $\order \CA_K = \prod_{\chi \in \CX_K} \order \CA_\chi$ for 
suitable $\Z_p[\G]$-modules $\CA_\chi$, so that $\order \CA_\chi = A_\chi$ 
for all $\chi$; the corresponding Finite AMC being the same statement, replacing 
rational characters $\chi$ by $p$-adic ones $\varphi$, under the existence of natural 
relations $\order \CA_\chi = \prod_{\varphi \mid \chi} \order \CA_\varphi$ and 
$A_\chi = \prod_{\varphi \mid \chi} A_\varphi$ for suitable $\Z_p[\G]$-modules
$\CA_\varphi$; the main problem being precisely what definition for the isotopic
components $\CA_\chi$ and $\CA_\varphi$.

\subsection{Main results of the article} \label{maindef} 
Let $\BM = (\BM_K)_{K\in \CK}$ be a family of $\Z[\G]$-modules, indexed with 
the set $\CK$ of finite abelian extensions and provided with the arithmetic norms 
$\Norm_{K/k}$ and transfer maps $\J_{K/k}$, for any $k \subseteq K$, where 
$\J_{K/k} \circ \Norm_{K/k} = \Nu_{K/k} \in \Z[\Gal(K/k)]$ (algebraic norm).
We associate with $\BM$ the family of $\Z_p[\G]$-modules $\CM := \BM \otimes \Z_p$.

\smallskip
We will give more definitions and details in Section \ref{subI3},
but we take note of the fact that, in the class field theory framework about
$p$-class groups and generalizations, the following remarks are of great
specific significance:

\begin{remarks}\label{nonramified}
(i) Let $H_k^\nr$ and $H_K^\nr$ be the $p$-Hilbert class fields of $k$ and
$K$, respectively; then the map $\Gal(H_K^\nr/K) \to \Gal(H_k^\nr/k)$,
given by the restriction of the Artin automorphisms, corresponds, by class
field theory, to the map $\Norm_{K/k} : \CH_K \to \CH_k$ (from norms of ideals) 
which is surjective as soon as the $p$-sub-extension of $K/k$ is totally ramified, 
which is almost always the case in the present abelian theory; more
precisely, \textit {this is always the case} when $K = K_\chi$, since then $K$ is the 
compositum of $K_0$, of prime-to-$p$ degree, with $K'$ cyclic of $p$-power
degree over $\Q$, thus totally ramified.

\smallskip
(ii) On the contrary, the transfer map $\J_{K/k}$, corresponding to extension
of classes (from that of ideals), is not necessarily injective in $p$-extensions; 
if this fact is well known precisely in $H_k^\nr/k$ (but $H_k^\nr$ is not abelian 
over $\Q$), it is very frequent in totally ramified abelian $p$-extensions as
$K/K_0$, described above; a fact less 
known which has interesting consequences (see, e.g., \cite{Gra13,Gra14,Gra15} 
for an extensive study of capitulation phenomena, where numerical experiments
show that capitulation is a common occurrence contrary to what one might think).
\end{remarks}

We define various $\chi$-components $\BM^\alg_\chi$, $\BM^\ar_\chi$,
$\CM^\alg_\chi$, $\CM^\ar_\chi$ (for $\chi \in \CX$) and the associated
$\varphi$-components $\CM^\alg_\varphi$, $\CM^\ar_\varphi$ 
(for $\varphi \in \Phi$), as follows:

\medskip
Let $P_\chi$ be the global $g_\chi$th cyclotomic polynomial, let $P_\varphi$ 
be the local cyclotomic polynomial associated with $\varphi \mid \chi$ 
(so that $P_\chi = \prod_{\varphi \mid \chi} P_\varphi$ in $\Z_p[X]$). We define:
\begin{equation*}
\left\{\begin{aligned}
\BM^\alg_\chi &  := \big\{ x \in \BM_{K_\chi}, \, x^{P_\chi(\sigma_\chi)} = 1\}, \ \
\CM^\alg_\chi := \BM^\alg_\chi \otimes \Z_p, \\
\CM^\alg_\varphi & := \big\{ x \in \CM^\alg_\chi,\,  x^{P_\varphi (\sigma_\chi)} = 1 \big\}, \\
\BM^\ar_\chi & := \{x \in \BM_{K_\chi}, \Norm_{K_\chi/k}(x) = 1,\ \forall \,
k \varsubsetneqq K_\chi \}, \ \ \CM^\ar_\chi  := \BM^\ar_\chi \otimes \Z_p, \\
\CM^\ar_\varphi & := \{x \in \CM^\alg_\varphi,  \, \Norm_{K_\chi/k}(x) = 1,\ \forall \,
k \varsubsetneqq K_\chi \}.
\end{aligned}\right.
\end{equation*}

\noindent
Then $\CM^\ar_\varphi = \big\{ x \in \CM_{K_\chi},\ 
x^{P_\varphi (\sigma_\chi)} = 1 \ \, \&\ \,  \Norm_{K_\chi/k}(x) = 1, \,
\forall \,k \varsubsetneqq K_\chi\big\}$, also equal to the $\varphi_0^{}$-component
of $\CM^\ar_\chi$.

\medskip
Being annihilated by $P_\chi(\sigma_\chi)$ (resp. $P_\varphi (\sigma_\chi)$)
$\BM^\alg_\chi$ and $\CM^\alg_\chi$ (resp. $\BM^\alg_\varphi$ and 
$\CM^\alg_\varphi$) are $\Z[\mu_{g_\chi^{}}]$-modules (resp. 
$\Z_p[\mu_{g_\chi^{}}]$-modules), for the law defined via 
$\sigma \in \G \mapsto \psi(\sigma) \in \mu_{g_\chi^{}}$, 
for $\psi \mid \chi$ (resp. $\psi \mid \varphi$).

\medskip
(i) Then we have the following results: 

\medskip
$\bullet$ $\BM^\alg_\chi = \big\{ x \in \BM_{K_\chi}, \  \Nu_{K_\chi/k}(x) = 1,
\  \forall \, k \varsubsetneqq K_\chi \}$ (Theorem \ref{theoI1}),

\smallskip
$\bullet$ $\CM^\alg_\chi = \plus_{\varphi \mid \chi} \CM^\alg_\varphi$,\ \,
$\CM^\ar_\chi \, = \plus_{\varphi \mid \chi} \CM^\ar_\varphi$ 
(Theorems \ref{theoI2}, \ref{theoI2bis}).

\smallskip
(ii) Assume that $K/\Q$ is cyclic and $\BM_K$ finite:

\medskip
\quad (ii\,$'$) If, for all sub-extensions $k/k'$ of $K/\Q$, the norm maps 
$\Norm_{k/k'}$ are surjective, then:

\smallskip
$\bullet$ $\order \BM_K = \prd_{\chi \in \CX_K} \order \BM^\ar_\chi$ (Theorem \ref{theoI5}),

\quad (ii\,$''$) Let $K/K_0$ be the maximal $p$-sub-extension of $K$;
if, for all sub-extensions $k/k'$ of $K/K_0$, the norm maps 
$\Norm_{k/k'}$ are surjective, then:

\smallskip
$\bullet$ $\order \CM_\chi^\ar =  \prd_{\varphi \mid \chi} \order \CM^\ar_\varphi$ 
(Theorem \ref{theoI2bis}).

The above conditions of surjectivity of the norms are automatically 
fulfilled for the families $\BH$ (class groups), $\CH = \BH \otimes \Z_p$
($p$-class groups) and $\CT$ (torsion groups of abelian $p$-ramification).
 
\medskip
(iii) Applying this to $\BH$ and $\CT$, we obtain:

\medskip
\quad (iii\,$'$) For all characters $\chi \in \CX^-$, we have:

\medskip
$\bullet$ $\ \ \BH^\ar_\chi = \BH^\alg_\chi$ and 
$\CH^\ar_\varphi = \CH^\alg_\varphi$, $\  \forall \,\varphi \mid \chi$ 
(Theorem \ref{theoII1});

\medskip
$\bullet$ $\order \BH^\ar_\chi = \order \BH^\alg_\chi = 2^{\alpha_\chi} \cdot w_\chi \cdot 
\prd_{\psi \mid \chi}\big(- \hbox{$\frac{1}{2}$}\, \BB_1(\psi^{-1}) \big)$ 
(Theorem \ref{theoII2}), in terms of generalized Bernoulli numbers.

\smallskip
\quad (iii\,$''$) For all characters $\chi \in \CX^+$, we have: 

\medskip
$\bullet$ $\ \BH^\ar_\chi \subseteq \BH^\alg_\chi$ and 
$\CH^\ar_\varphi \subseteq \CH^\alg_\varphi$, $\  \forall \,\varphi \mid \chi$
(see Examples \ref{Ex1}, \ref{Ex2} for strict inclusions);

\medskip
$\bullet$ $\order \BH^\ar_\chi = w_\chi \cdot 
\big (\BE_{K_\chi} : \wh \BE_{K_\chi} \!\!\cdot \BF_{K_\chi} \big)$
(Theorem \ref{theoIII2}), in terms of cyclotomic units, where
$\wh \BE_{K_\chi} := \langle \, \BE_k \, \rangle_{k \varsubsetneqq K_\chi}^{}$.

\quad (iii\,$'''$) For all even characters $\chi$, we have: 

\medskip
$\bullet$ $\ \CT^\ar_\chi = \CT^\alg_\chi$ and 
$\CT^\ar_\varphi = \CT^\alg_\varphi$, $\ \forall \,\varphi \mid \chi$ 
(Theorem \ref{theotorsion});

\smallskip
$\bullet$ $\order \CT^\ar_\chi =  w_\chi^{\,\cyc} \cdot 
\prd_{\psi \mid \chi} \hbox{$\frac{1}{2}$} \,\BL_p (1,\psi)$ 
(Theorem \ref{theotorsion2}), in terms of $p$-adic $\BL$-functions.

\medskip
(iv) The Arithmetic Invariants of finite $\Z_p[\G]$ modules $\CM_K$ are defined 
by means of the obvious algebraic writing of $\Z_p[\mu_{g^{}_\chi}]$-modules:
\begin{equation*}
\CM^\ar_\varphi \simeq \prd_{i \geq 1} \Big[\Z_p[\mu_{g^{}_\chi}] \big / 
{\mathfrak p}_\varphi^{\,n^\ar_{\varphi, i}{\sst(\CM)}} \Big], \ \ \ \ 
m^\ar_\varphi{\sst(\CM)} := \sm_i n^\ar_{\varphi, i}{\sst(\CM)},
\end{equation*}
where ${\mathfrak p}_\varphi$ is the maximal ideal of $\Z_p[\mu_{g^{}_\chi}]$;
the definition of the Analytic Invariants $m^\an_\varphi{\sst(\CM)}$ 
comes directly from the formulas of $\order \CM^\ar_\chi$ given above in (iii), 
taking into account the decompositions $\CM^\ar_\chi = 
\oplus_{\varphi \mid \chi} \CM^\ar_\varphi$, whence the statement of the 
Finite AMC  ``$m^\ar_\varphi{\sst(\CM)} = m^\an_\varphi{\sst(\CM)}$, 
for all $\varphi \in \Phi$'' (Section \ref{MainConj}, Conjecture \ref{mainconj}).

\section{Definition and study of the \texorpdfstring{$\varphi$}{Lg}-objects}

We shall give, in this section, the general definition of $\theta$-objects, $\theta$ 
being an irreducible character (rational or $p$-adic), the Galois modules 
which intervene in the definition of the $\theta$-objects being not necessarily 
finite, as it is the case for unit groups; finally, the prime $p$ is arbitrary and we 
shall emphasize on the non semi-simple framework.

\subsection{The Algebraic and Arithmetic \texorpdfstring{$\G$}{Lg}-families}\label{subI3}
Let $\CK$ be the family of finite extensions $K$ of $\Q$, contained
in $\Q^\ab$, of Galois group $G_K$.
We assume to have a family $\BM$ of (multiplicative) $\Z[\G]$-modules, 
indexed by $\CK$ (called a $\G$-family), $\BM=(\BM_K)_{K \in \CK}$. 

\smallskip
In general there exist two families of $\G$-homomorphisms, indexed by the set of 
sub-extensions $K/k$, $\Norm_{K/k} : \BM_K \to \BM_k$ (arithmetic norms), 
$\J_{K/k} : \BM_k \to \BM_K$ (arithmetic transfers). For all sub-extensions 
$K/k$, we put $\Nu_{K/k}\! :=\! \sm_{\sigma \in \Gal(K/k)} \! \sigma \in \Z[\Gal(K/k)]$
(algebraic norm).

\smallskip
We consider the three following conditions:

\smallskip
\quad (a) For all $K \in  \CK$, $\BM_K^{\Gal(\Q^\ab/K)} = \BM_K$
(so, for $x \in \BM_K$ and $\sigma \in \G$, $x^\sigma = x^{\sigma_K^{}}$, where
$\sigma_K^{} \in G_K$ is the restriction of $\sigma$ to $K$).

\smallskip
\quad (b) For all sub-extension $K/k$, the arithmetic maps $\Norm_{K/k}$
and $\J_{K/k}$ are $\G$-module homomorphisms fulfilling the transitivity formulas:
\[\hbox{$\Norm_{K/k} \circ \Norm_{L/K} = \Norm_{L/k}$ and 
$\J_{L/K} \circ \J_{K/k} = \J_{L/k}$,} \]
for all $k, K, L \in \CK$, $k \subseteq K \subseteq L$.

\smallskip
\quad (c) For all sub-extension $K/k$, $\J_{K/k} \circ \Norm_{K/k} = \Nu_{K/k}$
on $\BM_K$.

\begin{definitions}
(i) If $\BM = (\BM_K)_{K \in \CK}$ only fulfills condition (a), 
we shall say that the family $(\BM, \Nu\ )$ is an algebraic $\G$-family;
one may only use Galois theory in $K/k$ and the algebraic norms 
$\Nu_{K/k} \in \Z[\Gal(K/k)]$. 

\smallskip
(ii) If moreover, there exist two families $(\Norm_{K/k})$ and $(\J_{K/k})$
(canonically associated with $\BM$) fulfilling conditions (b) and (c), we shall say 
that  the family $(\BM, \Norm, \J)$ is an arithmetic $\G$-family.
\end{definitions}

The following properties of $\BM_K$ and $\CM_K := \BM_K \otimes \Z_p$ are elementary:

\begin{proposition}\label{propI2} 
(i) For all $K \in \CK$, $\Nu_{K/K}$,
$\Norm_{K/K}$, $\J_{K/K}$ are the identity, $\id$, on $\BM_K$.

\smallskip
(ii) If the map $\Norm_{K/k}$ is surjective or if the map $\J_{K/k}$ is injective,
then $\Norm_{K/k} \circ \J_{K/k} = [K : k]$.
\end{proposition}

\begin{remark}
Note that cohomology is only of algebraic nature since, for instance
in the case of a cyclic extension $K/k$ of Galois group $G =: \langle \sigma \rangle$, 
using the class group $\BH_K$, we have:
\[\Hom^1(G,\BH_K) \simeq \Ker(\Nu_{K/k}) \big / \BH_K^{1-\sigma}, \ \ \ \ 
\Hom^2(G,\BH_K) \simeq \BH_K^G \big / \Nu_{K/k}(\BH_K); \]
in general $\Nu_{K/k}(\BH_K)$ is not isomorphic to $\Norm_{K/k}(\BH_K) 
\subseteq \BH_k$, even if the arithmetic norm is surjective. 
\end{remark}

\begin{examples}\label{ssI3c}
The most straightforward examples of such arithmetic $\G$-families $\BM_K$
are the following ones:

\smallskip
 (i) the group $\BE_K$ of units of $K$ (for which maps $\J_{K/k}$
are injective);

\smallskip
 (ii)  the class group $\BH_K$ of $K$, or the $p$-class group $\CH_K$.

\smallskip
 (iii)  the torsion group ${\mathcal T}_K$
of the Galois group of the maximal $p$-ramified abelian pro-$p$-extension of $K$.

\smallskip
 (iv) the group-algebra $\Aa[G_K]$, where $\Aa$ is a commutative ring;
then $\Aa[G_K]$ is a $\Aa[\G]$-module if one puts $\sigma \cdot \Omega =
\sigma_K^{} \Omega$ (product in $\Aa[G_K]$), for all $\Omega \in \Aa[G_K]$
and $\sigma \in \G$. 
The maps $\Norm_{K/k}$ and $\J_{K/k}$ are defined by $\Aa$-linearity by
$\Norm_{K/k}(\sigma_K^{}) := \sigma_k^{}$ and, for $\sigma_k^{} \in G_k$, by
$\J_{K/k}(\sigma_k^{}) := \sum_{\tau \in \Gal(K/k)} \tau \cdot \sigma'_k 
= \Nu_{K/k}\cdot \sigma'_k = \Nu_{K/k} \sigma'_k$, 
where $\sigma'_k$ is any extension of $\sigma_k^{}$ in $G_K$. 
So, for $\sigma_K^{} \in G_K$, 
$\Nu_{K/k}(\sigma_K^{}) = 
\big(\sum_{\tau \in \Gal(K/k)} \tau \big) \cdot  \sigma_K^{} = \Nu_{K/k} \sigma_K^{}$.
\end{examples}

\subsection{Definition of the \texorpdfstring{$\G$}{Lg}-modules 
\texorpdfstring{$\BM^\alg_\chi$}{Lg}, 
\texorpdfstring{$\BM^\ar_\chi$, $\CM^\alg_\varphi$, $\CM^\ar_\varphi$}{Lg}}\label{subI4}
We shall assume in the sequel that $\Aa \in \{\Z,\ \Z_p\}$.

\subsubsection{The \texorpdfstring{$\Gamma_{\BKappa}$}{Lg}-conjugation}\label{ssI4a}
Let $\chi \in \CX$. Let $P_\chi(X) \in \Z[X]$ be the $g^{}_\chi$th global cyclotomic 
polynomial. Let $\BKappa^{}\!_\Aa$ be the field of quotients of $\Aa$ and let 
$\BKappa^{}_\Aa(\mu_{g^{}_\chi})/\BKappa^{}_\Aa$ be 
the extension by the $g^{}_\chi$th roots of unity; so,
$\Gamma_{\BKappa^{}_\Aa,\chi} := \Gal(\BKappa^{}_\Aa(\mu_{g^{}_\chi})/\BKappa^{}_\Aa)$ is 
isomorphic to a subgroup of $(\Z/g^{}_\chi \Z)^\times$.

\smallskip
One defines, following \cite{19Ser}, the $\Gamma_{\BKappa^{}_\Aa}$-conjugation on 
$\Psi$ by putting, for all $\tau \in \Gamma_{\BKappa^{}_\Aa,\chi}$
and $\psi \in \Psi$, $\psi \mid \chi$, $\psi^\tau := \psi^a$, where $a \in \Z$
is a representative of $\tau$ in $(\Z/g^{}_\chi \Z)^\times$. Then the $\psi^\tau(\sigma_\chi)$ are the
conjugates of $\psi(\sigma_\chi)$ in ${\BKappa^{}_\Aa}(\mu_{g^{}_\chi})/{\BKappa^{}_\Aa}$.
This defines the irreducible characters over $\BKappa_\Aa^{}$ (with values in $\Aa$),
$\theta = \sum_{\tau \in \Gamma_{{\Kappa^{}_\Aa},\,\chi}} \psi^\tau$.

\subsubsection{Correspondence between characters and 
cyclotomic polynomials}\label{ssI4b}
Let $\chi \in \CX$. In ${\BKappa^{}_\Aa}[X]$, $P_\chi$ 
splits into a product of irreducible distinct polynomials~$P_{\chi, i}$; each $P_{\chi, i}$ 
splits into degree $1$ polynomials over ${\BKappa^{}_\Aa}(\mu_{g^{}_\chi})$ and
is of degree $[{\BKappa^{}_\Aa}(\mu_{g^{}_\chi}) : {\BKappa^{}_\Aa}]$. 

\smallskip
If $\zeta_i \in \mu_{g^{}_\chi}$ is a root of $P_{\chi, i}$, the other roots are the 
$\zeta_i^\tau$ for $\tau \in \Gamma_{{\BKappa^{}_\Aa},\,\chi}$; thus, these sets 
of roots are in one by one correspondence with the sets of the form 
$(\psi^\tau(\sigma_\chi))_{\tau \in \Gamma_{\BKappa^{}_\Aa,\,\chi}}$,
$\psi^\tau \mid \chi$, $\psi^\tau \in \Psi$ of order $g^{}_\chi$, describing a
representative set of characters for the $\Gamma_{\BKappa^{}_\Aa}$-conjugation. 
One may index, \textit {non-canonically}, the irreducible divisors of $P_\chi$ in 
${\BKappa^{}_\Aa}[X]$ by means of the characters $\theta$ obtained from the 
characters $\psi \in \Psi$ of orders $g^{}_\chi$ and by choosing a generator 
$\sigma_\chi$ of $G_\chi$. Put:
\begin{equation}\label{deftheta}
P_\theta := \prd_{\psi \mid \theta} (X - \psi(\sigma_\chi)) \in \Aa[X]. 
\end{equation}
Thus $P_\chi = \prd_{\theta \mid \chi} P_\theta$;  
for $\Aa=\Z_p$ we get $P_\chi = \prd_{\varphi \in \Phi,\,\varphi \mid \chi} P_\varphi$, 
for $\Aa=\Z$, $P_\chi$ is irreducible. So,
$\Aa[G_\chi]/(P_\theta(\sigma_\chi)) \simeq 
\Aa[X]/(X^{g^{}_\chi}-1,P_\theta(X)) \simeq \Aa[\mu_{g^{}_\chi}]$; 
then any module annihilated by $P_\theta(\sigma_\chi)$ is a
$\Aa[\mu_{g^{}_\chi}]$-module; the law is realized, for $\psi \mid \theta$, 
via $\sigma \in G_\chi \mapsto \psi(\sigma) \in \mu_{g^{}_\chi}$.

\subsubsection{The \texorpdfstring{$\Z[\mu_{g^{}_\chi}]$}{Lg}-modules 
\texorpdfstring{$\BM^\alg_\chi$}{Lg} and the 
\texorpdfstring{$\Z_p[\mu_{g^{}_\chi}]$}{Lg}-modules 
\texorpdfstring{$\CM^\alg_\varphi$}{Lg}}\label{ssI4c}
We fix a prime $p$ and consider the set  $\Phi$ of
irreducible $p$-adic characters of $\G$.
\begin{definition}\label{defI2}
Let $\BM = (\BM_K)_{K\in \CK}$ be a family of $\Z[\G]$-modules
and let $\CM := \BM  \otimes \Z_p = (\CM_K)_{K\in \CK}$. 
Put, for $\chi \in \CX$ and $\varphi \mid \chi$, $\varphi \in \Phi$:
\begin{equation*}
\left\{\begin{aligned}
\BM^\alg_\chi &  := \big\{ x \in \BM_{K_\chi},\,  x^{P_\chi(\sigma_\chi)} = 1 \big\}, \\
\CM^\alg_\chi & := \BM^\alg_\chi  \otimes \Z_p =
\big\{ x \in \CM_{K_\chi},\  x^{P_\chi (\sigma_\chi)} = 1 \big\}, \\
\CM^\alg_\varphi & := \big\{ x \in \CM_{K_\chi},\  x^{P_\varphi (\sigma_\chi)} = 1 \big\}
= \big\{ x \in \CM^\alg_\chi,\  x^{P_\varphi (\sigma_\chi)} = 1 \big\}.
\end{aligned}\right.
\end{equation*}

\noindent
So, $\CM^\alg_\varphi$ is a sub-$\Z_p[\mu_{g^{}_\chi}]$-module of 
$\CM_{K_\chi}$ (or of $\CM^\alg_\chi$), for the law $\sigma \in 
G_\chi \mapsto \psi(\sigma)$, $\psi \mid \varphi$, and the elements 
of $\CM^\alg_\varphi$ are called algebraic $\varphi$-objects.
\end{definition}

From relation \eqref{deftheta}, the polynomials $P_\varphi$ depend on 
the choice of the generator $\sigma_\chi$ of $G_\chi$, but we have the 
following property:

\begin{lemma} \label{propI3}
The Definitions \ref{defI2}, of the $\Z[\mu_{g^{}_\chi}]$-modules $\BM^\alg_\chi$
and the $\Z_p[\mu_{g^{}_\chi}]$-modules $\CM^\alg_\varphi$, do not depend on 
the choice of $\sigma_\chi$.
\end{lemma}

\begin{proof} 
Let $\varphi \mid \chi$.
We have $P_\varphi (\sigma_\chi) = \prd_{\psi \mid \varphi} (\sigma_\chi - \psi(\sigma_\chi))$
and, for $a>0$ with $\gcd (a, g^{}_\chi) = 1$, let $\sigma'_\chi = :\sigma_\chi^a$ another 
generator of $G_\chi$ giving the relation $P'_\varphi (\sigma'_\chi) = 
\prd_{\psi \mid \varphi} (\sigma'_\chi - \psi(\sigma'_\chi))$; 
one must compare $P_\varphi (\sigma_\chi)$ and $P'_\varphi (\sigma'_\chi)$. Then: 
\[P'_\varphi (\sigma_\chi^a) = \prd_{\psi \mid \varphi} 
(\sigma_\chi^a - \psi(\sigma_\chi^a)) = \!
\prd_{\psi \mid \varphi} \big[ (\sigma_\chi - \psi(\sigma_\chi))
\times (\sigma_\chi^{a-1} + \cdots + \psi^{a-1}(\sigma_\chi)) \big],\]
and similarly, writing $1 \equiv a \, a^* \pmod {g^{}_\chi}$, where $a^* > 0$ 
represents an inverse of $a$ modulo $g^{}_\chi$, we have, from 
$\sigma_\chi = (\sigma_\chi^a)^{ a^*}$:
\[P_\varphi (\sigma_\chi) = \prd_{\psi \mid \varphi}\big[ (\sigma_\chi^a - \psi(\sigma_\chi^a))
\times (\sigma_\chi^{a(a^*-1)}+ \cdots + \psi^{a(a^*-1)}(\sigma_\chi))\big]. \]

Since $P'_\varphi (\sigma'_\chi) \in P_\varphi (\sigma_\chi) \Z_p[G_\chi]$ and
$P_\varphi (\sigma_\chi) \in P'_\varphi (\sigma'_\chi) \Z_p[G_\chi]$ the invariance 
of the definition of the $\varphi$-objects follows, as well as that of $\chi$-objects
since $P_\chi = \prd_{\varphi \mid \chi} P_\varphi$. 
\end{proof}

\subsubsection{Characterization of \texorpdfstring{$\BM^\alg_\chi$, $\CM^\alg_\chi$}{Lg}, 
with algebraic norms}\label{ssI4d}
For any $\chi \in \CX$, we have defined $\BM^\alg_\chi$ and $\CM^\alg_\chi$.
We then have the following characterization, only valid for rational characters, but
which will allow another definition of $\chi$ and $\varphi$-objects (that 
of ``Arithmetic'' objects):

\begin{theorem}\label{theoI1}
Let $\BM$ be a $\G$-family of $\Z[\G]$-modules and for $\chi \in \CX$, let 
$\BM^\alg_\chi := \big\{ x \in \BM_{K_\chi},\  x^{P_\chi (\sigma_\chi)} = 1 \big\}$. Then:
\begin{equation*}
\left\{\begin{aligned}
\BM^\alg_\chi & = \big\{ x \in \BM_{K_\chi}, \ \, \Nu_{K_\chi/k}(x) = 1,\ 
\hbox{for all $k \varsubsetneqq K_\chi$}\}, \\
\CM^\alg_\chi & = \big\{ x \in \CM_{K_\chi}, \ \,  \Nu_{K_\chi/k}(x) = 1,\
\hbox{for all $k \varsubsetneqq K_\chi$}\}
\end{aligned}\right.
\end{equation*}
(one may limit the norm conditions to $\Nu_{K_\chi/k_\ell^{}}(x) = 1$
for all prime divisors $\ell$ of $[K_\chi : \Q]$, where $k_\ell^{} \subset K_\chi$ is 
such that $[K_\chi : k_\ell^{}] = \ell$). 
\end{theorem}

\begin{proof}\textit {With a contribution of a personal communication 
from Jacques Martinet (October 1968).}
We need three preliminary lemmas:

\begin{lemma}\label{lemI1}
Let $n \geq 1$ and let $q$ be an arbitrary prime number. Denote by $P_n$ the $n$th 
cyclotomic polynomial in $\Z[X]$; then:

\smallskip
\quad (i) $P_n(X^q)=P_{nq}(X)$, if $q \mid n$;

\smallskip
\quad (ii) $P_n(X^q)=P_{nq}(X)\,P_n(X)$, if $q \nmid n$.
\end{lemma}

\begin{proof}
Obvious for (i), (ii) by means of comparison of the sets of roots of these polynomials.
\end{proof}

\begin{lemma}\label{lemI2}
Let $n = \ell_1 \cdots \ell_t$, $t \geq 2$,
the $\ell_i$'s being distinct prime numbers.
Then for all pair $(i, j)$, $i \ne j$, there exist $A_i^j$ and $A_j^i$ in $\Z[X]$,
such that $A_i^j P_{\frac{n}{\ell^{}_i}} + A_j^i P_{\frac{n}{\ell^{}_j}}=1$.
\end{lemma}

\begin{proof}
This can be proved by induction on $t \geq 2$.

\smallskip
If $t=2$, $n=\ell_1 \ell_2$ and:
\[\hbox{$P_{\frac{n}{\ell^{}_2}} = P_{\ell^{}_1}  = X^{\ell^{}_1-1}+ \cdots +X+1$, \ 
$P_{\frac{n}{\ell^{}_1}} = P_{\ell^{}_2} = X^{\ell^{}_2-1}+ \cdots +X+1$.} \]

Let's call ``geometric polynomial'' any polynomial in $\Z[X]$ of the form 
$X^d+X^{d-1}+ \cdots + X+1$, $d \geq 0$ (including the polynomial $0$). 

\smallskip
Then if $P$ and $Q \ne 0$ are geometric, the residue $R$ 
of $P$ modulo $Q$ is geometric with residue $(P-R) Q^{-1} \in \Z[X]$;
indeed, if $m \geq n$ and $m+1 = q (n+1) + r$, $0 \leq r < n$, we get:
\begin{equation*}
\begin{aligned}
& X^m+ \cdot\cdot + X + 1 = \\
& (X^n + \cdot\cdot + X+1)\! \times\! \big[ X^{m+1-(n+1)}\! + X^{m+1-2(n+1)} 
\!+ \cdot\cdot + X^{m+1-q (n+1)} \big] \\
& \hspace{7.5cm}  + 1+X+ \cdot\cdot + X^{r-1}
\end{aligned}
\end{equation*}

\noindent
(if $r \geq 1$, otherwise the residue $R$ is $0$).
In particular, the gcd algorithm gives geometric polynomials; as the unique non-zero
constant geometric polynomial is $1$, it follows that if $P$ and $Q$ are co-prime
polynomials in $\Q[X]$, $\gcd(P,Q)=1$ and the B\'ezout relation takes place in $\Z[X]$,
which is the case for the geometric polynomials $P_{\ell^{}_1}$ and $P_{\ell^{}_2}$.

\smallskip
Suppose $t \geq 3$. Let $\ell_i$, $\ell_j$, $q$, be three distinct primes 
dividing $n$; put $n':= \ffrac{n}{q}$; by induction, 
since $\ell_i$ and $\ell_j$ divide $n'$, there exist
polynomials $A'^j_i, A'^i_j$ in $\Z[X]$, such that
$A'^j_i(X) P_{\frac{n'}{\ell^{}_i}}(X) + A'^i_j(X) P_{\frac{n'}{\ell^{}_j}}(X) = 1$, 
thus,
$A'^j_i(X^q) P_{\frac{n'}{\ell^{}_i}}(X^q) + A'^i_j(X^q) P_{\frac{n'}{\ell^{}_j}}(X^q) = 1$.
But Lemma \ref{lemI1}\,(ii) gives:
\[P_{\frac{n'}{\ell^{}_i}}(X^q)=P_{\frac{n}{\ell^{}_i}}(X) P_{\frac{n'}{\ell^{}_i}}(X) \ \ \ \ \&
\ \ \ \ P_{\frac{n'}{\ell^{}_j}}(X^q)=P_{\frac{n}{\ell^{}_j}}(X) P_{\frac{n'}{\ell^{}_j}}(X), \] 
which yields 
$A'^j_i(X^q) P_{\frac{n}{\ell^{}_i}}(X) P_{\frac{n'}{\ell^{}_i}}(X) 
+ A'^i_j(X^q) P_{\frac{n}{\ell^{}_j}}(X) P_{\frac{n'}{\ell^{}_j}}(X) = 1$. 

We have proved the co-maximality, in $\Z[X]$, of any pair of ideals
($P_{\frac{n}{\ell^{}_i}}(X)$), ($P_{\frac{n}{\ell^{}_j}}(X)$), $i \ne j$ (the case $n=\ell$
giving the prime ideal ($P_\ell(X)\Z[X]$)).
\end{proof}

\begin{lemma}\label{lemI3}
Let $n = \prd_{i=1}^t \ell_i^{a_i} >1$, $a_i \geq 1$;
put $N_{n,\ell}(X) := \sm_{i=0}^{\ell-1} X^{\frac{n}{\ell}\, i}$ for any prime $\ell$
dividing $n$. Then there exist polynomials $A_\ell(X) \in \Z[X]$ such that
$P_n(X) \!=\! \sm_{\ell \mid n} A_\ell(X) N_{n,\ell}(X)$ and 
$\big\langle  N_{n,\,\ell}(X), \ \ell \mid n \big\rangle^{}_{\Z[X]} = P_n(X) \Z[X]$.
\end{lemma}

\begin{proof}
Assume by induction on $n$ that
$P_n(X) = \sm_{\ell \mid n} A_\ell(X) N_{n,\ell}(X)$ (with $t$ fixed),
and let $q \mid n$; we have, from Lemma \ref{lemI1}\,(i):
\[P_{nq}(X) = P_n(X^q) = \sm_{\ell \mid n} A_\ell(X^q) N_{n,\ell}(X^q). \]
Since we have
$N_{n,\ell}(X^q) = \sm_{i=0}^{\ell-1} X^{\frac{n}{\ell}q\, i} = N_{nq,\ell}(X)$,  
we obtain that if the lemma is true for $n$, it is true for $nq$ for all $q \mid n$. 
It follows that if the property is true for all square-free integers $n$, it is true for all $n>1$. 
So we may assume $n$ square-free to prove the lemma by induction on $t$.

\smallskip
If $n=\ell_1$, $P_{\ell^{}_1}(X) = X^{\ell^{}_1-1}+ \cdots + X + 1 = N_{\ell^{}_1,\ell^{}_1}(X)$
and the claim is obvious. 
If $n = \ell_1 \ell_2 \cdots \ell_t$, $t \geq 2$, with distinct primes, put $n_k = \frac{n}{\ell^{}_k}$ 
for all $k$; by assumption,
$P_{n^{}_k}(X) = \sm_{1\leq s \leq t, \ s \ne k} A_s^k(X) N_{n^{}_k,\ell^{}_s}(X)$, hence:
\begin{equation*}
\begin{aligned}
P_{n^{}_k}(X^{\ell^{}_k}) & = P_{n^{}_k \ell_k}(X) \cdot P_{n^{}_k}(X) \\
& = P_n(X) P_{n^{}_k}(X) = \sm_{1\leq s \leq t \ s \ne k} A_s^k(X^{\ell^{}_k}) N_{n,\ell^{}_s}(X),
\end{aligned}
\end{equation*} 
whence $P_n(X) P_{n^{}_k}(X) \in \big\langle  N_{n,\,\ell}(X), 
\ \ell \mid n  \big\rangle^{}_{\Z[X]}, \ \hbox{for all $k$}$; 
since $t \geq 2$, Lemma \ref{lemI2} applies; a B\'ezout relation in $\Z[X]$ 
between any two of the $ P_{n^{}_k}$ (say $P_{n^{}_i}$ and $P_{n^{}_j}$) yields
$P_n(X) \times 1 \in \langle  N_{n,\,\ell}(X), \, \ell \mid n \rangle^{}_{\Z[X]}$, 
giving the result.

\smallskip
We have proved that the ideal generated, in $\Z[X]$, by the $N_{n,\ell}(X)$, 
$\ell \mid n$, contains $P_n(X) \Z[X]$. Let's see that $P_n(X)$ contains that ideal; 
it is sufficient to see that for all $\ell \mid n$, $N_{n,\ell}(X) = P_\ell(X^{\frac{n}{\ell}})$;
any root of unity $\zeta_n$ of order $n$ (i.e., root of $P_n(X)$), is a root of $N_{n,\ell}(X)$
since $\zeta_n^{\frac{n}{\ell}} = \zeta_\ell \ne 1$ and $\sm_{i=0}^{\ell - 1}\zeta_\ell^i = 0$;
then $P_n(X) \mid N_{n,\ell}(X)$ in $\Z[X]$ (monic polynomials).
\end{proof}

We apply this to $P_\chi(\sigma_\chi) = P_{g^{}_\chi}(\sigma_\chi)$ and to 
$N_{g^{}_\chi,\ell}(\sigma_\chi) = \Nu_{K_\chi/k_\ell^{}}$, where $k_\ell^{}$ is, 
for all $\ell\mid g^{}_\chi$, the unique sub-extension of $K_\chi$ such that 
$[K_\chi : k_\ell^{}] = \ell$. 
The theorem immediately follows.
\end{proof}

\subsubsection{Application to the definition of \texorpdfstring{$\BM^\ar_\chi$}{Lg}}
Let  $\BM$ be an arithmetic $\G$-family, provided 
with norms $\Norm$ and transfer maps $\J$ with $\J \circ \Norm = \Nu$.

\begin{definition}\label{defI4e}
By analogy with Theorem \ref{theoI1} giving, for $\chi$-objects, the characterization 
$\BM^\alg_\chi := \big\{ x \in \BM_{K_\chi},  \, \Nu_{K_\chi/k}(x) = 1,\, 
\hbox{for all $k \varsubsetneqq K_\chi$}\}$ and $\CM^\alg_\chi = \BM^\alg_\chi \otimes \Z_p$,
we define the modules of arithmetic $\chi$-objects:
\begin{equation*}
\left\{\begin{aligned}
\BM^\ar_\chi & := \{x \in \BM_{K_\chi},\ \,\Norm_{K_\chi/k}(x) = 1,\ \hbox{for all 
$k \varsubsetneqq K_\chi$} \}  \subseteq \BM^\alg_\chi \\  
\CM^\ar_\chi & := \BM^\ar_\chi \otimes \Z_p.
\end{aligned}\right.
\end{equation*} 

\noindent
Then $\BM^\ar_\chi$  is a sub-$\Z[\mu_{g^{}_\chi}]$-module  of $\BM^\alg_\chi$
and $\CM^\ar_\chi$ is a sub-$\Z_p[\mu_{g^{}_\chi}]$-module of $\CM^\alg_\chi$,
with laws defined via the choice of $\psi \mid \chi$ (resp. $\psi \mid \varphi$).
\end{definition}

We have $\BM^\ar_\chi = \BM^\alg_\chi$ as soon as the $\J_{K_\chi/k}$'s 
are injective (for all $k \varsubsetneqq K_\chi$ or simply the $k_\ell^{}$'s).
One verifies easily that if the norms $\Norm_{K_\chi/k_\ell^{}}$ are surjective 
for all $\ell \mid g_\chi$, then $\BM^\alg_\chi/\BM^\ar_\chi$ has exponent a 
divisor of $\prod_{\ell \mid g^{}_\chi}\!\! \ell$, whence $\CM^\alg_\chi/\CM^\ar_\chi$
of exponent $1$ or $p$.

\subsection{Comparison with classical definitions of 
\texorpdfstring{$\theta$}{Lg}-components}\label{tensor}

In all classical papers, the $\theta$-components $\BM_\theta$ 
($\theta$ rational or $p$-adic, above $\psi \in \Psi$) is defined,
in an abelian field $K$ of Galois group $G_K$, by:
\[\BM_\theta := \BM \otimes_{\Aa[G_K]}^{} \Aa[\theta], \]
where $\Aa[\theta] := \Aa[\psi]$ is the ring of values of $\theta$ over $\Aa$;
the action being defined via $(\sigma, x) \in G_K \times \BM_\theta \mapsto 
x^{\psi(\sigma)} \in \BM_\theta$. 
We shall compare this definition with Definition \ref{defI4e} considering irreducible $p$-adic 
characters $\varphi$. We have the classical algebraic definition of $\varphi$-objects 
attached to $\CM$, that is to say, the largest quotient such that $G_\chi$ acts by 
$\psi$ (\cite[Definition, p. 451]{Grei}, \cite[\S\,1.3]{PR}, \cite{Maz}):
\[\wh \CM_\varphi := \CM \otimes^{}_{\Z_p[G_\chi]} \Z_p[\mu_{g^{}_\chi}]
\simeq \CM/P_\varphi(\sigma_\chi) \cdot \CM \]

Another viewpoint \cite[\S\,II.1, pp. 469--471]{Sol}, is to define 
$\wh \CM{}^\varphi$ as the largest sub-$\Z_p[G_\chi]$-module of $\CM$, 
such that $G_\chi$ acts by $\psi$. Whence:
\[\wh \CM{}^\varphi := \{x \in \CM,\ \, x^{P_\varphi(\sigma_\chi)} = 1 \} 
=  \CM^\alg_\varphi, \]

\noindent
with the exact sequence $1\to\wh \CM{}^\varphi = \CM_\varphi^\alg \too \CM \too 
P_\varphi(\sigma_\chi) \cdot \CM \to 1$ giving the equalities $\order \wh \CM_\varphi 
= \order \wh \CM^\varphi = \order \CM^\alg_\varphi$ for finite modules.

\smallskip
Moreover, our forthcoming Definition \ref{defI3} of $\CM^\ar_\varphi$:
\[\CM^\ar_\varphi := \CM^\ar_\chi \cap \CM^\alg_\varphi \ \hbox{
(with Definition \ref{defI4e} of $\CM^\ar_\chi$),} \]
introduces another kind of computations.
Indeed, the Main Theorem on abelian fields in the literature
is concerned by algebraic definitions similar to $\wh \CM_\varphi$
or $\wh \CM^\varphi$, but our conjecture given in the 1970's used  
$\CM^\ar_\varphi$ and new analytic expressions giving $\order \CM^\ar_\chi$, 
justifying the conjectural values of $\order \CM^\ar_\varphi$ for finite $\CM_K$'s.

\smallskip
It is immediate to verify that, in the non semi-simple case $p \mid g_\chi$,
$(\CM_\varphi^\alg : \CM_\varphi^\ar)$ is equal to the order
of the capitulation kernel of $\J_{K_\chi/k_p}$, 
where $k_p$ is the subfield of $K_\chi$ such that $[K_\chi : k_p] = p$.
In the semi-simple case $p \nmid \order G_\chi$, 
$\CM \simeq \CM_\varphi  \oplus \big[ P_\varphi(\sigma_\chi) \cdot \CM \big]$ 
whatever the definitions (see again Examples of Appendix \ref{ex12}).

\subsection{Arithmetic factorization of \texorpdfstring{$\order \BM_K$}{Lg} and 
\texorpdfstring{$\order \CM_K$}{Lg}}\label{subI6}
Let $\BM$ be an arithmetic $\G$-family where all the $\Z[\G]$-modules
$\BM_K$, $K \in \CK$, are finite; then we can state:

\begin{theorem} \label{theoI5}
Let $K/\Q$ be a cyclic extension and assume that for all sub-extension $k/k'$ of $K/\Q$, 
the maps $\Norm_{k/k'}$ are surjective. Then:
\[\order \BM_K = \prd_{\chi \in \CX_K} \order \BM^\ar_\chi, \]
where $\BM^\ar_\chi := \{x \in \BM_{K_\chi}, \ \Norm_{K_\chi/k}(x) = 1,\  \forall \,
k \varsubsetneqq K_\chi\}$ (Definition \ref{defI4e}).

\smallskip\noindent
Assuming only the cyclicity of the $p$-Sylow subgroup of $G_K$, one obtains, 
$\order \CM_K = \prd_{\chi \in \CX_K} \order \CM^\ar_\chi$. 
\end{theorem}

\begin{proof}
One may replace the $\BM_k$, $k \subseteq K$, by the finite $\Z_p[G_K]$-modules 
$\CM_k :=\BM_k \otimes \Z_p$, for all primes dividing $\order \BM_K$, using the previous 
results, then globalizing at the end. 
Two classical lemmas are necessary.

\begin{lemma} \label{lemI4}
Assume that $p \nmid [k : k']$. If $\Norm_{k/k'} : \CM_k \too \CM_{k'}$ is surjective (resp. 
if $\J_{k/k'} : \CM_{k'} \too \CM_k$ is injective), then $\J_{k/k'}$ is injective (resp. $\Norm_{k/k'}$ 
is surjective).
\end{lemma}

\begin{proof} From Proposition \ref{propI2}, we know that $\Norm_{k/k'} \circ \J_{k/k'}
= [k : k']$; whence the proofs since $[k : k']$ is invertible modulo $p$.
\end{proof}

Put $G_K = G_0 \oplus H$, where  $G_0$ is a subgroup of prime-to-$p$ order
and $H$ (cyclic of order $p^n$) is the $p$-Sylow subgroup of $G_K$.
Let $K_0$ (resp. $K'_n$) be the field fixed by $H$ (resp. $G_0$).

\smallskip
The set of subfields of $K$ is of the form $\{K_{\chi_i^{}},\ \chi^{}_i \in \CX_K, \,0 \leq i \leq n \}$,  
where $\chi^{}_i$ is the rational character above $\psi^{}_i := \psi^{}_{0}\, \psi_p^{p^{n-i}}$, where
$\psi_p^{} \in \Psi_{K'_n}$ is of order $p^n$ and $\psi^{}_{0} \in \Psi_{K_0}$; thus 
$K_{\chi_i^{}}$ is the compositum $K_{\chi_0^{}} K'_i$:

\subsubsection{Schema I} \label{figI}
\unitlength=0.75cm
\[\vbox{\hbox{\hspace{0.5cm}\vspace{0.2cm}
\begin{picture}(11.5,3.7)
\put(4.5,3.0){\line(1,0){2.6}}
\put(2.0,3.0){\line(1,0){1.5}}
\put(4.5,1.5){\line(1,0){2.6}}
\put(2.0,1.5){\line(1,0){1.6}}
\put(4.5,0.0){\line(1,0){2.6}}
\put(2.6,0.0){\line(1,0){1.2}}
\put(1.5,1.9){\line(0,1){0.75}}
\put(1.5,0.4){\line(0,1){0.75}}
\put(4.00,1.9){\line(0,1){0.75}}
\put(4.00,0.4){\line(0,1){0.75}}
\put(7.6,0.65){\ft$p^i$}
\put(7.5,1.9){\line(0,1){0.75}}
\put(7.5,0.4){\line(0,1){0.75}}

\put(7.2,2.9){\ft$K_n \!=\! K$}
\put(3.7,2.9){\ft$K_{\chi_n^{}}$}
\put(3.8,1.4){\ft$K_{\chi_i^{}}$}
\put(1.3,2.9){\ft$K'_n$}

\put(7.3,1.4){\ft$K_i$}
\put(1.3,1.4){\ft$K'_i$}

\put(3.8,-0.1){\ft$K_{\chi_0^{}}$}
\put(7.3,-0.1){\ft$K_0$}
\put(1.25,-0.1){\ft$K'_0 \!=\! \Q$}
\put(4.2,3.6){\ft$G_0$}
\put(9.1,1.4){\ft$H$}
{\color{red}
\bezier{400}(1.5,3.2)(4.4,3.8)(7.3,3.2)
\bezier{350}(8.5,0.1)(9.6,1.45)(8.5,2.8)
\bezier{250}(4.2,2.7)(5.7,2.35)(7.2,2.7)
\bezier{250}(1.65,2.7)(2.5,2.5)(3.8,2.7)}
\put(5.5,2.25){\ft$g^{}_0$}
\put(2.4,2.2){\ft $\ov G_0$}
\end{picture}}} \]
\unitlength=1.0cm

Let $\CM^*_{K_{\chi_i^{}}} := \Ker(\Norm_{K_{\chi_i^{}}/K_{\chi_{i-1}^{}}})$,
$1 \leq i \leq n$, then put $\CM^*_{K_{\chi_0^{}}} := \CM_{K_{\chi_0^{}}}$.
By assumption, we have the exact sequences of $\Z_p[G_K]$-modules:
\begin{equation}\label{exact}
1 \too \CM^*_{K_{\chi_i^{}}} \tooo \CM_{K_{\chi_i^{}}}  
\ds \mathop{\relbar\lien\relbar\lien\relbar\lien\tooo}_{}^{\!\!\Norm_{K_{\chi_i}\!/\!K_{\chi_{i-1}}}} 
\CM_{K_{\chi_{i-1}^{}}} \too 1,\ \, 1 \leq i \leq n.
\end{equation}

One considers them as exact sequences of $\Z_p[G_0]$-modules. 
The idempotents of this algebra are, for all $\chi_0^{} \in \CX_{K_0}$, of the form:
\[e_{\chi_0^{}} = \ffrac{1}{\order G_0} \sm_{\sigma \in G_0} \chi_0^{} (\sigma^{-1}) \,\sigma 
\in \Z_p[G_0]. \]

From Leopoldt \cite{15Leo}, \cite[Chap. V, \S\,2]{16Leo}, as the norm maps are
surjective and the transfer maps injective, regarding the sub-extensions $k/k'$ of prime-to-$p$
degrees in $K/\Q$, we get the following canonical identifications:

\begin{lemma}
Let $\CM$ be an arithmetic $\G$-family whose elements $\CM_K$ are 
$\Z_p[G_0 \oplus H]$-modules in the above sense. 
Then $\CM_{K_i}^{e_{\chi_0^{}}} \simeq \CM_{K_{\chi_i^{}}}^{e_{\chi_0^{}}}$ and 
$(\CM^*_{K_i})^{e_{\chi_0^{}}} \simeq (\CM^*_{K_{\chi_i^{}}})^{e_{\chi_0^{}}}$.
\end{lemma}

\begin{proof}
For all $i$, we identifie $\Gal(K_i / K'_i)$ with $G_0$ acting by restriction
and put $\ov G_0 := G_0/g^{}_0$, where $g^{}_0 := \Gal(K_n/K_{\chi_n^{}})$.
Thus, by abuse of notation, we identify $\Nu_{K_i/K_{\chi_i^{}}}$ with
$\Nu_{K_n/K_{\chi_n^{}}} =: \Nu_{g_0^{}}$; moreover, since the degrees
of these extensions are prime to $p$, we may identify $\Norm_{K_i/K_{\chi_i^{}}}$ 
with $\Norm_{K_n/K_{\chi_n^{}}} =: \Norm_{g_0^{}}$ and
$\J_{K_i/K_{\chi_i^{}}}$ with $\J_{K_n/K_{\chi_n^{}}} =: \J_{g_0^{}}$.
Thus $\Norm_{g_0^{}}$ is surjective and $\J_{g_0^{}}$ injective.
One computes that $\ds e_{\chi_0^{}} = \ffrac{\nu_{g_0^{}}}
{\order g^{}_0}\,\ov e_{\chi_0^{}}$, where $\ds \ov e_{\chi_0^{}} := 
\ffrac{1}{\order \ov G_0} \sm_{\ov \sigma \in \ov G_0} \chi_0^{}
(\ov \sigma^{-1}) \, \sigma \in \Z_p[G_0]$; but we have:
\begin{equation}\label{iso}
\Nu_{g_0^{}} (\CM_ {K_i}) =\J_{g_0^{}} \circ 
\Norm_{g_0^{}} (\CM_ {K_i}) \simeq \Norm_{g_0}^{}(\CM_ {K_i})
\simeq \CM_{K_{\chi_i^{}}};
\end{equation}
whence $\CM_ {K_i}^{e_{\chi_0^{}}} \simeq \CM_{K_{\chi_i^{}}}^{\ov e_{\chi_0^{}}}$. 
To get $(\CM^*_{K_i})^{e_{\chi_0^{}}} \simeq 
\Norm_{g_0^{}} (\CM^*_{K_i})^{\ov e_{\chi_0^{}}} \simeq 
(\CM^*_{K_{\chi_i^{}}})^{\ov e_{\chi_0^{}}}$, it suffices to verify that, 
for all $i \geq 1$, $\Norm_{g_0^{}}(\CM^*_ {K_i}) = \CM^*_{K_{\chi_i^{}}}$. 
The inclusion $\Norm_{g_0^{}}(\CM^*_ {K_i}) \subseteq \CM^*_{K_{\chi_i^{}}}$
being obvious, let $x \in \CM^*_{K_{\chi_i^{}}}$; we have
$x=\Norm_{g_0^{}}(y)$, $y \in \CM_ {K_i}$, then
$1 = \Norm_{K_{\chi^{}_i}/K_{\chi^{}_{i-1}}}\circ \Norm_{g_0^{}}(y) = 
\Norm_{g_0^{}} \circ \Norm_{K_i/K_{i-1}} (y)$. 
Let $z := \Norm_{K_i/K_{i-1}}(y)$, we have $\Norm_{g_0^{}}(z)=1$; 
applying $\J_{K_{i-1}/K_{\chi_{i-1}^{}}}$, one gets 
$\Nu_{g_0^{}} (z) = 1$; but we have, as for \eqref{iso},
$\Nu_{g_0^{}} (\CM_ {K_{i-1}}) \simeq \CM_{K_{\chi_{i-1}^{}}}$;
whence $z=1$, $y \in \CM^*_{K_i}$ and $x \in \Norm_{g_0^{}}(\CM^*_{K_i})$.
\end{proof}

From \cite[Chap.\,I, \S\,1, 2; formula (6), p. 21]{15Leo}
or our previous norm computations since $p \nmid \order G_0$,
we have the relations (surjectivity of the norms and Lemma~\ref{lemI4}):
\begin{equation*}
\left\{\begin{aligned}
\CM_{K_{\chi_i^{}}}^{\ov e_{\chi_0^{}}} & = \{ x \in \CM_{K_{\chi_i^{}}}, 
\ \, \Norm_{K_{\chi_i^{}}/k}(x)=1
\ \hbox{for all $k$, \, $K'_i \subseteq k  \varsubsetneqq K_{\chi_i^{}}$} \}, \\
\CM^{*\,\ov e_{\chi_0^{}}}_{K_{\chi_i^{}}} & =  \{ x \in \CM^*_{K_{\chi_i^{}}}, 
\ \, \Norm_{K_{\chi_i^{}}/k}(x)=1
\ \hbox{for all $k$, \, $K'_i \subseteq k  \varsubsetneqq K_{\chi_i^{}}$} \} .
\end{aligned}\right.
\end{equation*}

From the norm definitions of $(\CM^\ar_{K_{\chi_i^{}}})_{\chi_0^{}}$ and from:
\[\CM^*_{K_{\chi_i^{}}} 
:= \{x \in \CM_{K_{\chi_i^{}}},\ \, \Norm_{K_{\chi_i^{}}/K_{\chi_{i-1}^{}}}(x) = 1\}, \] 
it follows that $\CM_{K_{\chi_i^{}}}^{*\,\ov e_{\chi_0^{}}} = \CM^\ar_{\chi_i^{}}$,
for all $i \geq 1$.
In the finite case, this yields, using the above, the exact sequence \eqref{exact}
and $\CM^*_{K_0} := \CM_{K_0}$:
\begin{equation}\label{semicontext}
\left\{\begin{aligned}
&\ \  \prd_{i = 0}^n \  \order \CM_{K_{\chi_i^{}}}^{*\,\ov e_{\chi_0^{}}}  =
\order \CM_{K_0}^{*\,\ov e_{\chi_0^{}}} \prd_{i =1}^n
 \frac{\order \CM_{K_i}^{\ov e_{\chi_0^{}}}}{\order \CM_{K_{i-1}}^{\ov e_{\chi_0^{}}}} 
 = \order \CM_K^{\ov e_{\chi_0^{}}}, \\
& \prd_{\chi \in \CX_K} \order \CM^\ar_\chi  = 
\prd_{\,\chi_0^{}} \order \CM_K^{\ov e_{\chi_0^{}}} =  \order \CM_K.
\end{aligned}\right.
\end{equation}
Which ends the proof of the theorem and gives useful relations.
\end{proof}

The assumption on the surjectivity of the norms is fulfilled for class groups $\BH$
(resp. $p$-class groups $\CH$ and $p$-torsion groups $\CT$), as soon as $K/\Q$ 
(resp. the maximal $p$-sub-extension of $K/\Q$) is cyclic, whence totally ramified,
class field theory implying the claim (see Remark \ref{nonramified}\,(i)).

\section{Semi-simple decomposition of \texorpdfstring{$\CA_\chi := 
\Z_p[G_\chi]/(P_\chi(\sigma_\chi))$}{Lg}}\label{algebra}

Let $\CM$ be a $\G$-family of $\Z_p[\G]$-modules provided with norms 
and transfer maps as usual. From $\psi \in \Psi$ given, 
there exist unique $\psi^{}_0$, $\psi_p^{} \in \Psi$ such that $\psi = \psi^{}_0 \,\psi_p^{}$,
$\psi^{}_0$ of prime-to-$p$ order and $\psi_p^{}$ of $p$-power order.
We restrict the study to $K := K_\chi$ for the rational character $\chi$ above $\psi$, so that,
from the previous \S\,\ref{subI6}, $G_K$ becomes $G_\chi = G_0 \oplus H$ of order
$g^{}_\chi = g^{}_{\chi_0^{}}\! \cdot p^n$. 

\smallskip
We shall use what we call the ``semi-simple idempotents'' of $\Z_p[G_\chi]$:
\begin{equation}\label{simpleidemp}
e^{\varphi_0^{}} := \ffrac{1}{g^{}_{\chi_0^{}}} \sm_{\sigma \in G_0}
\varphi_0^{} (\sigma^{-1})\, \sigma \in \Z_p[G_0],   
\end{equation}
where $\varphi_0^{}$ is the $p$-adic character over $\psi^{}_0$.

\subsection{Semi-simple decomposition of the 
\texorpdfstring{$\CA_\chi$}{Lg}-modules \texorpdfstring{$\CM^\alg_\chi$}{Lg}}
The algebra $\CA_\chi$ occurs naturally because the $\CM^\alg_\chi$ are, by definition,
$\Z_p[G_\chi]$-modules annihilated by $P_\chi(\sigma_\chi)$, then modules
over $\CA_\chi$; this algebra is an integral domain if and only if $p$ does not split 
in $\Q(\mu_{g^{}_\chi})/\Q$. We shall see that it is semi-simple even when 
$G_\chi$ is not of prime-to-$p$ order.

\begin{theorem} \label{theoI2}
Let $\CM$ be a $\G$-family of $\Z_p[\G]$-modules. 

\smallskip
(i) For all $\chi \in \CX$ we get, by means of the irreducible $p$-adic 
characters $\varphi \in \Phi$, the decompositions
$\CM^\alg_\chi = \plus_{\varphi \mid \chi} \CM^\alg_\varphi \ \, 
\hbox{(cf. Definition \ref{defI2})}$. 

\noindent
More generally, if $\CM'_\chi$ is a sub-$\CA_\chi$-module of $\CM^\alg_\chi$, then
$\CM'_\chi \!= \oplus_{\varphi \mid \chi} \CM'_\varphi$, where $\CM'_\varphi 
= \{x' \in \CM'_\chi,\ x'^{P_\varphi (\sigma_\chi)} = 1 \} \subseteq \CM^\alg_\varphi$.

\smallskip
(ii) The sub-$\CA_\chi$-modules $\CM^\alg_\varphi$, $\varphi \mid \chi$, 
coincide with the $(\CM^\alg_\chi{})^{e^{\varphi_0^{}}}$'s, 
where $e^{\varphi_0^{}}$ is the semi-simple idempotent \eqref{simpleidemp} 
associated to $\varphi_0^{}$ above the component $\psi^{}_0$ of prime-to-$p$ order
of $\psi \mid \varphi \mid \chi$. 

\smallskip
(iii) These modules $\CM^\alg_\varphi$, $\CM'_\varphi$ are canonically 
$\Z_p[\mu_{g^{}_\chi}]$-modules by means of the choice of $\psi \mid \varphi$
and the action $\sigma \in G_\chi \mapsto \psi(\sigma) \in \mu_{g_\chi^{}}$.
\end{theorem}

\begin{proof}
One may suppose that $g^{}_\chi \equiv 0 \pmod p$, otherwise we are in the 
semi-simple case and the proof is obvious \cite[Part II]{17Or}.

\smallskip
Let $\varphi^{}_1$ and $\varphi^{}_2$ be two distinct $p$-adic characters dividing $\chi$
(if $\chi = \varphi$ is $p$-adic irreducible, the result is trivial).
Put $P_{\varphi^{}_1} =: Q_1$, $P_{\varphi^{}_2}(X) =: Q_2$ (cf. \S\,\ref{ssI4b}
for the definition of $P_\varphi$). The following lemma is probably clear for
cyclotomic polynomials, but it is not general (e.g., for $p=5$, take 
$P=x^4 - 2 x^3 + 55 x^2 - 54 x + 379$, irreducible in $\Z[X]$, 
giving, in $\Z_5[X]$, $P \equiv (x^2 + 24 x + 12) \cdot (x^2 + 24 x + 17) \pmod {5^2}$
and the PARI relation ${\sf bezout(x^2+24*x+12,\, x^2+24*x+17) = [-1/5,1/5,1]}$).

\begin{lemma}\label{lemI5}
There exist $U_1, U_2 \in \Z_p[X]$ such that $U_1Q_1 + U_2Q_2=1$.
\end{lemma}

\begin{proof}
We assume that such a relation does not exist and we shall find a contradiction. 
Since the distinct polynomials $Q_1$ and $Q_2$ are irreducible in $\Q_p[X]$,
one may write a B\'ezout relation in $\Z_p[X]$ of the form
(with $U_1$, $U_2$ not both in $p\Z_p[X]$):
\[\hbox{$U_1Q_1 + U_2Q_2=p^k$, $\ k \geq 1$,}\]
choosing $U_1$ (resp. $U_2$) of degree less than the degree of $Q_2$
(resp. of $Q_1$); moreover, since $Q_1$ and $Q_2$ are monic,
one may suppose that (for instance):
\[U_2 \notin p\Z_p[X], \] 
otherwise, since $k \geq 1$, necessarily $U_1 \in p\Z_p[X]$, which is excluded.

\smallskip
Let $D_\chi$ be the decomposition group of $p$ in $\Q(\mu_{g^{}_\chi})/\Q$ and let
$\zeta \in \mu_{g^{}_\chi}$ be a root of $Q_1$
($\zeta$ is of order $g^{}_\chi$ and the other roots are the $\zeta^a$
for Artin symbols $\sigma_a \in D_\chi$); we then have:
\begin{equation}\label{defk}
\hbox{$U_2(\zeta)\, Q_2(\zeta) = p^k$ in $\Z_p[\mu_{g^{}_\chi}]$;} 
\end{equation}
but $Q_2(X) = \prod_{\sigma_a \in D_\chi} (X - \zeta_1^a)$,
where $\zeta_1 =: \zeta^c$, for some $\sigma_c \notin D_\chi$; thus:
\[Q_2(\zeta) = \prd_{\sigma_a \in D_\chi} (\zeta - \zeta_1^a) =
\prd_{\sigma_a \in D_\chi} (\zeta - \zeta^{ac}) = 
\prd_{\sigma_a \in D_\chi}\big [\, \zeta (1 - \zeta^{ac-1})\, \big]. \]

Recall that $g^{}_\chi =  g^{}_{\chi_0^{}} p^n$, $n \geq 1$. Then
$1 - \zeta^{ac-1}$ is non invertible in $\Z_p[\mu_{g^{}_\chi}]$ if and only if
$ac -1 \equiv 0 \pmod {g^{}_{\chi_0^{}}}$, which implies $\sigma_a\sigma_c \in D_\chi$
since $\Gal(\Q(\mu_{g^{}_\chi})/\Q(\mu_{g^{}_{\chi_0^{}}}))\subseteq D_\chi$ because of the 
total ramification of $p$ in the $p$-extension, but $\sigma_a \in D_\chi$ implies $\sigma_c \in D_\chi$
(absurd). So $Q_2(\zeta)$ is a $p$-adic unit, whence, from \eqref{defk},
$U_2(\zeta) \equiv 0 \pmod {p^k}, \ k \geq 1$.

\smallskip
Denote by ${\mathfrak p}$ the maximal ideal of $\Z_p[\mu_{g^{}_\chi}]$
and let $\ov F_p := \Z_p[\mu_{g^{}_\chi}]/ {\mathfrak p}$ be the residue field; 
for any $P \in \Z_p[X]$, let $\ov P$ be its image in $\F_p[X]$ and let $\ov \zeta$ 
be the image of $\zeta$ in $\ov F_p$. We have, in $\F_p[X]$:
\begin{equation}\label{q1}
\ov Q_1 = (\ov Q_0)^e, 
\end{equation}
where $e = p^{n-1} (p-1)$ (ramification
index of $p$ in $\Q(\mu_{g^{}_\chi})/\Q$) and where $\ov Q_0$ is irreducible in 
$\F_p[X]$ (i.e., the irreducible polynomial of $\ov \zeta$, in fact
that of the image of a generator of $\mu_{g^{}_{\chi_0^{}}}$).

\smallskip
With these notations, any polynomial $P \in \Z_p[X]$ such that 
$P(\zeta) \equiv 0 \pmod{{\mathfrak p}}$ is such that 
$\ov P \in \ov Q_0\, \F_p[X]$; in particular, it is the case of $\ov U_2$,
so we will have, in $\F_p[X]$ (since $\ov U_2 \ne 0$ in $\F_p[X]$ by assumption),
$\ov U_2 = \ov A\, (\ov Q_0)^\alpha$, $\alpha \geq 1$, $\ov A \ne 0$, 
$\ov Q_0 \nmid \ov A$. We may assume that $A$, $Q_0 \in \Z_p[X]$ 
have same degrees as their images in $\F_p[X]$. This yields:
\[\hbox{$U_2 =  A \, Q_0^\alpha + p B$, $B \in \Z_p[X]$,} \] 
thus $U_2 (\zeta)=  A (\zeta) \, Q_0^\alpha (\zeta) + p \,B (\zeta) \equiv 0 \pmod {p^k}$,
whence $A (\zeta) \, Q_0^\alpha (\zeta) \equiv 0 \pmod p$. But 
$A (\zeta)$ is a $p$-adic unit (since $\ov Q_0 \nmid \ov A$), which gives:
\begin{equation}\label{qzero}
Q_0^\alpha (\zeta) \equiv 0 \pmod p. 
\end{equation}

Let's show that $\alpha \geq e$; the unique case where, possibly, $p \mid g^{}_\chi$
and $e=1$ is the case $p=2$, $n=1$; this case trivially gives $\alpha \geq e$.
Consider the $g^{}_{\chi_0^{}}\!$th cyclotomic polynomial.
Assuming $e > 1$, we have:
\[P_{g^{}_{\chi_0^{}}}(\zeta) = \prd_{a \in (\Z/g^{}_{\chi_0^{}} \Z)^*}
(\zeta - \zeta^{p^n a}) = \prd_ {a} [ \,\zeta (1 - \zeta^{p^n a-1}) \,] ; \] 
 $\zeta^{p^n a-1}$ is of $p$-power order if and 
only if $p^n a \equiv 1 \pmod {g^{}_{\chi_0^{}}}$;
taking into account the domain of $a$, this defines 
$a_0$ such that $p^n a_0  \equiv 1 \pmod {g^{}_{\chi_0^{}}}$, 
whence $p^n a_0 \not\equiv 1 \pmod {p g^{}_{\chi_0^{}}}$
and $1 - \zeta^{p^n a_0  - 1} \in {\mathfrak p} \setminus {\mathfrak p}^2$, thus the fact
that $P_{g^{}_{\chi_0^{}}}(\zeta) \in {\mathfrak p} \setminus {\mathfrak p}^2$; it follows,
from $P_{g^{}_{\chi_0^{}}} = C\,Q_0^\beta + p D$, $\beta \geq 1$,
$C, D \in \Z_p[X]$, $C(\zeta)
\not\equiv 0 \pmod {{\mathfrak p}}$, that $P_{g^{}_{\chi_0^{}}}(\zeta) \equiv  
C(\zeta)\,Q_0^\beta(\zeta)  \pmod {{\mathfrak p}^e}$, thus $Q_0^\beta(\zeta) 
\in {\mathfrak p} \setminus {\mathfrak p}^2$ since $e > 1$.
This implies $\beta=1$ and $Q_0(\zeta) \in {\mathfrak p} \setminus {\mathfrak p}^2$.

\smallskip
The congruence \eqref{qzero}, written
$Q_0^\alpha(\zeta) \equiv 0 \pmod {{\mathfrak p}^e}$, implies $\alpha \geq e$ and
$U_2 = A' \,Q_0^e + p \,B$, where $A' := A \,Q_0^{\alpha - e}$; but  we also have from \eqref{q1}:
\[\hbox{$Q_1 = Q_0^e + p \,T$, $T \in \Z_p[X]$,} \] 
hence $U_2 = A' \,(Q_1 - p \,T) + p \,B = A' \, Q_1 + p \,S$, $S \in \Z_p[X]$. 
Since $A \ne 0$ may be chosen monic by assumption, $A' \ne 0$ is monic, 
$U_2$ is of degree larger or equal to that of $Q_1$ (absurd), whence $A'=0$
and $\ov U_2 = 0$, contrary to the assumption $U_2 \notin p\Z_p[X]$.
\end{proof}

Give now some properties of the system of idempotents of $\CA_\chi =
 \Z_p [G_\chi]/ (P_\chi(\sigma_\chi))$.

\smallskip
Let $\{\varphi^{}_1, \ldots , \varphi_{g^{}_p}^{}\}$ be the set of distinct $p$-adic 
characters dividing $\chi$ (thus, $g^{}_p \mid \phi(g^{}_{\chi_0^{}})$ is the number of 
prime ideals dividing $p$ in $\Q(\mu_{g^{}_{\chi_0^{}}})/\Q$, so that, only the case
$g^{}_p=1$ is trivial for the Finite AMC); from the 
property of co-maximality, given by Lemma \ref{lemI5}, one may write:
\begin{equation}\label{e(X)}
\Z_p[X]\big / (P_\chi(X) ) 
\simeq \prd_{u = 1}^{g^{}_p} \Z_p[X]/ \big(Q_u(X)\big) \simeq( \Z_p[\mu_{g^{}_{\chi}}])^{g_p}.
\end{equation}

There exist elements $e_{\varphi_u^{}}(X) \in \Z_p[X]$, whose images 
modulo $P_\chi(X)$ constitute an exact system of orthogonal idempotents 
of $\Z_p [X]/ (P_\chi(X))$. Whence the system of orthogonal idempotents 
$e_{\varphi_u^{}}(\sigma_\chi)$ of $\Z_p[G_\chi]$.

\smallskip
Since $(\CM^\alg_\chi)^{P_\chi(\sigma_\chi)} = 1$, we obtain 
(in the algebraic meaning):
\begin{equation}\label{cmchi}
\CM^\alg_\chi = \plus_{u=1}^{g^{}_p} (\CM^\alg_\chi)^{e_{\varphi_u^{}}(\sigma_\chi)}. 
\end{equation}

It remains to verify that:
\[(\CM^\alg_\chi)^{e_{\varphi_u^{}}(\sigma_\chi)} = \CM^\alg_{\varphi_u^{}}
= \{x \in \CM^\alg_\chi, \  x^{P_{\varphi_u^{}}(\sigma_\chi)} = 1\}. \]
If $x \in (\CM^\alg_\chi)^{e_{\varphi_u^{}}(\sigma_\chi)}$, $x = y^{e_{\varphi_u^{}}(\sigma_\chi)}$
with $y \in \CM^\alg_\chi$; then we have $x^{P_{\varphi_u^{}}(\sigma_\chi)} = 
y^{e_{\varphi_u^{}}(\sigma_\chi) P_{\varphi_u^{}}(\sigma_\chi)}$, but
$e_{\varphi_u^{}}(\sigma_\chi) \, P_{\varphi_u^{}}(\sigma_\chi)) 
\equiv  0 \! \pmod{P_\chi(\sigma_\chi)}$, whence 
$ y^{e_{\varphi_u^{}}(\sigma_\chi)\, P_{\varphi_u^{}}(\sigma_\chi)} = 1$ 
since $y \in \CM^\alg_\chi$ and $x \in \CM^\alg_{\varphi_u^{}}$.

\smallskip\noindent
If $x \in \CM^\alg_{\varphi_u^{}}$, then $x^{P_{\varphi_u^{}}(\sigma_\chi)} = 1$;
writing $x = \prod_{j=1}^{g^{}_p} x^{e_{\varphi^{}_v}(\sigma_\chi)}$, we get
$e_{\varphi^{}_v}(\sigma_\chi) \equiv \delta_{u, v}\!\! \pmod{P_{\varphi_u^{}}(\sigma_\chi)}$, thus 
$e_{\varphi^{}_v}(\sigma_\chi) \equiv 0 \pmod {P_{\varphi_u^{}}(\sigma_\chi)}$ for $v \ne u$
and $x^{e_{\varphi^{}_v}(\sigma_\chi)} = 1$, for  $v \ne u$. 
Whence $x = x^{e_{\varphi_u^{}}(\sigma_\chi)}$.

\medskip
In the algebra $\CA_\chi = \Z_p [G_\chi]/ (P_\chi(\sigma_\chi))$, we obtain two systems of 
idempotents, that is to say, the images in $\CA_\chi$ of the 
$e_{\varphi^{}_{u,0}} \in \Z_p[G_0]$, where
$\varphi^{}_{u,0}$ is above the component $\psi^{}_{u,0}$, of prime-to-$p$ order, of $\psi^{}_u$,
and that of the $e_{\varphi_u^{}}(\sigma_\chi)$ corresponding to $\varphi_u^{}$.
Fixing the character $\varphi_u^{} =: \varphi$ above $\psi =: \psi^{}_0\,\psi_p^{}$
and its non $p$-part $\varphi_0^{}$  above $\psi^{}_0$, we consider both:
\begin{equation}\label{idempotent0}
e^{\varphi_0^{}} := \ffrac{1}{g^{}_{\chi_0^{}}} \sm_{\sigma \in G_0} \varphi_0^{} (\sigma^{-1})\, \sigma
\end{equation} 
and $e_{\varphi_0^{}}(\sigma_\chi)$ defined as follows by means 
of polynomial relations in $\Z[X]$ deduced from \eqref{e(X)}:
\begin{equation}\label{idempotent1}
\left\{\begin{aligned}
e_{\varphi_0^{}}(\sigma_\chi) = & \Lambda_\varphi(\sigma_\chi) \cdot \!\!
\prd_{\varphi' \ne \varphi}  P_{\varphi'}(\sigma_\chi), \ \hbox{such that:}\ \  \\ 
& \Lambda_\varphi(X) \cdot \prd_{\varphi' \ne \varphi} P_{\varphi'}(X) 
\equiv 1 \!\!\pmod  {P_\varphi(X)};
\end{aligned}\right.
\end{equation} 

\noindent
we denote $e_{\varphi_0^{}}(\sigma_\chi)$ simply by $e_{\varphi_0^{}}$, which is 
legitimate by Lemma \ref{propI3}.

\smallskip
To verify that $(\CM^\alg_\chi)^{e^{\varphi_0^{}}} = (\CM^\alg_\chi)^{e_{\varphi_0^{}}}$, it suffices 
to show that $e^{\varphi_0^{}}$ and $e_{\varphi_0^{}}$ correspond to the same simple 
factor of the algebra $\CA_\chi$. 
For this, we remark that the homomorphism defined, for 
the fixed character $\varphi$, by $\sigma_\chi \mapsto \psi (\sigma_\chi)$, 
$\psi \mid \varphi$, induces a surjective homomorphism
$\CA_\chi \too \Z_p[\mu_{g^{}_\chi}]$ whose kernel 
is equal to $\plus_{\varphi  \ne \varphi} \CA_\chi \, e_{\varphi'_0}$.

\noindent
Thus, to show that $\CA_\chi e^{\varphi_0^{}} = \CA_\chi e_{\varphi_0^{}}$,
it suffices to show that $\psi(e^{\varphi_0^{}}) \ne 0$; but, from \eqref{idempotent0}, 
$e^{\varphi_0^{}}$ is a sum of the idempotents $e_{\psi'_0} = \frac{1}{g^{}_{\chi_0^{}}}
\!\!\sm_{\sigma^{} \in  G_0} \psi'_0(\sigma) \sigma^{-1}$  
where $\psi'_0 \mid \varphi_0^{}$. It follows, since $\psi = \psi^{}_0 \,\psi_p^{}$,
that $\psi (\sigma) = \psi^{}_0 (\sigma)$ and then:
\[\psi(e_{\psi'_0}) = \ffrac{1}{g^{}_{\chi_0^{}}}
\sm_{\sigma^{} \in G_0} \psi'_0(\sigma)\psi (\sigma)^{-1}
=\ffrac{1}{g^{}_{\chi_0^{}}}\sm_{\sigma^{} \in G_0} 
\psi'_0(\sigma)\psi^{}_0 (\sigma)^{-1}, \]
which is zero for all $\psi'_0$ except $\psi'_0=\psi_0$ where $\psi(e_{\psi_0^{}})=1$.
Whence $\psi(e^{\varphi_0^{}}) \ne 0$. 
Let $\CM_\chi^\alg$ as $\CA_\chi$-module; on may write 
$\CM^\alg_\chi = \plus_{\varphi \mid \chi} (\CM^\alg_\chi)^{e_{\varphi_0^{}}}$ 
(from \eqref{cmchi}) but 
$(\CM^\alg_\chi)^{e_{\varphi_0^{}}}$ coincides with 
$(\CM^\alg_\chi)^{e^{\varphi_0^{}}} = \CM^\alg_\varphi$ (Definition \eqref{idempotent0}); 
then, due to the properties of the $e_{\varphi_0^{}}$ (defined by \eqref{idempotent1}):
\[(\CM^\alg_\chi)^{e_{\varphi_0^{}}} = \{x \in \CM^\alg_\chi,\ \, x^{P_\varphi(\sigma_\chi)} = 1\}
= \CM^\alg_\varphi .\] 

Denote by $e_{\varphi_0^{}}$ any of these two 
semi-simple idempotents $e^{\varphi_0^{}}$ or $e_{\varphi_0^{}}$.

\smallskip
If $\CM'_\chi$ is a sub-$\CA_\chi$-module of $\CM^\alg_\chi$, then:
\[\CM'_\varphi := (\CM'_\chi)^{e_{\varphi_0^{}}} =  
\{x' \in \CM'_\chi,\, x'^{P_\varphi(\sigma_\chi)} = 1\}. \]
Since $\CA_\chi\,e_{\varphi_0^{}} \simeq \Z_p[\mu_{g^{}_\chi}]$, $\CM^\alg_\varphi$
and $\CM'_\varphi$ are canonically $\Z_p[\mu_{g^{}_\chi}]$-modules.

\smallskip
This finishes the proof of Theorem \ref{theoI2}.
\end{proof}

\subsection{Semi-simple decomposition of the \texorpdfstring{$\CA_\chi$}{Lg}-modules 
\texorpdfstring{$\CM^\ar_\chi$}{Lg}}
From Definition \ref{defI4e}, 
$\CM^\ar_\chi := \{x \in \CM_{K_\chi},\ \Norm_{K_\chi/k}(x) = 1,\ 
\hbox{for all $k \varsubsetneqq K_\chi$} \}$.
This invites to give the following arithmetic definition: 

\begin{definition} \label{defI3}
Let $\CM$ be an arithmetic family of $\Z_p[\G]$-modules. For any $\varphi \mid \chi$,
$\chi \in \CX$, $\varphi \in \Phi$, we define the arithmetic $\Z_p[\mu_{g^{}_\chi}]$-module:
\[\CM^\ar_\varphi := \CM^\alg_\varphi \cap \CM^\ar_\chi 
 = \{x \in  \CM^\alg_\varphi ,\ \,\Norm_{K_\chi/k}(x) = 1,\ \hbox{for all 
$k \varsubsetneqq K_\chi$} \}.\]
\end{definition}

Note that if $p \mid g_\chi$, then the norm conditions may be limited to 
$\Norm_{K_\chi/k_p}(x) = 1$, with $[K_\chi : k_p] = p$.

\begin{remark}\label{idempotents}
So, $\CM^\ar_\varphi = (\CM^\ar_\chi)^{e_{\varphi_0^{}}}$,
$e_{\varphi_0^{}}$ being defined by \eqref{idempotent0} or \eqref{idempotent1}, and
$\CM^\ar_\varphi$ is a sub-$\Z_p[\mu_{g^{}_\chi}]$-module of $\CM^\alg_\varphi$.
In the sequel, we use both the notations $\CM^\ar_\varphi =
 \{x \in \CM^\ar_\chi,\, x^{P_\varphi(\sigma_\chi)} = 1\}$ and
$(\CM^\ar_\chi)^{e_{\varphi_0^{}}}$. In some recent papers
we privilege the notations $\CM^\ar_\varphi = (\CM^\ar_\chi)^{e_{\varphi_0^{}}}
=: (\CM^\ar_\chi)_{\varphi_0^{}}$, giving, for instance, the $\varphi$-component 
$(\CE_{K_\chi}/\wh \CE_{K_\chi} \!\cdot \CF_{K_\chi})_{\varphi_0^{}}$
of $\CE_{K_\chi}/\wh \CE_{K_\chi}\! \cdot \CF_{K_\chi}$, since this module is 
a $\chi$-object for trivial reasons.
\end{remark}

So, we have the arithmetic version of Theorem \ref{theoI2}:

\begin{theorem}\label{theoI2bis}
Let $\CM$ be a $\G$-family of $\Z_p[\G]$-modules. Then we get,
for all $\chi \in \CX$, the decomposition
$\CM^\ar_\chi = \plus_{\varphi \mid \chi} \CM^\ar_\varphi$.
\end{theorem}

\subsection{Summary of the properties of the \texorpdfstring{$\G$}{Lg}-families 
\texorpdfstring{$\CM^\alg$, $\CM^\ar$}{Lg}}\label{mainresults}
From Notations \ref{notations}, Theorems \ref{theoI5}, \ref{theoI2}, 
\ref{theoI2bis}, Definitions \ref{defI2}, \ref{defI4e}, \ref{defI3}:

\medskip
(i) Recall that $P_\chi$ (resp. $P_\varphi \mid P_\chi$)
is the $g^{}_\chi$th global cyclotomic poly\-nomial (resp. the local $\varphi$-cyclotomic 
polynomial); let's define:
\begin{equation*}
\left\{\begin{aligned}
\CM^\alg_\chi & := \big\{ x \in \CM_{K_\chi},\  x^{P_\chi(\sigma_\chi) } = 1 \big\}, \\
\CM^\alg_\varphi & := \big\{ x \in \CM_{K_\chi},\  x^{P_\varphi (\sigma_\chi) } = 1 \big\}
=: (\CM^\alg_\chi)_{\varphi_0^{}} ,\\
\CM^\ar_\chi & := \{x \in \CM^\alg_\chi,\ \Norm_{K_\chi/k}(x) = 1,\ \forall\, 
k \varsubsetneqq K_\chi \}, \\
\CM^\ar_\varphi & :=  \big\{ x \in \CM_\varphi^\alg,\  \Norm_{K_\chi/k}(x) = 1,\ \forall\, 
k \varsubsetneqq K_\chi \} =: (\CM^\ar_\chi)_{\varphi_0^{}} .
\end{aligned}\right.
\end{equation*} 

Then $\CM^\alg_\chi = \plus_{\varphi \mid \chi} \CM^\alg_\varphi$ and 
$\CM^\ar_\chi = \plus_{\varphi \mid \chi} \CM^\ar_\varphi$.
All these components are $\Z_p[\mu_{g_\chi^{}}]$-modules
via $\sigma \in G_\chi \mapsto \psi(\sigma)$, for
$\psi \mid \chi$, $\psi \mid \varphi$, respectively.

\medskip
(ii) Assume that the maximal $p$-sub-extension of $K/\Q$
is cyclic and such that for all its sub-extensions $k/k'$, 
the norms $\Norm_{k/k'}$ are surjective. Then, if $\CM_K$ is finite,
$\order \CM_K = \prod_{\chi \in \CX_K} \order \CM^\ar_\chi =
\prod_{\varphi \in \Phi_K} \order \CM^\ar_\varphi$.

\section{Application to relative class groups}\label{secII}

\subsection{Arithmetic definition of relative class groups}\label{subII1}

We will apply the previous results using first odd characters $\chi$ giving 
$\BH^\alg_\chi$ and $\BH^\ar_\chi$.  The case of even characters requires some 
deepening of Leopoldt's results \cite{15Leo}; it will be considered in the next section.

\smallskip
For $K \in \CK$, we denote by $\BH_K$ the class group of $K$
in the ordinary sense. If $K$ is imaginary, with maximal real subfield $K^+$, 
we define the relative class group of $K$:
\begin{equation}\label{classmoins}
(\BH^\ar_K)^- := \{h \in \BH_K,\ \, \Norm_{K/K^+}(h)=1\} 
\end{equation}

\noindent
(the notation $\BH^\ar$ recalls that the definition of the minus part uses 
the arithmetic norm and not the algebraic one $\Nu_{K/K^+}$).

\smallskip
It is classical to put $\BH_K^+ := \BH_{K^+}$; since $K/K^+$
is ramified for the real infinite places of $K^+$, class field theory 
implies that $\Norm_{K/K^+}$ is surjective for class groups in the 
ordinary sense, giving the exact sequence:
\[1\to (\BH^\ar_K)^- \too \BH_K\ds 
\mathop{\tooo}_{}^{{}_{\Norm_{K\!/\!K^+}}} \BH_{K^+} = \BH_K^+ \to 1\] 
and the formula:
\begin{equation}\label{order+-}
\order \BH_K = \order (\BH^\ar_K)^- \cdot \order \BH_K^+. 
\end{equation}

We denote by $\CH_K$ (resp. $(\CH^\ar_K)^-$ and $\CH_K^+ := \CH_{K^+}$), 
the $p$-Sylow subgroup of $\BH_K$ (resp. $(\BH^\ar_K)^-$ and $\BH_K^+$). 
For the $\Z_p[\G]$-modules $\CH_K$, we introduce the $\CA_\chi$-modules 
$\CH^\alg_\chi$ and $\CH^\ar_\chi$ for $\chi \in \CX$, then their $\varphi$-components 
(Definitions \ref{defI2}, \ref{defI4e}, \ref{defI3}) which are $\Z_p[\mu_{g^{}_\chi}]$-modules.

\subsection{Proof of the equality \texorpdfstring{$\BH^\ar_\chi = \BH^\alg_\chi$}{Lg}, for all
\texorpdfstring{$\chi \in \CX^-$}{Lg}}\label{subII2}

To prove this equality and then the equalities $\CH^\ar_\varphi = \CH^\alg_\varphi$, 
$\varphi \mid \chi$, it is sufficient to consider, for any $p \geq 2$, the $p$-Sylow 
subgroups $\CH_{K_\chi}$ and to prove the equality of the $\chi$-components 
$\CH_\chi^\alg$, $\CH_\chi^\ar$.

\begin{lemma}\label{lemII1}
Assume that $\CH^\ar_\chi \varsubsetneqq \CH^\alg_\chi$. Then there exists a unique
sub-extension $K_{\chi'}$ of $K_\chi$, such that 
$[K_\chi : K_{\chi'}] = p$ (i.e., if $\psi \mid \chi$ then $\chi'$ is above $\psi' = \psi^p$), 
and a class $h \in \CH_\chi^\alg$ such that 
$h' := \Norm_{K_\chi / K_{\chi'}}(h)$ fulfills the following properties:

\smallskip
(i) For all prime $\ell \ne p$ dividing $g^{}_\chi$, $\Nu_{K_{\chi'}/k'_\ell}(h')=1$,
where $k'_\ell$ is the unique sub-extension of $K_{\chi'}$ such that $[K_{\chi'} : k'_\ell] = \ell$;

(ii) $\J_{K_\chi / K_{\chi'}}(h') = 1$;

(iii) $h'$ is of order $p$ in $\CH_{K_{\chi'}}$.
\end{lemma}

\begin{proof}
Indeed, if $[K_\chi : \Q]$ is prime to $p$, we are in the semi-simple case
and $\CH^\alg_\chi = \CH^\ar_\chi$. So we assume that $p \mid [K_\chi : \Q]$, 
whence the existence and unicity of $K_{\chi'}$.

\smallskip
Let $h \in \CH^\alg_\chi$, $h \notin \CH^\ar_\chi$, and let $h' := \Norm_{K_\chi / K_{\chi'}}(h)$.
Let $\ell \mid g^{}_\chi$, $\ell \ne p$.

\smallskip
(i) We have the following diagram where $k_\ell^{}$ is the unique sub-extension
of $K_\chi$ such that $[K_\chi : k_\ell^{}] = \ell$ and then $k'_\ell = k_\ell^{} \cap K_{\chi'}$:
\subsubsection{Schema II} \label{figII}
\unitlength=0.75cm
\[\vbox{\hbox{\hspace{-5.0cm} \vspace{-0.15cm}
\begin{picture}(11.5,2.9)
\put(6.1,2.50){\line(1,0){2.8}}
\put(6.1,0.50){\line(1,0){2.8}}
\put(5.50,0.9){\line(0,1){1.20}}
\put(9.40,0.9){\line(0,1){1.20}}
\put(9.1,2.4){\ft$K_\chi$}
\put(5.4,2.4){\ft$k_\ell^{}$}
\put(5.4,0.40){\ft$k'_\ell$}
\put(9.1,0.4){\ft$K_{\chi'}$}
{\color{red}\put(10.2,2.4){$h$}
\put(10.2,0.4){\ft$h' \!:=\! \Norm_{K_\chi / K_{\chi'}}(h)$}}
\put(7.5,2.6){\ft$\ell$}
\put(7.5,0.6){\ft$\ell$}
\put(9.5,1.4){\ft$p$}
\put(5.65,1.4){\ft$p$}
\end{picture}   }} \]
\unitlength=1.0cm

We have $\Nu_{K_\chi /k_\ell^{}}(h)=1$ since $h \in \CH^\alg_\chi$; applying 
$\Norm_{K_\chi / K_{\chi'}}$, we get $\Nu_{K_{\chi'} /k'_\ell}(h')=1$.

\smallskip
(ii) We have $\J_{K_\chi / K_{\chi'}}(h') = 
\J_{K_\chi / K_{\chi'}} \circ \Norm_{K_\chi / K_{\chi'}}(h) = \Nu_{K_\chi / K_{\chi'}}(h)=1$
since $h \in \CH^\alg_\chi$.

\smallskip
(iii) Since the class $h'$ capitulates in $K_\chi$, its order is $1$ or $p$.
Suppose that $h'=1$; 
for $\ell \ne p$, the maps $\J_{K_\chi/k_\ell^{}}$ and $\J_{K_{\chi'}/k'_\ell}$
are injective, so $\Norm_{K_\chi /k_\ell^{}}(h) = 1$,
for all $\ell \ne p$ dividing $g^{}_\chi$; since moreover $h'=\Norm_{K_\chi / K_{\chi'}}(h)=1$, 
this yields by definition $h \in \CH^\ar_\chi$ (absurd).
\end{proof}

\begin{lemma}\label{lemII2}
Let $K/k$ be a cyclic extension of degree $p$ 
and Galois group $G =: \langle \sigma \rangle$. Let $\BE_k$
and $\BE_K$ be the unit groups of $k$ and $K$, respectively. Consider
the transfer map 
$\J_{K/k} : \CH_k \to \CH_K$; then $\Ker (\J_{K/k})$ is isomorphic to a 
subgroup of $\Hom^1(G,\BE_K) \simeq \BE_K^*/\BE_K^{1-\sigma}$
(where $\BE_K^* = \Ker(\Nu_{K/k}))$.
The group $\BE_K^*/\BE_K^{1-\sigma}$ is of exponent $1$ or $p$.
\end{lemma}

\begin{proof}
Let $\BZ_k$ and $\BZ_K$ be the rings of integers of $k$ and $K$, respectively;
let $\cl_k({\mathfrak a}) \in \CH_k$, with ${\mathfrak a} \BZ_K =: (\alpha)\BZ_K$,
$\alpha \in K^\times$. We then have $\alpha^{1-\sigma} =: \varepsilon \in 
\BE_K^*$. The map, which associates with 
$\cl_k({\mathfrak a}) \in \Ker (\J_{K/k})$ the class of $\varepsilon$ modulo 
$\BE_K^{1-\sigma}$, is obviously injective.

\smallskip
If $\varepsilon \in \BE_K^*$, then $1 = \varepsilon^{1+ \sigma + \cdots + \sigma^{p - 1}} 
= \varepsilon^{p + (\sigma - 1) \Omega}$, $\Omega \in \Z[G]$; whence
$\varepsilon^p \in \BE_K^{1-\sigma}$.
\end{proof}

\subsubsection{Study of the case \texorpdfstring{$p \ne 2$}{Lg}}\label{ssII2b}
We are in the context of Lemma \ref{lemII1}. Put $K:=K_\chi$ and $k:=K_{\chi'}$; 
then $K/k$ is of degree $p$ and the class $h' = \Norm_{K/k}(h) \in \CH_k$ is of
order $p$ and capitulates in $K$.

\smallskip
Assume that $K$ is imaginary (i.e., $\chi$ is odd, thus $h \in (\CH^\ar_K)^-$); 
since $K/k$ is of degree $p \ne 2$, $k$ is also imaginary and $h' \in (\CH^\ar_k)^-$. 

\smallskip
We introduce the maximal real subfields, giving the diagram:
\subsubsection{Schema III} \label{figIII}
\unitlength=0.75cm
\[\vbox{\hbox{\hspace{-3.7cm} \vspace{-0.2cm}
\begin{picture}(11.5,3.0)
\put(6.1,2.50){\line(1,0){1.8}}
\put(6.1,0.50){\line(1,0){1.8}}
\put(5.50,0.9){\line(0,1){1.20}}
\put(8.3,0.9){\line(0,1){1.20}}
\put(8.1,2.4){\ft$K$}
\put(5.3,2.4){\ft$K^+$}
\put(5.3,0.40){\ft$k^+$}
\put(8.2,0.4){\ft$k$}
{\color{red}\put(8.95,2.4){\ft$h$}
\put(8.95,0.4){\ft$h' \!:=\! \Norm_{K / k}(h)$}}
\put(7.0,2.65){\ft$2$}
\put(7.0,0.65){\ft$2$}
\put(7.9,1.4){\ft$p$}
\put(5.12,1.4){\ft$p$}
\put(9.0,1.35){\ft$G\!=\!\langle \sigma \rangle$}
{\color{red}\bezier{250}(8.6,0.6)(9.2,1.5)(8.6,2.4)}
\end{picture}   }} \]
\unitlength=1.0cm

\begin{lemma}\label{lemII3}
Let $\mu_K^*$ be the $p$-torsion sub-group of $\BE_K^*$, that is to say the 
set of $p$-roots of unity $\zeta$ of $K$ such that $\Norm_{K/k}(\zeta)=1$. Then 
the image of $(\CH^\ar_k)^- \cap \Ker(\J_{K/k})$, by the map $\Ker(\J_{K/k}) \to 
\BE_K^*/\BE_K^{1-\sigma}$ of Lemma \ref{lemII2}, is contained in the image 
of $\mu_K^*$ modulo $\BE_K^{1-\sigma}$.
\end{lemma}

\begin{proof}
Let $q$ be the map $\BE_K^* \to \BE_K^*/\BE_K^{1-\sigma}$. 
Denote by $x \mapsto \ov x$ the complex conjugation in $K$. 
If $h' \in (\CH^\ar_k)^- \cap \Ker(\J_{K/k})$, then $\Norm_{k/k^+}(h') = 1$ and
$\Nu_{k/k^+}(h') = h' \ov {h'} =1$; if $h' = \cl_k({\mathfrak a})$ we then 
have ${\mathfrak a} \ov {\mathfrak a}=a \BZ_k$, $a \in k^\times$, and
${\mathfrak a}\BZ_K \ov {\mathfrak a}\BZ_K=a \BZ_K$, with 
${\mathfrak a}\BZ_K=(\alpha) \BZ_K$ and $\ov {\mathfrak a}\BZ_K
=(\ov \alpha) \BZ_K$, $\alpha \in K^\times$
(since ${\mathfrak a}$ and $\ov {\mathfrak a}$ become 
principal in $K$), which yields relations of the form 
$\alpha^{1-\sigma} = \varepsilon$, $\ov \alpha^{1-\sigma} 
= \ov \varepsilon$, $\varepsilon, \ov \varepsilon \in \BE_K^*$.
From the relation ${\mathfrak a} \ov {\mathfrak a}=a \BZ_k$, 
one obtains, in $K$, $\alpha \ov \alpha = \eta a$, $\eta \in \BE_K$, then
$\alpha^{1-\sigma} \ov \alpha^{1-\sigma} = \eta^{1-\sigma}$, giving
$\varepsilon \ov \varepsilon = \eta^{1-\sigma}$.

\smallskip
From \cite[Satz 24]{10Has}, $\varepsilon = \varepsilon^+\, \zeta$, 
$\varepsilon^+ \in \BE_{K^+}$,
$\zeta \in \mu_K^{}$. So $q(\varepsilon \ov \varepsilon) = q(\varepsilon^{+2})=1$.
Since $p$ is odd and $\BE_K^*/\BE_K^{1-\sigma}$ of exponent divisor of $p$,
$\varepsilon^+ \in \BE_K^{1-\sigma}$; since $\varepsilon \in \BE_K^*$, we have
$\zeta \in \BE_K^*$, whence:
\[q(\varepsilon) = q(\zeta) \in q(\mu_K^*) = \mu_K^*/(\BE_K^{1-\sigma} \cap \mu_K^*),\]
and the lemma.
\end{proof}

\begin{lemma}\label{lemII4}
The group $q(\mu_K^*)$ (of order $1$ or $p$) is of order $p$ if and
only if $\mu_K^* = \langle \zeta_1 \rangle$ and
$\BE_K^{1-\sigma} \cap \langle \zeta_1 \rangle = 1$,
where $\zeta_1$ is of order $p$.
\end{lemma}

\begin{proof}
A direction being obvious, assume that $q(\mu_K^*) = 
\mu_K^*/ (\BE_K^{1-\sigma} \cap \mu_K^*)$ is of order $p$
and let $\zeta$ be a generator of $\mu_K^*$ (necessarily, $\zeta \ne 1$). 
If $\zeta \in k$, then $\Norm_{K/k}(\zeta) = \zeta^p$, so $\zeta^p=1$ and
$\zeta = \zeta_1 \in k$. 

\smallskip
If $\zeta \notin k$, $K=k(\zeta)$; it follows that 
$\zeta_1 \in k$ and that $\zeta^p \in k$  (since $[K : k]=
[\Q(\zeta) : k \cap \Q(\zeta)]=p$), thus $K/k$ is a Kummer extension 
of the form $K = k(\sqrt[p]{\zeta_r})$, $\zeta_r$ of order $p^r$,
$r \geq 1$, $\zeta = \zeta_{r+1}$,
and $\zeta^{1-\sigma}=\zeta_1$, giving $\Norm_{K/k}(\zeta) = \zeta^p = 1$, 
hence $\zeta = \zeta_1 \in k$ (absurd).
So we have $\zeta = \zeta_1 \in k$ and
$\BE_K^{1-\sigma} \cap \mu_K^* \subseteq \langle \zeta_1 \rangle$. 
Thus, $q(\mu_K^*)$ being of order $p$, necessarily $\BE_K^{1-\sigma} \cap \mu_K^*=1$.
\end{proof}

\begin{lemma}\label{lemII5}
If $(\CH^\ar_k)^- \cap \Ker(\J_{K/k}) \ne 1$, this group is of order $p$ and $K/k$ is a 
Kummer extension of the form $K=k(\sqrt[p]{a})$, $a \in k^\times$, 
$a\BZ_k={\mathfrak a}^p$, the ideal ${\mathfrak a}$ of $k$ being non-principal
(such a Kummer extension is said to be ``of class type'').
\end{lemma}

\begin{proof}
If $h' \in (\CH^\ar_k)^- \cap \Ker(\J_{K/k})$, $h' := \cl_k({\mathfrak a}) \ne 1$, this means that 
${\mathfrak a} \BZ_K = \alpha \BZ_K$, $\alpha \in K^\times$; so $\alpha^{1-\sigma} = \varepsilon$,
$\varepsilon \in \BE_K^*$; from Lemma \ref{lemII4}, $q(\varepsilon) = q(\zeta_1)^\lambda$,
hence $\varepsilon = \zeta_1^\lambda \eta^{1-\sigma}$, $\eta \in \BE_K$, whence
$\alpha^{1-\sigma} = \zeta_1^\lambda  \eta^{1-\sigma}$ and in the equality
${\mathfrak a}\BZ_K=\alpha \BZ_K$ one may suppose $\alpha$ chosen 
modulo $\BE_K$ such that
$\alpha^{1-\sigma} = \zeta_1^\lambda$; moreover we have $\lambda \not\equiv 0 \pmod p$,
otherwise $\alpha$ should be in $k$ and ${\mathfrak a}$ should be principal.
Thus $\alpha^{1-\sigma} = \zeta'_1$ of order $p$ and $\alpha^p = a \in k^\times$,
whence $K=k(\alpha)$ is the Kummer extension $k(\sqrt[p]{a})$;
we have $a\BZ_K={\mathfrak a}^p \BZ_K$, hence $a \BZ_k = {\mathfrak a}^p$, 
since extension of ideals is injective.
\end{proof}

We shall show now that the context of Lemma \ref{lemII5} is not possible for
a cyclic extension $K/\Q$, which will apply to $K_\chi/\Q$:

\subsubsection{Schema IV} \label{figIV}
\unitlength=0.6cm
\[\vbox{\hbox{\hspace{-4.2cm} \vspace{-0.2cm}
\begin{picture}(11.5,4.2)
\put(6.6,4.0){\line(1,0){3}}
\put(6.5,2.50){\line(1,0){3}}
\put(6.5,0.50){\line(1,0){3}}
\put(6.00,2.9){\line(0,1){0.8}}
\put(6.00,0.9){\line(0,1){1.3}}
\put(10.00,2.9){\line(0,1){0.8}}
\put(10.00,0.9){\line(0,1){1.3}}
\put(9.8,3.9){\ft$K \!=\! k(\sqrt[p]{a})$}
\put(5.75,3.9){\ft$K'$}
\put(5.75,2.4){\ft$k'$}
\put(9.8,2.4){\ft$k$}
\put(9.8,0.4){\ft$K_0$}
\put(5.8,0.4){\ft$\Q$}
\put(10.15,3.15){\ft$p$}
\put(10.15,1.5){\ft$p^{n-1}$}
\end{picture}   }} \]
\unitlength=1.0cm

Since $K=k(\sqrt[p]{a})$, with $a \BZ_k = {\mathfrak a}^p$,
only the prime ideals dividing $p$ can ramify in $K/k$.
Consider the above decomposition of the extension $K/\Q$ for $p \ne 2$,
with $K/K_0$ and $K'/\Q$ cyclic of $p$-power degree $p^n$,
$K/K'$ and $K_0/\Q$ of prime-to-$p$ degree, and let $\ell$ be a prime 
number totally ramified in $K'/\Q$ (such a prime does 
exist since $G_{K'} \simeq \Z/p^n\Z$); 
this prime is then totally ramified in $K/K_0$, hence in $K/k$, which implies 
$\ell=p$ and $p$ is the unique ramified prime in $K'/\Q$.

\smallskip
This identifies the extension $K'/\Q$. Its conductor is $p^{n+1}$, $n \geq 1$, 
since $p\ne 2$; thus $K'$ is the unique sub-extension of 
degree $p^n$  of $\Q(\mu_{p^{n+1}})$ and $k'$ is the unique sub-extension of 
degree $p^{n-1}$ of $\Q(\mu_{p^n})$ (in other words, $K'$ is contained in the
cyclotomic $\Z_p$-extension). Since $\zeta_1 \in k$, one has
$\mu_{p^n} \subset k$, $\mu_{p^{n+1}} \subset K$ and
$\mu_{p^{n+1}} \not\subset k$, so $K=k(\zeta) = k(\sqrt[p]{\zeta^p})$, with
$\zeta$ of order $p^{n+1}$.

\smallskip
It suffices to apply Kummer theory which shows that $k(\sqrt[p]{a}) = 
k(\sqrt[p]{\zeta^p})$ implies $a = \zeta^{\lambda p}b^p$, with
$p \nmid \lambda$ and $b \in k^\times$; so $a\BZ_k=b^p \BZ_k={\mathfrak a}^p$,
whence ${\mathfrak a} = b \BZ_k$ principal (absurd).

\smallskip
So in the case $p \ne 2$, for $K/\Q$ imaginary cyclic and $K/k$ cyclic of degree
$p$, we have the relation $(\CH^\ar_k)^- \cap \Ker(\J_{K/k})=1$ (injectivity of $\J_{K/k}$
on the relative $p$-class group).

\subsubsection{Case \texorpdfstring{$p=2$}{Lg}}\label{ssII2c}
The extension $K/\Q$ is still imaginary cyclic, $k$ is necessarily 
equal to $K^+$ and $\sigma$ is the complex conjugation $s_\infty$. 

\smallskip
From  \cite[Satz 24]{10Has} the ``index of units'' $Q_K^-$
is trivial in the cyclic case; thus for all $\varepsilon \in \BE_K^*$, $\varepsilon = 
\varepsilon^+ \zeta$, $\varepsilon^+ \in k$, $\zeta$ root of unity of $2$-power 
order; then $\Norm_{K/k}(\varepsilon) = 1$ yields $\varepsilon^{+ 2}=1$, thus 
$\varepsilon^+ = \pm 1$ and $\varepsilon = \zeta' = \pm \zeta$; 
since $K/\Q$ is cyclic (whence $\Q(\zeta)/\Q$ cyclic), we shall have 
$\varepsilon \in \{1, -1, i, -i \}$.
Recall that $h' = \Norm_{K/k}(h) \in \Ker(\J_{K/k})$,  $h'=\cl_k({\mathfrak a}) \ne 1$,
with ${\mathfrak a}\BZ_K = \alpha \BZ_K$ and $\alpha^{1-\sigma} =\varepsilon
\in \BE_K^*$. One may assume
$\varepsilon \in \{ -1, i, -i \}$ ($\varepsilon \ne 1$ since $\alpha \notin k^\times$):

\smallskip
(i) Case $\varepsilon = -1$. Then $\alpha^{1-\sigma}=-1$, $\alpha^2 =: a \in k^\times$,
$\alpha \notin k^\times$,
and we get the Kummer extension $K = k(\sqrt a)$ with
$a\BZ_k= {\mathfrak a}^2$, ${\mathfrak a}$ non-principal (Kummer extension of class type).

\smallskip
(ii) Case $\varepsilon = \pm i$. Then $\alpha^{1-\sigma}=\pm i$ with
$-1 = (\pm i)^{1-\sigma}$; one may assume 
$\alpha^{1-\sigma}= i$. This yields $\alpha^2 i^{-1} \in k^\times$. Put 
$\alpha^2 = i c$, $c \in k^\times$; it follows 
${\mathfrak a}^2\BZ_K = \alpha^2 \BZ_K = c \BZ_K$, hence 
${\mathfrak a}^2 = c \BZ_k$. 

\smallskip
Let $\tau$ be a generator of
$G_K$; one has $\alpha^{2 \tau}=i^\tau c^\tau
=-i c^\tau =-c^{\tau-1}\alpha^2$, hence $\alpha^{2 \tau}=\alpha^2 d$,
$d := -c^{\tau-1} \in k^\times$; we obtain $(\alpha \BZ_K)^{2 \tau} = (\alpha \BZ_K)^2 d\BZ_K$,
thus ${\mathfrak a}^{2 \tau}\BZ_K = {\mathfrak a}^2\BZ_K  d\BZ_K$
giving ${\mathfrak a}^{2 \tau} = {\mathfrak a}^2 d\BZ_k$. 

\smallskip
If $d \in k^{\times 2}$, $d = e^2$, $e \in k^\times$, and 
${\mathfrak a}^\tau \sim {\mathfrak a}$
saying that $h'$ is an invariant class in $k/\Q$.

\smallskip
If $d \notin k^{\times 2}$, the relation $\alpha^{2 \tau} = \alpha^2 d$
shows that $d =( \alpha^{\tau-1})^2 \in K^{\times 2}$; from Kummer theory,
since $K=k(\sqrt d) = k(i)$, one obtains $d =-\delta^2$, $\delta \in k^\times$,
and ${\mathfrak a}^{2\tau} = {\mathfrak a}^2 \delta^2 \BZ_K$, still giving
${\mathfrak a}^{\tau} = {\mathfrak a} \cdot \delta \BZ_k$ and an invariant class in $k/\Q$. 

\smallskip
But $K$ is the direct compositum over $\Q$
of $k = K^+$ and $\Q(i)$ and must be cyclic, so $[k : \Q]$ is necessarily odd 
and an invariant class in $k/\Q$ is of odd order giving the principality
of ${\mathfrak a}$ in $k$ (absurd). 

\smallskip
So, only case (i) is a priori possible. 

\smallskip
Consider the following diagram, with $K/K_0$ and $K'/\Q$ cyclic of 
$2$-power order, then $K/K'$ and $K_0/\Q$ of odd degree,
where we recall that $a \BZ_k = {\mathfrak a}^2$ with 
${\mathfrak a}$ non-principal and ${\mathfrak a} \BZ_K=
\alpha \BZ_K$, $\alpha \in K^\times$.
Similarly, since $K/k$ is only ramified at $2$,
then $K/K_0$ and $K'/\Q$ are totally ramified at $2$, the conductor of $K'$
is a power of $2$, say $2^{r+1}$, $r \geq 1$ 
($K'$ is an imaginary cyclic subfield of $\Q(\mu_{2^{r+1}})$):

\subsubsection{Schema V} \label{figV}
\unitlength=0.55cm
\[\vbox{\hbox{\hspace{-4.2cm} \vspace{-0.1cm}
\begin{picture}(11.5,4.3)
\put(6.6,4.0){\line(1,0){3}}
\put(6.5,2.50){\line(1,0){3}}
\put(6.5,0.50){\line(1,0){3}}
\put(6.00,2.9){\line(0,1){0.70}}
\put(6.00,0.9){\line(0,1){1.20}}
\put(10.00,2.9){\line(0,1){0.70}}
\put(10.00,0.9){\line(0,1){1.20}}
\put(9.8,3.9){\ft$K \!=\! k(\sqrt a)$}
\put(5.75,3.9){\ft$K'$}
\put(5.75,2.4){\ft$k'$}
\put(9.8,2.4){\ft$k=K^+$}
\put(9.8,0.4){\ft$K_0$}
\put(5.8,0.4){\ft$\Q$}
\put(10.2,3.2){\ft$2$}
\put(12.9,3.15){\ft$\langle \, s_\infty \,\rangle$}%
{\color{red}\bezier{250}(12.5,2.5)(13.1,3.2)(12.5,3.9)}
\end{picture}   }} \]
\unitlength=1.0cm

The Kummer extension
$K'/k'$ is $2$-ramified of the form $K' = k'(\sqrt {a'})$, $a' \in k'^\times$. 
So we have $a' \BZ_{k'}= {\mathfrak a}'^2$ or $a' \BZ_{k'}= {\mathfrak a}'^2 {\mathfrak p}'$,
where ${\mathfrak p}' \mid 2$ in $k'$. But all the subfields of 
$\Q(\mu_{2^\infty})$ have a trivial $2$-class group; thus, one may suppose that $a'$ 
is, up to $k'^{\times 2}$, a unit or an uniformizing parameter of $k'$.
Then $K = k(\sqrt{a'})$ is not of class type (absurd); so $h'=1$. Whence:

\begin{proposition}\label{propII1}
For any imaginary cyclic extension $K/\Q$ and any relative extension $K/k$
of prime degree, $(\CH^\ar_k)^- \cap \Ker(\J_{K/k}) = 1$ if $p \ne 2$
(the relative classes of $k$ do not capitulate in $K$), then
$\Ker(\J_{K/K^+}) = 1$ if $p = 2$ (the real $2$-classes of $k=K^+$ do not capitulate 
in $K$).  
\end{proposition}

Using the order formula \eqref{order+-} yields:

\begin{corollary}\label{coroII1}
We get $\J_{K/K^+}(\CH_{K^+}) \simeq \CH_K^+ := \CH_{K^+}=\Norm_{K/K^+}(\CH_K)$ 
and the direct sum $\CH_K = (\CH^\ar_K)^- \oplus \J_{K/K^+}(\CH_{K^+})$.
\end{corollary}

We have obtained the following result about relative class groups:

\begin{theorem}\label{theoII1}
Let $K$ be an imaginary cyclic field of maximal real subfield $K^+$. Let
$p$ be any prime number and set $\CH = \BH \otimes \Z_p$. Define:
\begin{equation} 
\left\{\begin{aligned}
(\CH^\ar_K)^- & := \{h \in \CH_K,\, \Norm_{K/K^+}(h)=1\} \\
(\CH^\alg_K)^-& := \{h \in \CH_K,\,  \Nu_{K/K^+}(h)=1\}. 
\end{aligned}\right.
\end{equation} 

\noindent
Then $\CH^\ar_K = \CH^\alg_K$, $\CH^\ar_\varphi = \CH^\alg_\varphi$ 
for all $\varphi \in \Phi_K^-$, $(\BH^\ar_K)^- = (\BH^\alg_K)^-$.
\end{theorem}

\begin{proof}
For all subfield $k$ of $K$ with $[K : k]=p$, 
$\J_{K/k}$ is injective on  $(\CH^\ar_k)^-$ if $p \ne 2$ and $\J_{K/K^+}$ 
is injective on $\CH_{K^+}$ for $p=2$; so $\Nu_{K/k} = \J_{K/k} \circ  \Norm_{K/k}$
yields $(\CH^\ar_K)^- = (\CH^\alg_K)^-$ from Definition \ref{defI4e},
then $(\BH^\ar_K)^- = (\BH^\alg_K)^-$ by globalization.
\end{proof}

We shall write simply $\BH_K^-$ for the two notions ``$\alg$'' and ``$\ar$''
in the cyclic case.
Using Theorem \ref{theoI2} we may write, for all $\chi \in \CX^-$,
$\order \CH^\alg_\chi=\order \CH^\ar_\chi = 
\prd_{\varphi \mid \chi} \order \CH^\ar_\varphi$.

\begin{corollary}\label{theoII2+}
Let $K/\Q$ be an imaginary cyclic extension. Then:
\[\hbox{$\order \BH_K^+ = \prod_{\chi \in \CX_K^+} \order \BH^\ar_\chi$
\ \ \ $\&$ \ \ \ $\order \BH_K^- = \prod_{\chi \in \CX_K^-} \order \BH^\ar_\chi$.}\]
\end{corollary}

\begin{proof}
To apply Theorem  \ref{theoI5},
we shall prove that all the arithmetic norms are surjective in any
sub-extension $k/k'$ of $K/\Q$; we do this for each $p$-class group; so 
the proof of the surjectivity is only necessary in the sub-extensions $k/k'$ of
$p$-power degree; then we use the fact that this property holds as soon as 
$k/k'$ is totally ramified at some place. This comes from Remark \ref{nonramified}
about cyclic extensions. So Theorem \ref{theoI5} implies
$\order \BH_K = \prd_{\chi \in \CX_K} \order \BH^\ar_\chi$. 

From \eqref{order+-}, $\order \BH_K = \order \BH_K^- \cdot \order \BH_K^+$
and we can also apply Theorem \ref{theoI5} to the maximal real subfield $K^+$
of $K$, giving
$\order \BH_K^+ = \prd_{\chi \in \CX_K^+} \order \BH^\ar_\chi$, 
whence the formulas taking into account the relation $\BH^\ar_\chi = \BH^\alg_\chi$
for odd characters (Theorem \ref{theoII1}).
\end{proof}

\subsection{Computation of \texorpdfstring{$\order \BH^\ar_\chi$}{Lg} for 
\texorpdfstring{$\chi \in \CX^-$}{Lg}}\label{subII3}

For an arbitrary imaginary extension $K/\Q$, we have (e.g., from \cite[p. 12]{10Has} 
or \cite[Theorem 4.17]{Was}) the formula:
\[\order \BH_K^- = Q_K^- w_K^- \prd_{\psi \in \Psi_K^-}
\big(- \hbox{$\frac{1}{2}$}\, \BB_1(\psi^{-1}) \big), \ \, 
\BB_1(\psi^{-1}) := \ffrac{1}{f_\chi} \sm_{a \in [1, f_\chi[} \psi^{-1} (\sigma_a)\, a ,\]

\noindent
where $w_K^-$ is the order of the group
of roots of unity of $K$ and $Q_K^-$ the index of units; from \cite[Satz 24]{10Has},
$Q_K^-=1$ when $K/\Q$ is cyclic. 
Recall that $\BH^\ar_\chi := \{h \in \BH_{K_\chi},\ \Norm_{K_\chi/k}(x) = 1$,
for all $k \varsubsetneqq K_\chi \}$; then:

\begin{theorem} \label{theoII2}
Let $\chi \in \CX^-$, let $g^{}_\chi$ be the order of $\chi$, $f_\chi$ its conductor; 
then $\order \BH^\ar_\chi = \order \BH^\alg_\chi = 2^{\alpha_\chi} \cdot w_\chi \cdot 
\prd_{\psi \mid \chi}\big(- \hbox{$\frac{1}{2}$}\, \BB_1(\psi^{-1}) \big)$, 
where $\alpha_\chi = 1$ (resp. $\alpha_\chi = 0$) if $g^{}_\chi$ is a 
$2$-power (resp. if not) and:

\smallskip
\quad (i) $w_\chi = 1$ if $K_\chi$ is not an imaginary cyclotomic field;

\smallskip
\quad (ii) $w_\chi = p$ if $K_\chi = \Q(\mu_{p^n})$, $p \ne 2$ prime, $n \geq 1$;

\smallskip
\quad (iii) $w_\chi = 2$ if $K_\chi = \Q(\mu_4^{})$ for $p=2$.
\end{theorem}

\begin{proof}
We use \cite[Proposition III\,(g)]{18Or} or \cite[Chap. I,\,\S\,1\,(4)]{15Leo}
recalled in Theorem \ref{chiformula};
it is sufficient to prove that for any imaginary cyclic extension $K/\Q$,
$\order \BH_K^- = \prd_{\chi \in  \CX_K^-}
\big( 2^{\alpha_\chi} \cdot w_\chi \cdot 
\prd_{\psi \mid \chi}\big(- \hbox{$\frac{1}{2}$}\, \BB_1(\psi^{-1}) \big) \big)$, 
the expected equality will come from Theorem \ref{theoII1} and the relation:
\[\order \BH_K^- = \prd_{\chi \in  \CX_K^-} \order \BH^\ar_\chi. \] 
So, it remains to prove that
$\!\prod_{\chi \in  \CX_K^-} \big( 2^{\alpha_\chi} \cdot w_\chi \big)=w_K^-$. 

\smallskip
Consider the following diagram, where $K/K_0$ and $K'/\Q$ are cyclic 
of $2$-power degree and where $K/K'$ and $K_0/\Q$ are of odd degree. :

\subsubsection{Schema VI} \label{figVI}
\unitlength=0.55cm
\[\vbox{\hbox{\hspace{-3.0cm} \vspace{-0.3cm}
\begin{picture}(11.5,4.3)
\put(6.6,4.0){\line(1,0){3}}
\put(6.5,2.50){\line(1,0){3}}
\put(6.5,0.50){\line(1,0){3}}
\put(6.00,2.9){\line(0,1){0.70}}
\put(6.00,0.9){\line(0,1){1.20}}
\put(10.00,2.9){\line(0,1){0.70}}
\put(10.00,0.9){\line(0,1){1.20}}
\put(9.8,3.9){\ft$K$}
\put(5.75,3.9){\ft$K'$}
\put(5.75,2.4){\ft$K'^+$}
\put(9.8,2.4){\ft$K^+$}
\put(9.8,0.4){\ft$K_0$}
\put(5.8,0.4){\ft$\Q$}
\put(10.2,3.15){\ft$2$}
\put(5.6,3.15){\ft$2$}
\end{picture}   }} \]
\unitlength=1.0cm

\noindent
As $K^+$ and $K'^+$ are real, $\alpha_\chi = 0$, except 
when $g_\chi$ is a $2$-power, hence for the unique $\chi_0^{}$
defining $K'$ for which $\alpha_\chi = 1$; whence 
$\prod_{\chi \in  \CX_K^-} 2^{\alpha_\chi} = 2$.

\smallskip
If $K$ does not contain any cyclotomic field (different from $\Q$), then
$w_K^- = 2$, moreover, all the $w_\chi$ are trivial and the required
equality holds in that case. 
So, let $\Q(\mu_{p^n})$, $n \geq 1$, be
the largest cyclotomic field contained in $K$; this yields two possibilities: 

\subsubsection{Schema VII} \label{figVII}
\unitlength=0.6cm
\[\vbox{\hbox{\hspace{-5.4cm}\vspace{0.2cm}
\begin{picture}(11.5,4.6)
\put(6.6,4.0){\line(1,0){3}}
\put(6.5,2.50){\line(1,0){3}}
\put(6.5,0.50){\line(1,0){3}}
\put(3.6,0.50){\line(1,0){1.2}}
\put(6.00,2.9){\line(0,1){0.70}}
\put(6.00,0.9){\line(0,1){1.20}}
\put(10.00,2.9){\line(0,1){0.70}}
\put(10.00,0.9){\line(0,1){1.20}}
\put(9.8,3.9){\ft$K$}
\put(5.75,3.9){\ft$K^+$}
\put(4.75,2.4){\ft$\Q(\mu_{p^n})^+$}
\put(9.8,2.4){\ft$\Q(\mu_{p^n})$}
\put(3.0,0.4){\ft$\Q$}
\put(9.8,0.4){\ft$\Q(\mu_{p})$}
\put(4.9,0.4){\ft$\Q(\mu_{p})^+$}
\put(13.5,2.50){\line(1,0){3}}
\put(13.5,0.50){\line(1,0){3}}
\put(13.00,0.9){\line(0,1){1.20}}
\put(17.00,0.9){\line(0,1){1.20}}
\put(12.75,2.4){\ft$K^+$}
\put(16.8,2.4){\ft$K$}
\put(16.8,0.4){\ft$\Q(\mu_4)$}
\put(12.8,0.4){\ft$\Q$} 
\put(6.3,-0.4){\ft$p \ne 2$}
\put(14.6,-0.4){\ft$p=2$}
\end{picture}   }} \]
\unitlength=1.0cm

If $p \ne 2$, $\prod_{\chi \in  \CX_K^-} w_\chi = p^n$ 
(due to the $n$ odd characters defined by the $\Q(\mu_{p^i})$,
$1 \leq i \leq n$) and, for $p=2$, this gives 
$\prod_{\chi \in  \CX_K^-} w_\chi = 2$; whence the result 
(cf. \cite[Chap. III, \S\,33, Theorem 34 and others]{10Has}).
\end{proof}

\begin{remark}\label{remII2}
We have $\order \BH_K^- = \ffrac{Q_K^- w_K^-}{2^{n^-_K}}
\prd_{\chi \in  \CX_K^-} \order \BH^\alg_\chi$, 
for any imaginary extension $K$, where $n_K^-$ is the number
of imaginary cyclic sub-extensions of $K$ of $2$-power degree and $w_K^-$ 
is the $2$-part of $w_K$ (resp. $\frac{1}{2} w_K$) if $\Q(\mu_4^{}) \not\subset K$
(resp. $\Q(\mu_4^{}) \subset K$). See \cite[Remarque II\,2, p. 32]{Gra}.
\end{remark}

\subsection{Annihilation theorem for \texorpdfstring{$\CH_K^-$}{Lg}}
Before significant improvements by means of Stickelberger's elements
(leading to the construction of $p$-adic measures, to index formulas and 
annihilators of various invariants), Iwasawa \cite{11Iwa} proves the following 
formula for the cyclotomic fields $K=\Q(\mu_{p^n})$, $p \ne 2$, $n \geq 1$, 
of Galois group $G_K$:
\[\order \BH_K^- = \big ( \Z[G_K]^- : \BB_K\Z[G_K] \cap \Z[G_K]^- \big ), \] 
where $\Z[G_K]^- := \{\Omega \in  \Z[G_K], \, (1+s_\infty) \cdot \Omega = 0\}$, 
$s_\infty$ being the complex conjugation, and
$\BB_K := \ffrac{1}{p^n} \sm_{a \in [1,\, p^n[,\, p\, \nmid \, a} a \, \sigma_a^{-1}$
where $\sigma_a \in G_K$ denotes the corresponding Artin automorphism.

\smallskip
This formula does not generalize for arbitrary
imaginary extension $K/\Q$ (see the counterexample
given in \cite[p. 33]{Gra}).
Many contributions have appeared
(e.g., \cite{16Leo, 4Gil, 2Coa, Gra2, All1, All2}; for more precise formulas, 
see \cite{Sin}, \cite[\S\,6.2, \S\,15.1]{Was}, among many other).
Nevertheless, we gave in \cite{Gra} another definition in the spirit of the $\varphi$-objects
which succeeded to give a correct formula.

\subsubsection{General definition of Stickelberger's elements}\label{ssII4b}
Let $K \in \CK \setminus \{\Q\}$. Let $f_K=: f  > 1$ be the conductor of $K$ 
and let $\Q(\mu_f^{})$ be the corresponding cyclotomic field.
Define the more suitable writing of the Stickelberger element defined
in \cite[Chap.IV,\,\S\,1]{Gra2} or \cite[Chap.I,\ \S\,1]{Gra3},
from the study of partial z\^eta-functions in \cite[\S\S\,2.1,\,3.2]{2Coa},
and that leads to a new normalized definition of Gauss sums (in the summation, 
integers $a$ are prime to $f$ and Artin symbols are taken over $\Q$):
\[\BB_{\Q(\mu_f^{})} :=-\sm_{a=1}^{f} 
\Big(\ffrac{a}{f} - \ffrac{1}{2} \Big) \cdot \Big(\ffrac{\Q(\mu_f^{})}{a} \Big)^{-1}. \]
 Note that the part $\sum_{a=1}^{f} \big(\frac{\Q(\mu_f^{})}{a} \big)^{-1}$ 
is the algebraic norm $\Nu_{\Q(\mu_f^{})/\Q}$ which does not modify the image 
of $\BB_{\Q(\mu_f^{})}$ by $\psi$, for $\psi \in \Psi$, $\psi \ne 1$.

\smallskip
We shall use two arithmetic $\G$-families: the $\G$-family $\BM$, for which 
$\BM_K=\Z[G_K]$ and the $\G$-family $\BS$ defined by:
\begin{equation}\label{resK}
\left\{\begin{aligned}
\BS_K& := \BB_K \Z[G_K] \cap \Z[G_K], \, \hbox{where} \\
\BB_K& := \Norm_{\Q(\mu_f^{})/K} (\BB_{\Q(\mu_f^{})}) \! =\! -\sm_{a=1}^{f} 
\Big(\ffrac{a}{f} - \ffrac{1}{2} \Big) \Big(\ffrac{K}{a} \Big)^{-1}.
\end{aligned}\right.
\end{equation}

\begin{lemma}\label{lemII6}
For any $c$, prime to $2 f$, let
$\BB_{K}^c :=\ds \Big(1 - c \,\Big(\ffrac{K}{c} \Big)^{\!-1} \Big) \! \cdot \! 
\BB_{K}$; then $\BB_{K}^c \in \Z[{G_K}]$. 
\end{lemma}

\begin{proof}
We have:
\begin{equation*} 
\BB_{K}^c=\ffrac{-1}{f} \sm_a 
\Big[a\, \Big(\ffrac{K}{a} \Big)^{-1} - a c \,\Big(\ffrac{K}{a} \Big)^{-1}
\Big(\ffrac{K}{c} \Big)^{-1}\Big] + \ffrac{1-c}{2} \sm_a \Big(\ffrac{K}{a} \Big)^{-1}.
\end{equation*}

Let $a'_{c} \in [1, f]$ be the unique integer such that $a'_{c} \cdot c \equiv a \pmod {f}$;
put:
\[\hbox{$a'_{c} \cdot  c = a + \lambda_a(c) f$, \ $\lambda_a(c) \in \Z$;} \]
using the bijection $a \mapsto a'_{c}$ 
in the summation of the second term in between $\big [\ \ \big]$ and the relation
$\ds\Big(\ffrac{K}{a'_{c}} \Big) \Big(\ffrac{K}{c} \Big) = \Big(\ffrac{K}{a} \Big)$,
this yields: 
\begin{equation*} 
\begin{aligned}
\BB_{K}^c  
&=\ffrac{-1}{f}  \Big[ \sm_a a \, \Big(\ffrac{K}{a} \Big)^{-1}\!\! - 
\sm_a a'_{c} \cdot c \,\Big(\ffrac{K}{a'_{c}} \Big)^{-1} \! \Big(\ffrac{K}{c} \Big)^{-1}\Big] 
 + \ffrac{1-c}{2} \sm_a \Big(\ffrac{K}{a} \Big)^{-1} \\
& = \ffrac{-1}{f} \sm_a \Big [a - a'_{c} \cdot  c \Big] \Big(\ffrac{K}{a} \Big)^{-1}
+ \ffrac{1-c}{2}\sm_a \Big(\ffrac{K}{a} \Big)^{-1} \\
& =  \sm_a \Big[\lambda_a(c)+ \ffrac{1-c}{2} \Big]
 \Big(\ffrac{K}{a} \Big)^{-1} \in \Z[G_K].
 \end{aligned}
 \end{equation*}

We have $\lambda_{f-a}(c) + \frac{1-c}{2} = 
-\big(\lambda_{a}(c) + \frac{1-c}{2} \big)$, which proves that:
\begin{equation}\label{demiS}
\BB_{K}^c = \BB'^c_{K} \cdot (1-s_\infty), \ \, \BB'^c_{K}\in \Z[G_K],
\end{equation}
useful in the case $p=2$ and giving $\Norm_{K/K^+}(\BB_{K}^c)=0$. 
\end{proof}

\begin{definition}\label{defII1}
Let $K$ be an imaginary abelian field. Put:
\[{\mathfrak A}_K := \{\Omega \in \Z[G_K], \  \Omega\, \BB_K \in \Z[G_K]\}\]
(${\mathfrak A}_K$ is an ideal of $\Z[G_K]$ and  
$\BS_K := \BB_K \cdot {\mathfrak A}_K$ (cf. \eqref{resK}).
Denote by $\Lambda_K \in {\mathfrak A}_K$ the least rational integer 
such that $\Lambda_K \BB_K \in \Z[G_K]$ (thus $\Lambda_K \mid 2 f$, 
where $f$ is the conductor of $K$).

\smallskip\noindent
For $K=K_\chi$, $\chi \in \CX^-$, we put ${\mathfrak A}_{K_\chi} 
=: {\mathfrak A}_\chi$ and $\Lambda_{K_\chi} =: \Lambda_\chi$.
\end{definition}

Since we will only use images by $\psi \in \Psi^-$ of elements of 
$\Q[G_K]$, we can neglect, by abuse, the term 
$\sum_{a=1}^{f} \frac{1}{2} \, \big(\frac{K}{a} \big)^{-1}$ in some reasonings and 
computations, using $\frac{1}{f} \sum_{a=1}^{f} a\, \big(\frac{K}{a} \big)^{-1}$ 
instead of $\BB_K$. 

Note that for any odd $c$ prime to $f$, 
$\Big(1 - c \,\Big(\ffrac{K}{c} \Big)^{-1} \Big)\cdot 
\sm_{a=1}^{f} \ffrac{1}{2} \Big(\ffrac{K}{a} \Big)^{-1}$ is in $\Z[G_K]$ 
and that such considerations only concerns the case $p=2$ when $f$ is an odd 
prime power with $[\Q(\mu_f^{}) : K]$ odd (see Example \ref{ex47} with $K=\Q(\mu_{47})$).

\begin{lemma}\label{lemII7}
Let $\alpha_\sigma$ be the coefficient of $\sigma \in G_K$ in the writing of
$\sm_{a=1}^{f} a\, \Big(\ffrac{K}{a} \Big)^{-1}$ on the canonical basis $G_K$ 
of $\Z[G_K]$; in particular, we have
$\alpha_1 = \sm_{a,\ \sigma_a{}_{\vert_K}=1} \, a$. 
Then $\alpha_\sigma \equiv c\,\alpha_1 \pmod{f}$, where $c$
is a representative modulo $f$ such that $\sigma_c = \sigma^{-1}$.
Thus, we have $\Lambda_K = \ffrac{f}{\gcd(f, \alpha_1)}$.
\end{lemma}

\begin{proof}
The first claim is obvious and $\Lambda_K$ is the least integer $\Lambda$ such that 
$\ffrac{\Lambda \cdot \alpha_1}{f} \in \Z$, since 
$\Lambda \sm_{a=1}^{f} \ffrac{a}{f}\, \Big(\ffrac{K}{a} \Big)^{-1}\!\! \in \Z[G_K]$
if and only if $\ffrac{\Lambda \cdot \alpha_\sigma}{f} \in \Z$ for all $\sigma \in G_K$,
thus, for instance, for $\sigma = 1$.
\end{proof}

\begin{proposition}\label{propII3}
(i) The ideal ${\mathfrak A}_K$ of $\Z[G_K]$ is a free $\Z$-module; a $\Z$-basis 
is given by the set $\big\{\cdots, \big(\frac{K}{a} \big) - a, \cdots; \ \Lambda_K \big\}$, 
for the representatives $a$ of $(\Z/f\Z)^\times \setminus \{1\}$.

\smallskip
(ii) If $K/\Q$ is cyclic, then ${\mathfrak A}_K$ is the ideal of $\Z[G_K]$
generated by $\big(\frac{K}{c} \big) - c$ and $\Lambda_K$, where 
$\big(\frac{K}{c} \big)$ is any generator of $G_K$.
\end{proposition}

\begin{proof}
See \cite[p. 35--36]{Gra}.
\end{proof}

\subsubsection{Study of the algebraic \texorpdfstring{$\G$}{Lg}-families 
\texorpdfstring{$\BM_K := \Z[G_K]$, $\BS_K := \BB_{K} {\mathfrak A}_K$}{Lg}} \label{ssII4c}
We then have (where $\BM_\chi$ and $\BS_\chi$ are ideals of $\BM_{K_\chi}$):
\begin{equation*}
\left\{\begin{aligned}
& \BM_{K_\chi} = \Z[G_\chi], 
&&\ \   \BS_{K_\chi} = \BB_{K_\chi} \, {\mathfrak A}_{\chi}, \\
& \BM_\chi \ \  = \{\Omega \in \Z[G_\chi],\  P_\chi (\sigma_\chi) \cdot \Omega = 0\}, 
&& \ \  \BS_\chi \ \  = \BB_{K_\chi} {\mathfrak A}_{\chi} \cap \BM_\chi
\end{aligned}\right.
\end{equation*}

\begin{lemma}\label{lemII8}
We have $\BM_\chi = \prod_{\ell \mid g^{}_\chi} \!\!(1 - \sigma_\chi^{g_\chi/\ell})\,
\Z[G_\chi]$, ${\mathfrak a}_\chi := \psi(\BM_\chi) =  \prod_{\ell \mid g^{}_\chi} 
\big (1 - \psi(\sigma_\chi)^{g_\chi/\ell} \big)$; then $\BS_\chi$ gives rise to an 
ideal ${\mathfrak b}_\chi$ multiple of ${\mathfrak a}_\chi$.
\end{lemma}

\begin{proof}
See \cite[Lemmes II.8 and II.9, pp. 37/39]{Gra}.
\end{proof}

The computation of ${\mathfrak b}_\chi$ needs to recall the norm action 
on Stickelberger's elements; because of the similarity of the result for
the norm action on cyclotomic numbers, we recall, without proof, the 
following classical formulas (see, e.g., \cite[Section 4]{Gra8}):

\begin{lemma}\label{lemII10}
Let $f  > 1$ and $m \mid f$, $m>1$, be any modulus; let $\Q(\mu_f^{})$,
$\Q(\mu_m^{}) \subseteq \Q(\mu_f^{})$, be the cor\-responding cyclotomic fields. Let:

\centerline{$\BB_{\Q(\mu_f^{})}  := - \sm_{a=1}^{f}
\Big(\ffrac{a}{f}-\ffrac{1}{2} \Big) \cdot \Big(\ffrac{\Q(\mu_f^{})}{a} \Big)^{-1}, \ \ \ \ \
\BC^{}_{\Q(\mu_f^{})}  := 1- \zeta_f$.}

\medskip\noindent
We have, where $\Norm_{\Q(\mu_f^{})/\Q(\mu_m^{})}\! : \Q[G_{\Q(\mu_f^{})}] \too \Q[G_{\Q(\mu_m^{})}]$:
\[ \Norm_{\Q(\mu_f^{})/\Q(\mu_m^{})} (\BB_{\Q(\mu_f^{})}) = 
\Omega \cdot \BB_{\Q(\mu_m^{})}, \ 
\Norm_{\Q(\mu_f^{})/\Q(\mu_m^{})} (\BC^{}_{\Q(\mu_f^{})}) = 
 \BC_{\Q(\mu_m^{})}^\Omega, \]
 
\noindent
where $\Omega := {\prod_{p \mid f,\  p \nmid m}
\big(1-\big(\ffrac{\Q(\mu_m^{})}{p} \big)^{-1}\big)}$.
\end{lemma}

We can conclude by the following \cite[Th\'eor\`emes II.5, II.6]{Gra}:

\begin{theorem} \label{theoII5}
Let $\chi \in \CX^-$ and $\psi \mid \chi$. The $\Z[\mu_{g^{}_\chi}]$-module
$\BH_\chi^\alg = \BH_\chi^\ar$ is annihilated by 
the ideal $\BB_1(\psi^{-1}) \cdot (\psi (\sigma_a) - a, \Lambda_\chi)$ of $\Z[\mu_{g^{}_\chi}]$, 
where $\sigma_a := \big(\frac{K}{a} \big)$ is any generator of $G_K$ 
(Lemma \ref{lemII7}, Proposition \ref{propII3}).

\smallskip\noindent
The ideal $(\psi(\sigma_a) - a, \Lambda_\chi)$ is the unit ideal except if $K_\chi \ne \Q(\mu_4^{})$
is an extension of $\Q(\mu_p^{})$ of $p$-power degree and if 
$\Lambda_\chi \equiv 0 \pmod p$, in which case, this ideal is a prime
ideal ${\mathfrak p}_\chi \mid p$ in $\Q(\mu_{g^{}_\chi})$.
If $K_\chi = \Q(\mu_4^{})$, this ideal is the ideal $(4)$.
\end{theorem}

\begin{theorem} \label{theoII6}
Let $\varphi \in \Phi^-$ and let $\psi \mid \varphi$. Then the $\Z_p [\mu_{g^{}_\chi}]$-module 
$\CH^\alg_\varphi = \CH^\ar_\varphi$ is annihilated by the ideal 
$\BB_1(\psi^{-1}) \cdot (\psi(\sigma_a) - a, \Lambda_\chi)$ of $\Z_p[\mu_{g^{}_\chi}]$, 
where $\sigma_a$ is any generator of $G_K$.

\smallskip\noindent
The ideal $(\psi(\sigma_a) - a, \Lambda_\chi)$ of $\Z_p [\mu_{g^{}_\chi}]$
is the unit ideal except if $K_\chi \ne \Q(\mu_4^{})$
is extension of $\Q(\mu_p^{})$ of $p$-power degree, if 
$\Lambda_\chi \equiv 0 \pmod p$ and if $\lambda = 1$ in the writing
$\psi = \omega^\lambda \cdot \psi_p^{}$ (where $\omega$ is the Teichm\"uller 
character and $\psi_p^{}$ of $p$-power order), in which case, this ideal is the 
prime ideal of $\Z_p [\mu_{g^{}_\chi}]$.

\smallskip\noindent
If $K_\chi = \Q(\mu_4^{})$, this ideal is the ideal $(4)$.
\end{theorem}

We have detailed, in Appendix \ref{ex47}, the case of $K := K_\chi = \Q(\mu_{47}^{})$ by
computing $\order \BH_\chi$ by means of the Bernoulli number with some annihilation 
properties.

\medskip
In \cite[Chap. IV, \S\,2; Th\'eor\`eme IV1]{Gra2}, \cite[Th\'eor\`emes 1, 2, 3]{Gra00},
we have given improvements of the annihilation for $2$-class groups but it is 
difficult to say if the case $p=2$ is optimal or not.
By way of example, we cite the following under the above context:

\begin{theorem} \label{annihilation2}
Let $\chi \in \CX^-$ and $\psi \mid \varphi \mid \chi$ with 
$\psi = \psi_0^{}\, \psi^{}_2$, $\psi_0^{} \ne 1$ of even order,
$\psi_2$ of $2$-power order. Put $K := K_\chi$.
The $\Z_2[\mu_{g^{}_\chi}]$-module $\CH_\varphi \big / \J_{K/K^+}
(\CH_{\varphi'}^+)$ is annihilated by $\big(\frac{1}{2 }\BB_1(\psi^{-1}) \big)$, where:
\[\CH_\varphi^+ := \{h \in \CH_{K^+},\  x^{P_{\varphi'}(\sigma_\chi) } = 1\}, \]
with $\varphi' \in \Phi^+$ above $\psi' := \psi_0\, \psi_2^2$.
\end{theorem}

\section{Application to torsion groups of abelian \texorpdfstring{$p$}{Lg}-ramification}\label{tor}

Let $K$ be a totally real number field and let $\CT_K$ be the torsion group of the Galois 
group of the maximal $p$-ramified abelian pro-$p$-extension $H_K^\pr$ of $K$. 

Under
Leopoldt's conjecture, we have $\CT_K = \Gal(H_K^\pr/K^{\,\cyc})$, where $K^{\,\cyc}$ 
is the cyclotomic $\Z_p$-extension of $K$.

\smallskip
Let $H_K^\nr$ be the $p$-Hilbert class field and let $H_K^\bp$ be the 
Bertrandias--Payan field \cite{BP}; the $\Z_p$-module $\CT_K^\bp\! := \Gal(H_K^\bp/K^\cyc)$ 
is the Bertrandias--Payan module (\cite[Sec.\,4]{Ngtor}, \cite[Sec.\,2\,(b)]{Jau5}).

\subsubsection{Schema VIII} \label{figVIII}
\unitlength=1.0cm 
\[\vbox{\hbox{\hspace{-3cm} 
 \begin{picture}(11.5,4.4)
\put(6.5,3.50){\line(1,0){1.3}}
\put(8.7,3.50){\line(1,0){2.05}}
\put(3.95,3.50){\line(1,0){1.3}}
\put(4.2,1.50){\line(1,0){1.25}}
{\color{red}\bezier{400}(3.6,3.7)(7.6,4.6)(10.9,3.7)}
\put(7.2,4.3){\ft$\CT_K$}
{\color{red}\bezier{350}(3.6,3.3)(7.3,2.4)(8.0,3.3)}
\put(6.4,2.5){\ft$\CT_K^\bp$}
\put(3.50,1.9){\line(0,1){1.25}}
\put(3.50,0.5){\line(0,1){0.8}}
\put(5.7,1.9){\line(0,1){1.25}}
{\color{red}\bezier{150}(3.8,0.3)(5.0,0.4)(5.65,1.3)}
\put(4.95,0.4){\ft$\CH_K$}
{\color{red}\bezier{400}(6.25,1.5)(9.5,1.6)(11.05,3.3)}
\put(8.95,1.7){\ft$\CU_K/\CE_K$}
\put(10.85,3.4){\ft$H_K^\pr$}
\put(5.35,3.4){\ft$K^\cyc\! H_K^\nr$}
\put(7.95,3.4){\ft$H_K^\bp$}
\put(6.7,3.2){\ft$\CR_K$}
\put(9.1,3.2){\ft$\CW_K$}
\put(4.3,3.2){\ft$\CH_K^\cyc$}
\put(3.3,3.4){\ft$K^\cyc$}
\put(5.5,1.4){\ft$H_K^\nr$}
\put(2.8,1.4){\ft$K^\cyc \!\cap \! H_K^\nr$}
\put(3.35,0.2){\ft$K$}
\end{picture}   }} \]
\unitlength=1.0cm

Let $K_v$ be the completion of $K$ at the place $v$.
The above diagram is related to the exact sequence:
\begin{equation}\label{hrw}
\begin{aligned}
1\to  \CW_K \too \  \tor_{\Z_p} & \big(\CU_K \big / \CE_K \big) \mathop {\tooo}^{\log_p}  \\
& \CR_K := \tor_{\Z_p} \big(\log_p \big (\CU_K \big) \big / \log_p (\CE_K) \big) \too 0,
\end{aligned}
\end{equation}

\noindent
where $\CW_K := \big(\oplus_{v \mid p} \mu_p(K_v^{})\big)/\mu_p(K)$,
$\CU_K$ denotes the group of local units at $p$ and $\CE_K = \BE_K \otimes \Z_p$ 
is identified with its diagonal image in $\CU_K$ 
(see \cite[ \S\,III.2, (c), Fig.\,2.2; Lemma III.4.2.4]{Gra6} and  \cite{Gra7}).

\smallskip
Since $[\Q_p(\mu_{p^e}) : \Q_p] = (p-1) p^{e-1}$, for $K$ fixed there are 
only finite number of primes $p$ such that $\CW_K \ne 1$; for $K$ totally real
$\mu_p(K) = 1$ for all $p >2$. For instance, if $K=\Q(\sqrt m)$ is a real quadratic 
field, then for $p=2$, $\CW_K \simeq \mu_2^{} \times \mu_2^{}/\mu_2^{}$
($2$ split in $K$) or $\mu_4^{}/\mu_2^{}$ ($m \equiv -1 \pmod 8$); for $p=3$, 
$\CW_K \simeq \mu_3^{}$ if and only if $m \equiv -3 \pmod 9$.

\smallskip
In all the sequel, we assume that $K$ is abelian real.

\subsection{Computation of \texorpdfstring{$\order \CT_K$}{Lg} for 
\texorpdfstring{$\chi \in \CX^+$}{Lg}}
The order of the $\Z_p[G_K]$-module $\CT_K$ is given, analytically, by 
the residue at $s=1$ of the $p$-adic $\zeta$-function of $K$, whence by the 
values at $s=1$ of $p$-adic $\BL$-functions of the non-trivial characters of $K$ 
(after \cite[Appendix]{2Coa}); see for instance \cite[\S\,3.4, formula (3.8)]{Gra9} 
for analytic context. 

\smallskip
In conclusion we can write, up to $p$-adic units:
\begin{equation}\label{torsion}
\order\CT_K  = \order \CH^{\,\cyc}_K\!\! \cdot \order{\mathcal R}_K 
\cdot \order{\mathcal W}_K  \sim  [K \cap \, \Q^{\,\cyc} \!: \Q] \cdot
\!\prd_{\psi \ne 1} \hbox{$\frac{1}{2}$}\,\BL_p (1,\psi).
\end{equation}

Since the arithmetic family of these $\Z_p[\G]$-modules $\CT_K$,
for real fields $K$, follows the most favorable properties (surjectivity 
of the norms, injectivity of the transfer maps in relative sub-extensions), 
we can state, in a similar context as for Theorems \ref{theoII1}:

\begin{theorem}\label{theotorsion}
For all $\chi \in \CX^+$ (resp. $\varphi \in \Phi^+$, $\varphi \mid \chi$), we have:
\begin{equation*}
\left\{\begin{aligned}
\CT^\ar_\chi = \CT^\alg_\chi & = \{x \in \CT_{K_\chi}, \ \, x^{P_\chi (\sigma_\chi)} = 1 \} \\
& = \{x \in \CT_{K_\chi},\  \Norm_{K_\chi/k}(x) = 1,\ 
\hbox{for all $k \varsubsetneqq K_\chi$} \}, \\
\CT^\ar_\varphi  = \CT^\alg_\varphi & = \{x \in \CT_{K_\chi}, \ \, x^{P_\varphi(\sigma_\chi) } = 1 \}.
\end{aligned}\right.
\end{equation*}

\noindent
Moreover, if $K/\Q$ is real cyclic,
$\order \CT_K = \prd_{\chi \in \CX_K} \order \CT^\ar_\chi = 
\prd_{\varphi \in \Phi_K} \order \CT^\ar_\varphi$.  
\end{theorem}

We denote simply $\CT_\chi$ (resp. $\CT_\varphi$) these components
in the algebraic and arithmetic senses. In the analytic point of view, we have the analogue of 
Theorems \ref{theoII2} and \ref{theoIII2} (see some $p$-adic formulas about $\BL_p$-functions,
from classical papers \cite{KL, AF, Gra3} and a broad presentation in 
\cite[Theorems 5.18, 5.24]{Was}):

\begin{theorem} \label{theotorsion2}
Let $\chi \in \CX^+ \setminus \{1\}$. Then $\order \CT_\chi \sim  w_\chi^{\,\cyc} \cdot 
\prd_{\psi \mid \chi} \hbox{$\frac{1}{2}$} \,\BL_p (1,\psi)$, 
where $w_\chi^{\,\cyc}$ is as follows, from analytic formula \eqref{torsion}:

\smallskip
\quad (i) $w_\chi^{\,\cyc} = 1$ if $K_\chi$ is not a subfield of $\Q^{\,\cyc}$;

\smallskip
\quad (ii) $w_\chi^{\,\cyc} = p$ if $K_\chi$ is a subfield of $\Q^{\,\cyc}$.
\end{theorem}

\subsection{Annihilation theorem for \texorpdfstring{$\CT_K$}{Lg}}\label{annTphi}
An annihilator of $\CT_K$ is given by the following statement 
\cite[Theorem 5.5]{Gra8} which does not assume any hypothesis 
on $K$ real and $p$ and gives again the following results (e.g., \cite{Gra4}, \cite{Or1}):

\begin{theorem}
Let $K$ be a real abelian field of conductor $f_K$.
Let $f_n$ be the conductor of 
$L_n := K \Q(\mu_{qp^n})$, $n$ large enough, where $q=p$ or $4$ as usual.
Let $c \in \Z$ be prime to $2p f_K$. 
For all $a \in [1, f_n]$, prime to $f_n$, let $a'_{c} \in [1, f_n]$ be the 
unique integer such that $a'_{c} \cdot c \equiv a \pmod {f_n}$ and put
$a'_{c} \cdot  c - a = \lambda^n_a(c)\, f_n$, $\lambda^n_a(c) \in \Z$. 
Then consider:
\[\BA_{K,n}(c) := \!\sm_{a=1}^{f_n} \lambda^n_a(c) \, 
a^{-1} \Big(\ffrac{K}{a} \Big)=: \BA'_{K,n}(c) \cdot (1+s_\infty) \in \Z_p[G_K], \]
where $s_\infty$ is the complex conjugation and
$\BA'_{K,n}(c) =\! \!\sm_{a=1}^{f_n/2} \lambda^n_a(c) \, 
a^{-1} \Big(\ffrac{K}{a} \Big)$.

\noindent
Let $\BA_K(c) :=  \ds\lim_{n \to \infty}\ \Big[\ 
\hbox{$\!\sm_{a=1}^{f_n}$} \lambda^n_a(c) \, a^{-1} \Big(\ffrac{K}{a} \Big)\Big]
=: \BA'_K(c) \cdot (1+s_\infty)$; then:

\smallskip
(i) For $p\ne 2$, $\BA'_K(c)$ annihilates the $\Z_p[G_K]$-module ${\mathcal T}_K$. 

\smallskip
(ii) For $p=2$, the annihilation is true for $2 \cdot \BA_K(c)$ and $4 \cdot \BA'_K(c)$.
\end{theorem}

It is immediate, using these formulas modulo a suitable power of $p$, to compute
annihilators; examples are given in Appendix \ref{annihilators}.

\begin{remarks}\label{fctL}
(i) In practice, when the exponent $p^e$ of $\CT_K$ is known, one can take 
$n=n_0+e$, where $n_0\geq 0$ is defined by $[K \cap \Q^{\,\cyc} : \Q] =: p^{n_0}$,
and use the annihilators $\BA_{K,n}(c)$, $\BA'_{K,n}(c)$ (but any $n \gg 0$ is suitable).
When $K=K_\chi$, the annihilator limit $\BA_{K_\chi}(c)$ is
related to $p$-adic $\BL$-functions via the formula:
\[\psi(\BA_{K_\chi}(c)) =  (1-\psi (c)) \cdot \BL_p(1,\psi),\ \hbox{ for $\psi \mid \chi$}. \]
If $g_\chi$ is not a $p$-power, one can choose $c$ such that
$1-\psi (c)$ is invertible giving $\psi(\BA_{K_\chi}(c)) \sim \BL_p(1,\psi)$;
if $g_\chi = p^n$, $n \geq 1$, $\psi(\BA_{K_\chi}(c)) \sim 
\pi_\chi \BL_p(1,\psi)$, where $\pi_\chi$ is an uniformizing parameter in 
$\Q_p(\mu_{p^n})$.

\smallskip
This theorem is the analog of Theorem \ref{theoII6}, using
Bernoulli's numbers, linked to $\BL_p (0,\omega \psi^{-1})$, instead of 
$\BL_p (1,\psi)$.

\smallskip
(ii) Some other annihilation theorems exist for the Jaulent logarithmic class group
(see \cite{Jau6, Jau7, Jau8}); \cite{Jau8} is related to Greenberg's conjecture and,
when $K$ contains $\mu_p^{}$, \cite{Jau6} obtains that the Stickelberger ideal annihilates 
the imaginary component of the logarithmic class group and that its reflection annihilates 
the real component of the Bertrandias--Payan module.
It will be interesting to formulate a ``Finite AMC'' about the $\varphi$-components 
of these modules.
\end{remarks}

\section{Application to class groups of real abelian extensions}\label{secIII}

Denote by $\BE$ the $\G$-family for which $\BE_K$, $K \in \CK$, is the group 
of absolute value of the global units of $K$, the Galois action being defined by 
$\vert \varepsilon \vert^\sigma = \vert \varepsilon^\sigma \vert$ for any unit 
$\varepsilon$ and any $\sigma \in \G$. As we explain in the beginning of the 
Appendix for explicit computations, conjugates of algebraic numbers are 
managed by PARI in a coherent manner corresponding to an (unknown) 
embedding of $\ov\Q$ in $\C$; thus $\vert \ \,\vert$ is, for us, the real absolute 
value, taken after a fixed embedding $K \to \R$, or after PARI numerical results. 

\smallskip
The $\BE_K$'s are free $\Z$-modules of rank $[K : \Q]-1$ for real fields $K$.

\subsection{The Leopoldt \texorpdfstring{$\chi$}{Lg}-units}
In \cite{15Leo} Leopoldt defined unit groups, $\BE_\chi$, that we shall call (as in 
\cite{18Or}) the group of  $\chi$-units for rational characters $\chi \in \CX^+ \setminus \{1\}$;
from the definition of $\chi$-objects and the results of the previous sections we can 
write (where $\Nu$ may be replaced by $\Norm$):
\begin{equation}\label{chiunits}
\begin{aligned}
\BE_\chi & = \{\vert \varepsilon \vert \in \BE_{K_\chi},\ \, 
\vert \varepsilon \vert^{P_\chi (\sigma_\chi) } = 1 \} \\
& = \{\vert \varepsilon \vert \in \BE_{K_\chi}, \,
\Nu_{K_\chi/k} (\vert \varepsilon \vert) = 1,\,  \hbox{for all 
$k \varsubsetneqq K_\chi$} \}.
\end{aligned}
\end{equation}

What follows is also available in \cite{15Leo, 16Leo, 18Or}.

\begin{definitions}\label{E0}
(i) For any cyclic real field $K$, denote by $\wh \BE_K$ the subgroup of $\BE_K$ 
generated by the $\BE_k$'s for all the subfields $k \varsubsetneqq K$ (or simply by 
each of the $k_\ell^{}$ such that $[K : k_\ell^{}] = \ell \mid [K : \Q]$, $\ell$ prime).

\smallskip
(ii) Let $Q_K = \big(\BE_K : \oplus_{\chi \in  \CX_K} \BE_\chi \big)$ where $\BE_\chi$ is 
the group of $\chi$-units (Definition \eqref{chiunits}) and, for all $\chi \in \CX_K^+$,
let $Q_\chi = \big( \BE_{K_\chi} :  \wh \BE_{K_\chi} \!\oplus \BE_\chi \big)$.

\smallskip
(iii) Let $\phi$ be the Euler totient function and put, for $\chi \in \CX^+$:
\begin{equation*}
\left\{\begin{aligned}
q_\chi & = \hbox{$\prod_{\ell \mid g^{}_\chi}$} \ell^{\frac{\phi(g^{}_\chi)}{\ell-1}}, 
 \hbox{ if $g^{}_\chi$ is not the power of a prime number,} \\
q_\chi & = \ell^{\frac{\phi(g^{}_\chi)}{\ell-1}-1} = \ell^{\ell^{n-1} -1}, 
 \hbox{ if $g^{}_\chi$ is a prime power $\ell^n$, $n \geq 1$,}\\
q_1^{} & =1.
\end{aligned}\right.
\end{equation*}

Set $q_K^{} = \Big( \ffrac{g^{g-2}}
{\prod_{\chi \in \CX_K} d_\chi}  \Big)^{\frac{1}{2}}$, where $g := [K : \Q]$ 
and $d_\chi$ is the discriminant of $\Q(\mu_{g^{}_\chi})$.
\end{definitions}

\begin{lemma} \label{lemIII1}\label{propIII1}\label{propIII2}
(i) We have $\wh \BE_{K_\chi} \!\!\cdot \BE_\chi = \wh \BE_{K_\chi} \! \oplus \BE_\chi$,
for all $\chi \in \CX^+$.

\smallskip
(ii) We have, for all cyclic real field $K$, $Q_K = \prod_{\chi \in  \CX_K} Q_\chi$.

\smallskip
(iii) We have, for all cyclic real field $K$, $q_K^{} = \prod_{\chi \in  \CX_K} q_\chi$.
\end{lemma}

\begin{proof}
(i) One may find various equivalent definitions of the $\chi$-units and their 
properties in \cite[Chap. 5, \S\,4]{15Leo} or \cite{18Or}; but knowing the norm 
characterization \eqref{chiunits} of $\BE_\chi$, the proof of (i) is obvious.

\smallskip
(ii) This may be proved locally; for this, we use the $\G$-family
$\CE_K := \BE_K \otimes \Z_p$, for any prime $p$, and the $\CE_\chi$'s as above. 
Then one uses, inductively, Lemma \ref{lemIII1}\,(i) with characters $\psi \mid \varphi \mid \chi$,
written as $\psi = \psi^{}_0 \, \psi_p^{}$ ($\psi^{}_0$ of prime-to-$p$ order, $\psi_p^{}$ 
of order $p^n$, $n \geq 0$). See \cite[pp. 72--75]{Gra}.

\smallskip
(iii) From \cite[\S\,15, p. 34; (2), p. 35]{10Has};
see \cite[pp. 76--77]{Gra} for more details.
\end{proof}

\subsection{The Leopoldt cyclotomic units}\label{subIII2} 

For the main definitions and properties of cyclotomic units, see 
\cite[\S\,8\ (1)]{15Leo} or \cite{17Or}.

\begin{definitions}\label{defIII3} 
(i) Let $\chi \in \CX^+$ of conductor $f_\chi$; we define the ``cyclotomic numbers''
$\BC_\chi := \prod_{a \in A_\chi} (\zeta_{2f_\chi}^a - \zeta_{2f_\chi}^{-a})$, with
$\zeta_{2f_\chi} := \exp \big(\ffrac{i \pi}{f_\chi} \big)$, where $A_\chi$ is a 
half-system of representatives, in $(\Z/ f_\chi \Z)^\times$, of
$\Gal(\Q(\mu_{f_\chi^{}})/K_\chi )$.

\smallskip
(ii) Let $K$ be a real abelian field and let $\BC_K$ be the multiplicative group
generated by the conjugates of $\vert \BC_\chi \vert$, for all $\chi \in \CX_K$. 
Then we define the group of cyclotomic units
$\BF_K := \BC_K \cap \BE_K$ and  $\CF_K := \BF_K \otimes \Z_p$. 
\end{definitions}

Recall that $\BC_\chi^2 \in K_\chi$ and that any conjugate $\BC'_\chi$ 
of $\BC_\chi$ is such that $\frac{\BC'_\chi}{\BC_\chi} \in \BE_{K_\chi}$.
If $f_\chi$ is not a prime power, then $\BC_\chi$ is a unit and $\BF_K = \BC_K$.

\subsection{Arithmetic computation of \texorpdfstring{$\order \BH^\ar_\chi$}{Lg},
\texorpdfstring{$\chi \in \CX^+$}{Lg}}\label{subIII3}

Using Leopoldt's formula \cite[Satz 21, \S\,8\ (4)]{15Leo} and  
Lemma \ref{propIII1}\,(ii),\,(iii), we obtain (see \cite[Th\'eor\`eme III.1]{Gra}):

\begin{proposition}\label{theoIII1}\label{div}
For all $\chi \in \CX^+ \setminus \{1\}$, let $\Delta_\chi = 
\prod_{\ell \mid g^{}_\chi} \big(1 - \sigma_\chi^{g_\chi/\ell} \big)$;
then $\order \BH^\ar_\chi = \ffrac{Q_\chi}{q_\chi} 
\cdot (\BE_\chi : \BC_\chi^{\Delta_\chi})$ and
 $\order \BH^\ar_\chi = \ffrac{1}{q_\chi} \big (\BE_{K_\chi} :  \wh \BE_{K_\chi} 
\! \oplus \BC_\chi^{\Delta_\chi} \big)$, interpreting $Q_\chi$ \cite[Corollaire III.1]{Gra}.
\end{proposition}

To interpret the coefficient $q_\chi$, we have replaced the Leopoldt group 
$\BC_\chi^{\Delta_\chi}$ of cyclotomic units by the larger group $\BF_{K_\chi} 
:= \BC_{K_\chi} \cap \BE_{K_\chi}$ (Definition~\ref{defIII3}); whence  the final 
result interpreting the coefficient 
$q_\chi$ and giving the analog of Theorem \ref{theoII2} for real class groups:

\begin{theorem}\label{theoIII2}
Let $\BH^\ar_\chi := \{x \in \BH_{K_\chi},\,  
\Norm_{K_\chi/k}(x) = 1,\ \hbox{for all $k \varsubsetneqq K_\chi$} \}$.
Let $g^{}_\chi$ be the order of $\chi  \in \CX^+ \setminus \{1\}$ and 
$f_\chi$ its conductor. Then:
\[\order \BH^\ar_\chi = w_\chi \! \cdot \big (\BE_{K_\chi}\! : \wh \BE_{K_\chi}\! \!
\cdot \BF_{K_\chi} \big), \]
where $w_\chi$ is defined as follows:

\medskip
(i) Case $g^{}_\chi$ non prime power. Then $w_\chi = 1$;

\medskip
(ii) Case $g^{}_\chi = p^n$, $p \ne 2$ prime, $n \geq 1$:

\smallskip
\qquad (ii$\,'$) Case $f_\chi = \ell^k$, $\ell$ prime, $k \geq 1$. Then 
$w_\chi = 1$;

\smallskip
\qquad (ii$\,''$) Case $f_\chi$ non prime power. Then $w_\chi = p$;

\medskip
(iii) Case $g^{}_\chi = 2^n$, $n \geq 1$:

\smallskip
\qquad (iii$\,'$) Case $f_\chi = \ell^k$, $\ell$ prime, $k \geq 1$. Then $w_\chi = 1$;

\smallskip
\qquad (iii$\,''$) Case $f_\chi$ non prime power. Then $w_\chi \in\{1, 2\}$.
\end{theorem}

\begin{proof}
For the ugly proof see \cite[Th\'eor\`eme III.2, pp. 78--85]{Gra}.
\end{proof}

\begin{corollary}\label{chisemisimple}
If $p \nmid g_\chi$, $\order \CH_\chi = \big (\CE_\chi : \CF_\chi \big) =
\prod_{\varphi \mid \chi} \big (\CE_\varphi : \CF_\varphi \big)$, where
$\CE_\varphi = \CE_{K_\chi}^{e_\varphi}$ and $\CF_{\!\varphi} =
\big(\langle \BC_\chi \rangle \otimes \Z_p\big)^{e_\varphi}$ now
giving $\order \CH_\varphi = \big (\CE_\varphi : \CF_\varphi \big)$.
\end{corollary}

\begin{proof}
In the semi-simple case $p \nmid g_\chi$, for any $\Z_p[G_K]$-module $\CM_K$, 
$\CM_\chi = \CM_K^{e_\chi}$ and $\CM_\varphi = \CM_K^{e_\varphi}$, 
with the usual semi-simple idempotents; thus,
$\wt \CE_\chi = \wt \CE_\chi^{e_\chi} = 
\CE_{K_\chi}^{e_\chi} / \wh \CE_{K_\chi}^{e_\chi} \cdot \CF_{\!K_\chi}^{e_\chi}
= \CE_\chi / \CF_\chi$, since $\wh \CE_{K_\chi}^{e_\chi}=1$. The claim for
$\varphi \mid \chi$ is the Main Theorem proved in the semi-simple context.
\end{proof}

\begin{remarks}\label{e0}
The viewpoint given by Theorem \ref{theoIII2}, 
which appears to have been ignored, seems more 
convenient than formulas trying to use Sinnott's cyclotomic units. 
Indeed, compare with \cite[Theorem 4.14]{Grei} using instead $\CH^\alg_{\chi}$ 
(in a partial semi-simple context as explained in Remark \ref{remgreither})
and Sinnott's group of cyclotomic units, larger than classical Leopoldt's group
of Definition \ref{defIII3}, but which gives rise to intricate index formulas. 
For the Iwasawa context, see for instance \cite{NLB}.

\smallskip
Moreover, as we have mentioned in \cite[Remark III.1]{Gra0}, an analytic formula for 
$\order \CH^\alg_\chi$, $\chi \in \CX^+$, does not seem obvious (if any) because of 
capitulation aspects (see the examples of Appendix \ref{ex12}).

\smallskip
Theorem \ref{theoIII2} suggests a new and simpler statement of the 
Finite AMC for the $\CH_\varphi$'s, especially in the non semi-simple 
real case (see \S\,\ref{analytics} for the corresponding analytic values).
Recent publications \cite{Gra13,Gra14,Gra15} greatly strengthen this definition
of the Finite AMC, using the $\chi$-objects 
$\wt \BE_\chi := \BE_{K_\chi} / \wh \BE_{K_\chi}\! \!\cdot \BF_{K_\chi}$
and $\wt \CE_\chi := \CE_{K_\chi} / \wh \CE_{K_\chi} \!\!\cdot \CF_{\!K_\chi}$
(algebraic and arithmetic).
Then:
\[\wt \CE_\chi = \plus_{\varphi \mid \chi} \wt \CE_\varphi
=  \plus_{\varphi \mid \chi} \{\wt x \in \wt \CE_\chi, \ \, 
\wt x^{P_\varphi (\sigma_\chi) } = 1\} =  \plus_{\varphi \mid \chi} 
(\wt \CE_\chi)_{\varphi_0^{}}. \]
\end{remarks}

\subsection{Class field theory and regulators}\label{diagram}
Let $K \in \CK$ be a real cyclic field defining $\chi \in \CX^+$ in what follows.
To simplify some diagrams, we assume to be in the most common case where 
$\CW_K=1$ and $K\cap \Q^\cyc = \Q$, which gives $\CT_K = \CT_K^\bp$ (cf. 
Diagram of Section \ref{tor}) and $\order \CT_K \sim \prod_{\psi \mid \chi,\,\psi \ne 1} 
\hbox{$\frac{1}{2}$} \, \BL_p (1,\psi)$ (Formula \eqref{torsion}). Otherwise, formulas 
are modified by means of standard coefficients or indices which do not modify the 
philosophy of the results/conjectures; moreover the character of $\CW_K$, related to 
local cyclotomic Teichm\"uller ones, gives trivial information for conjectural aspects.

\smallskip
The Galois group $\CR_K \subseteq \CT_K$ may be compared with a larger 
``cyclotomic regulator'' $\CR_K^\cyc$ interpreted as a Galois group only 
depending of $\chi$. For this purpose, the following diagram of the maximal abelian 
pro-$p$-extension $K^\ab$ of $K$ is necessary (from \cite[III.4\,(d) \& Diagram 
III.4.4.1]{Gra6} with our present notations),  where $H_K^\ta$ is the maximal 
tamely ramified abelian pro-$p$-extension of $K$ and $F_v^\times$ the 
$p$-Sylow subgroup of the multiplicative group of the residue field of the 
tame place~$v$; let $L := H_K^\pr H_K^\ta$:
\subsubsection{Schema IX} \label{figIX}
\unitlength=0.75cm
 {\color{black}
\[\vbox{\hbox{\vspace{-1.2cm}
\begin{picture}(11.5,5.9) 
\put(8.5,4.50){\line(1,0){3.0}}
\put(1.4,4.50){\line(1,0){4.4}}
\put(1.4,2.50){\line(1,0){4.4}}
\put(6.3,4.5){\line(1,0){1.55}}
\put(1.0,2.9){\line(0,1){1.20}}
\put(1.0,1.4){\line(0,1){0.70}}
\put(6.15,2.9){\line(0,1){1.20}}
{\color{red}\bezier{500}(1.2,4.9)(6.45,6.1)(11.7,4.9)}
\put(6.1,5.68){\ft${\prod_{v \nmid p}{F_v^\times}}$}
{\color{red}\bezier{400}(6.6,2.48)(11.8,2.8)(11.9,4.3)}
\put(10.75,2.85){\ft$\CU_K$}
{\color{red}\bezier{150}(6.15,4.3)(7.1,4.0)(8.05,4.3)}
\put(6.9,3.75){\ft$\wt \CE_\chi$}
{\color{red}\bezier{150}(8.50,4.3)(10.0,4.0)(11.6,4.3)}
\put(9.65,3.8){\ft$\wh \CE_K \! \CF_K$}
\put(11.7,4.4){\ft$K^\ab$}
\put(5.9,4.4){\ft$L$}
\put(7.8,4.4){\ft$L{\sst (\chi)}$}
{\color{red}\bezier{250}(6.2,4.7)(8.9,5.2)(11.65,4.7)}
\put(8.75,5.05){\ft$\CE_K$}
{\color{red}\bezier{150}(0.7,2.5)(0.1,3.45)(0.7,4.4)}
\put(-0.78,3.45){\ft$\CU_K \! / \!\CE_K$}
{\color{red}\bezier{200}(1.2,1.2)(2.2,2.8)(1.25,4.4)}
\put(1.75,2.8){\ft$\CT_K$}
\put(0.75,4.4){\ft$H_K^\pr$}
\put(5.9,2.4){\ft$H_K^\ta$}
\put(0.75,2.4){\ft$H_K^\nr$}
\put(0.8,1.1){\ft$K$}
\end{picture}   }} \]}
\unitlength=1.0cm

In this diagram, class field theory interprets $\Gal(K^\ab/H_K^\ta)$ 
as the $\Z_p$-module $\CU_K$ of principal local units at $p$ 
(isomorphic to the direct product of the inertia groups of the $p$-places) 
and $\Gal(K^\ab/L)$ as the $\Z_p$-module $\CE_K := \BE_K \otimes \Z_p$ 
(embedded both in $\CU_K$ and the product ${\prod_{v \nmid p}{F_v^\times}}$
of the inertia groups of the tame places, with suitable Artin maps
described in \cite[\S\,III.4.4.5.1]{Gra6}).

\smallskip
Now, put $\CU_K^* := \{u \in \CU_K, \ \Norm_{K/\Q}(u) = \pm1\}$; since $K$ is real,
$\CE_K$ is of finite index in $\CU_K^*$ and 
$\tor_{\Z_p}^{}(\CU_K/\CE_K) = \CU_K^*/\CE_K \simeq \CR_K$.

Assume $K^\cyc \cap H_K^\nr = K$ to simplify; so $H_K^\ta \cap K^\cyc = H_K^\nr$
then $F := H_K^\ta \, K^\cyc  H_K^\nr$ is fixed by $\CU_K^*$ and  
$F \cap H_K^\pr = K^\cyc H_K^\nr$.
Recall the exact sequence $1\to \CR^\ram_K \to \CR_K \to \CR^\nr_K \to 1$
\cite[\S\,2 \& Figure 3]{Gra10}, due to genus theory; so, a sub-extension of 
$L/F$ may be unramified. 

\smallskip
We have moreover
$\Gal(F/K^\cyc H_K^\nr) \simeq \Gal(L/H_K^\pr) \simeq 
\big(\hbox{$\prod_{v \nmid p}$} {F_v^\times} \big)/\CE_K$:

\subsubsection{Schema X} \label{figX}
\unitlength=1.16cm
 {\color{black}
\[\vbox{\hbox{\hspace{-4.0cm} \vspace{-1.5cm}
 \begin{picture}(9.5,5.2)
\put(4.65,3.0){\line(1,0){2.6}}
\put(1.9,3.0){\line(1,0){1.7}}
\put(1.7,1.5){\line(1,0){2.1}}
\put(4.9,4.05){\line(1,0){3.15}}
\put(8.3,4.05){\line(1,0){1.54}}
\put(10.3,4.05){\line(1,0){1.05}}
\put(1.5,1.65){\line(0,1){1.15}}
\put(4.0,1.65){\line(0,1){1.15}}
\put(4.75,2.75){\line(0,1){1.2}}
\put(1.55,2.2){\ft$\Z_p$}
\put(1.4,2.9){\ft$K^\cyc$}
{\color{red}\bezier{200}(1.65,3.1)(2.7,3.35)(3.85,3.1)}
\put(2.5,3.3){\ft$\CH_K$}
\put(3.6,2.9){\ft$K^\cyc\! H_K^\nr$}
\put(7.3,2.9){\ft$H_K^\pr$}
\put(4.6,4.0){\ft$F$}
\put(8.1,4.0){\ft$L$}
\put(9.8,4.0){\ft$L{\sst (\chi)}$}
\put(11.35,4.0){\ft$K^\ab$}
{\color{red}\bezier{450}(7.7,2.97)(10.8,3.0)(11.35,3.9)}
\put(9.6,2.9){\ft${\prod_{v \nmid p} \!{F_v^\times}}$}
\put(1.4,1.4){\ft$K$}
\put(3.85,1.4){\ft$H_K^\nr$}
\put(4.55,2.55){\ft$H_K^\ta$}
{\color{red}\bezier{500}(4.75,4.35)(8.1,5.55)(11.4,4.35)}
\put(8.0,5.04){\ft$\CU_K^*$}
{\color{red}\bezier{450}(4.8,4.2)(6.5,4.45)(8.1,4.2)}
\put(5.9,4.44){\ft$\CR_K\! \simeq\! \CU_K^*/ \CE_K$}
{\color{red}\bezier{300}(4.2,3.1)(5.7,3.35)(7.3,3.1)}
\put(5.55,3.3){\ft$\CR_K$}
{\color{red}\bezier{500}(1.6,3.2)(4.45,4.3)(7.3,3.2)}
\put(3.5,3.78){\ft$\CT_K$}
{\color{red}\bezier{250}(8.3,3.95)(10.2,3.6)(11.35,3.95)}
\put(9.96,3.6){\ft$\CE_K$}
{\color{red}\bezier{150}(8.3,4.2)(9.05,4.35)(9.9,4.2)}
\put(8.95,4.35){\ft$\wt \CE_\chi$}
{\color{red}\bezier{150}(9.98,4.2)(10.65,4.35)(11.4,4.2)}
\put(10.3,4.38){\ft$\wh \CE_K \! \CF_K$}
{\color{red}\bezier{500}(4.8,3.95)(7.3,3.65)(9.8,3.95)}
\put(7.0,3.55){\ft$\CR_K^\cyc$}
{\color{red}\bezier{550}(4.7,2.44)(11.0,1.0)(11.4,3.95)}
\put(9.2,1.85){\ft$\CU_K$}
{\color{red}\bezier{550}(4.7,2.47)(9.0,1.8)(8.3,3.9)}
\put(7.95,2.45){\ft$\CU_K/\CE_k$}
{\color{red}\bezier{550}(4.2,1.47)(8.2,0.8)(7.4,2.85)}
\put(7.1,1.36){\ft$\CU_K/\CE_K$}
\bezier{350}(4.1,3.15)(4.35,3.55)(4.6,3.95)
\bezier{350}(7.6,3.15)(7.85,3.55)(8.1,3.95)
\bezier{350}(4.1,1.65)(4.35,2.05)(4.6,2.45)
\end{picture}   }} \]}
\unitlength=1.0cm

Define (under the previous assumptions), $\CR_K^\cyc\! := \CU_K^*/\wh \CE_K\!\cdot \CF_K$, 
which yields, for $\chi \ne 1$ and $K= K_\chi$, the $\Z_p[G_K]$-modules isomorphism: 
\begin{equation}\label{R/E}
\CR_K \simeq \CR_K^\cyc / \wt \CE_\chi.
\end{equation}

We then have $\CR_K^\cyc \simeq \Gal(L{\sst (\chi)}/F)$, where
$L{\sst (\chi)}$ is the subfield of $K^\ab$ fixed by the image of $\wh \CE_K \CF_K$.

\begin{remark} \label{rankR}
Let $\chi \in \CX^+ \setminus \{1\}$, $K = K_\chi$; assume to simplify that $\CW_K=1$, 
$w_\chi=1$ in Theorem \ref{theoIII2}, $K\cap \Q^\cyc = \Q$ and $K^\cyc \cap H_K^\nr = K$:

\smallskip
(i) Theorem \ref{theoIII2} and isomorphism \eqref{R/E} give, in terms of $\chi$-compo\-nents:
\[\hbox{$\order \wt \CE_\chi = \order \CR_K^\cyc \big / \order \CR_K =
\order \CH^\ar_\chi$ and $\order \CT_\chi  = \order \CR_\chi^\cyc$. }\]
The $\CA_\chi$-modules $\CT_\chi$ and  $\CR_\chi^\cyc$
(resp. $\wt \CE_\chi$ and $\CH^\ar_\chi$) are not necessarily isomorphic as shown 
by the following excerpt giving cyclic cubic fields $K$ such that $\CR_\chi$ is of 
$7$-rank $2$ and $\CT_\chi$ of $7$-rank $\geq 3$ implying $\CH_\chi \ne 1$ 
with $\CH_\chi \simeq \Z/7\Z \oplus \Z/7\Z$ for the followings
(no example of $7$-rank $\geq 4$ exists in the interval considered):

\ft\begin{verbatim}
x^3+x^2-39666*x-2582719     Structure of the 7-torsion group:[7,7,7]
x^3+x^2-43300*x-3411104     Structure of the 7-torsion group:[7^2,7,7]
x^3+x^2-13226*x-508479      Structure of the 7-torsion group:[7^3,7,7]
x^3+x^2-427660*x-31551829   Structure of the 7-torsion group:[7^4,7,7]
x^3+x^2-2033484*x-966131001 Structure of the 7-torsion group:[7^2,7^2,7]
\end{verbatim}\ns

\smallskip
(ii) The sub-diagram given by the extension $K^\ab/K^\cyc$, opens  
an access way for an interpretation of the Finite AMC for even characters 
or at least for an annihilation theorem of $\CH^\ar_\varphi$ by $\wt \CE_\varphi$,
in the spirit of Thaine's theorem (see \S\,\ref{mysterious}, Conjectures \ref{annihilationthm}, 
\ref{omegaconj}). 
Indeed, $\wt \CE_\chi$ has same order as $\CH^\ar_\chi$ and the units may be seen 
diagonally embedded in the (infinite) product of the places of $K$. 
Remark that $\wt \CE_\varphi$ is a sub-module of $\CR_\varphi^\cyc$ 
(quotient $\CR_\varphi$) but $\CH^\ar_\varphi$ is a quotient of 
$\CT_\varphi$ (by $\CR_\varphi$).
\end{remark}

\subsection{Annihilation conjecture for real \texorpdfstring{$p$}{Lg}-class groups}\label{annihilation+}
Before any proof of the conjectural equality $\order \CH^\ar_\varphi = \order \wt \CE_{\varphi_0^{}} =
\order (\CE_{K_\chi} / \wh \CE_{K_\chi}\!\! \cdot\CF_{\!K_\chi})_{\varphi_0^{}}$
(giving a Main Theorem for $\varphi \in \Phi_K^+$), it will be interesting to prove that any 
annihilator of $\wt \CE_\varphi$ annihilates $\CH^\ar_\varphi$, which will be more 
precise than the annihilators of $\CT_\varphi$ (see Theorem \ref{annTphi}, 
Remarks \ref{fctL}, \ref{rankR}).

\smallskip
To our knowledge, the best known annihilation theorem of real
$p$-class groups is Thaine's Theorem \cite{Th}, \cite[Theorem 15.2]{Was}
saying that any annihilator of $\CE_{K_\chi}/\CF'_{\!K_\chi}$ (for a suitable definition of
the group of cyclotomic units $\CF'_{\!K_\chi}$) is an annihilator of $\CH_{K_\chi}$. 
But Thaine's Theorem only concerns the semi-simple case.

\smallskip
Mention also annihilation theorems by Solomon \cite{Sol1}, which are not often
optimal because of vanishing of Euler factors; this is discussed in \cite{Gra8}.
Finally mention the numerous papers of Greither and Ku\v cera (like \cite{GK0, GK1, GK2}) 
on the annihilation of real class groups, using special units or/and giving information on
the Fitting ideals.

\begin{conjecture} \label{annihilationthm}
Let $\chi \in \CX^+ \setminus \{1\}$ and let $\varphi \mid \chi$. 
Any element of $\Z[\mu_{g_\chi^{}}]$ (resp. $\Z_p[\mu_{g_\chi^{}}]$) annihilating 
$\BE_{K_\chi} / \wh \BE_{K_\chi}\!\! \cdot\BF_{K_\chi}$
(resp. $(\CE_{K_\chi} / \wh \CE_{K_\chi}\!\! \cdot\CF_{K_\chi})_{\varphi_0^{}}$), 
annihilates $\BH^\ar_\chi$ (resp. $\CH^\ar_\varphi$). 
\end{conjecture}

In this direction, we state the following lemma, giving some 
obvious prerequisites on the subject.

\begin{lemma}\label{mono}
Let $\BM_{K_\chi}$ be a torsion-free monogenic $\Z[G_\chi]$-module (i.e., $\Z$-free  
and $\Z[G_\chi]$-generated by a single element). Let $\BM'_{K_\chi}$ be a sub-module 
of $\BM_{K_\chi}$ such that $\BM_{K_\chi}/\BM'_{K_\chi}$ is annihilated by 
$P_\chi(\sigma_\chi)\,\Z[G_\chi]$ and finite. 
Then $(\CM_{K_\chi}/\CM'_{K_\chi})_\varphi \! := \! ((\BM_{K_\chi}/\BM'_{K_\chi}) 
\otimes \Z_p)_\varphi \!\simeq \Z_p[\mu_{g_\chi^{}}]/
{\mathfrak p}_\varphi^{\lambda_\varphi}$\!, $\lambda_\varphi \geq 0$, for all $\varphi \mid \chi$.
\end{lemma}

\begin{proof}
By assumption, $\BM_{K_\chi}/\BM'_{K_\chi}$ is  a finite monogenic $\Z[\mu_{g_\chi^{}}]$-module, 
of the form $\Z[\mu_{g_\chi^{}}]/{\mathfrak A}$, ${\mathfrak A} \ne 0$; 
so $\CM_{K_\chi}/\CM'_{K_\chi} \simeq (\Z[\mu_{g_\chi^{}}]/{\mathfrak A}) \otimes \Z_p$, giving
$\CM_{K_\chi}/\CM'_{K_\chi}  \simeq \hbox{$\bigoplus$}_{\varphi \mid \chi}
\big[\Z_p[\mu_{g_\chi^{}}]/ {\mathfrak p}_\varphi^{\lambda_\varphi} \big]$, 
with the usual correspondence between primes ${\mathfrak p} \mid p$ 
and $p$-adic characters $\varphi \mid \chi$; whence the claim.
\end{proof}

\begin{remark}
From the Dirichlet--Herbrand theorem on units (see, e.g., \cite[Corollary I.3.7.2, 
Remark I.3.7.3]{Gra6} or \cite[Ch. IX,\S\,4]{Lg}) there exists in $\BE_{K_\chi}$ a unit 
$\varepsilon$ generating, with its conjugates, a subgroup $\BE$ of $\BE_{K_\chi}$ of 
prime-to-$p$ finite index (we may call it a pseudo Minkowski unit since Minkowski unit,
in the strict sense, do not exist in general). 
Then $\CM := \Z_p [G_\chi] \cdot \vert \varepsilon \vert$ is monogenic and torsion-free.
\end{remark}

\smallskip
Let $\CM'_{K_\chi} := \wh \CE_{K_\chi}\!\! \cdot  \CF_{K_\chi}$.
Taking into account orders, monogenicity and the fact
that $(P_\chi(\sigma_\chi))$ annihilates $\CM_{K_\chi}/\CM'_{K_\chi}$,
Lemma \ref{mono} is coherent with an annihilation theorem of the 
$\CH^\ar_\varphi$'s from the results of \S\,\ref{diagram}.

\subsection{Mysterious link between cyclotomic units and classes}\label{mysterious}
The brief overview, that we give now, must be completed by technical elements 
that the reader can find especially in \cite[\S\,15.2, 15.3]{Was} (all of them borrow 
from classical arithmetic) and in the references that we talked about,
giving systematic generalizations of ``Euler systems''.

\smallskip
To simplify, consider the real semi-simple case for $p>2$ with $K=K_\chi$ of 
conductor $f$; for $\varphi \mid \chi$, we need to establish \textit {arithmetic links} between 
$\wt \CE_\varphi = \CE_\varphi/\CF_{\!\varphi}$ and $\CH_\varphi$,
where $\CE_\varphi =: \langle \varepsilon_\varphi^{} \rangle_{\Z_p}^{}$ and
$\CF_{\!\varphi} =: \langle \eta_\varphi^{} \rangle_{\Z_p}^{}$ is built from 
Leopoldt's cyclotomic units (Definitions \ref{defIII3}). 
But $\wt \CE_\varphi$ has, a priori, no obvious connection with class groups,
except the analytic equality $\prod_{\varphi \mid \chi}\order \CH_\varphi =
\prod_{\varphi \mid \chi} \order \wt \CE_\varphi$ (Corollary \ref{chisemisimple}). 

\smallskip
The trick, for the proof of the Finite AMC, 
consists in using a classical context of ``analytic genus theory'', by 
means of auxiliary cyclic $\ell$-ramified extensions $K(\mu_\ell^{})$ of degree 
multiple of the exponent $\lambda\,p^e$, $e \geq 1$, of $\CH_K$
(e.g., $\ell \equiv 1 \pmod{2p^N}$, $N \gg e$). 

\smallskip
Let $\ell \nmid f$, $\ell \equiv 1 \pmod{2p^N}$, totally split 
in $K$; put $L_0 = \Q(\mu_\ell^{})$ and $L := L_0 K$:

\smallskip
Let $\eta_{f\ell} = 1 - \zeta_{f\ell}$, $\eta_{f} = 1 - \zeta_{f}$,
$\eta_{\ell} = 1 - \zeta_{\ell}$ and consider the cyclotomic numbers
$\eta_L^{}:= \Norm_{\Q(\mu_{f \ell}^{})/L} \big(\eta_{f\ell} \big)$,
$\eta_K^{}:= \Norm_{\Q(\mu_{f }^{})/K}(\eta_f)$; by 
assumption on the total splitting of $\ell$ in $K/\Q$, 
$\Norm_{L/K}(\eta_L^{}) = 1$ (cf. Lemma \ref{lemII10}).
We remark that $\eta_{f\ell} \equiv \eta_f \pmod {\pi_\ell}$ where
$\pi_\ell := \eta_\ell$ is an uniformizing parameter at the places above $\ell$ in 
$L_0$, so that $\eta_L^{} \equiv \eta_K^{} \pmod{\pi_\ell}$,
giving a $\ell$-adic link between $\eta_K^{}$ and $\eta_L^{}$ which will be
fundamental for the congruences \eqref{congruences}:

\subsubsection{Schema XI} \label{figXI}
\unitlength=0.65cm
\[\vbox{\hbox{\hspace{1.5cm}
\begin{picture}(11.5,3.6)
\put(4.3,3.2){\line(1,0){2.85}}
\put(2.1,3.2){\line(1,0){1.65}}
\put(4.3,0.2){\line(1,0){2.85}}
\put(1.85,0.2){\line(1,0){1.85}}
\put(1.5,0.5){\line(0,1){2.4}}
\put(4.00,0.5){\line(0,1){2.4}}
\put(7.5,0.5){\line(0,1){2.4}}
\put(4.1,1.6){\ft$\langle s \rangle$}
\put(7.6,1.6){\ft$\ell - 1$}
\put(2.6,3.35){\ft$G_K$}
\put(7.2,3.1){\ft$\Q(\mu_{f \ell})$}
\put(3.85,3.1){\ft$L$}
\put(4.3,2.8){\ft$\eta_L^{}$}
\put(-0.25,2.8){\ft$\pi_\ell^{}$}
\put(0.0,3.1){\ft$L_0\!=\!\Q(\!\mu_{\ell}\!)$}
\put(8.5,2.8){\ft$\eta_{f \ell}^{}$}
\put(3.85,0.1){\ft$K$}
\put(4.3,0.5){\ft$\eta_K^{}$}
\put(7.3,0.1){\ft$\Q(\mu_{f})$}
\put(8.5,0.5){\ft$\eta_{f}^{}$}
\put(1.34,0.1){\ft$\Q$}
\end{picture}}} \]
\unitlength=1.0cm

A main step is to apply Hilbert's Theorem $90$ (Kummer's Theorem \cite[II]{Kum}), 
saying that $\eta_L^{} = \alpha_L^{s-1}$, where $s$ is a generator of 
$\Gal(L/K)$ and $\alpha_L^{} \in L^\times$ is such that 
$(\alpha_L^{}) \in \BI_L^{\langle s \rangle}$, where $\BI$ denotes ideal groups; 
since $\alpha_L^{}$ is defined modulo $K^\times$, we can take
$\alpha_L^{}$ integer in $L$ (or at least $\ell$-integer), whence:
\begin{equation}\label{resolvent}
(\alpha_L^{}) = \J_{L/K}({\mathfrak a}_K) \cdot {\mathfrak L}_0^{\Omega_\ell}, 
\end{equation}  
where ${\mathfrak a}_K \in \BI_K$ may be taken prime to $\ell$, 
where ${\mathfrak L}_0$ is a fixed prime ideal dividing $\ell$ in $L$ and:
\begin{equation}\label{omega}
\hbox{${\Omega_\ell} = \sum_{\sigma \in G_K} r_\sigma \cdot \sigma^{-1}$, $r_\sigma \geq 0$;} 
\end{equation}
thus, since $\Norm_{L/K}({\mathfrak L}_0) = {\mathfrak l}_0$, 
${\mathfrak L}_0 \mid {\mathfrak l}_0 \mid \ell$ in $L/K$:
\begin{equation}\label{resolvent0}
(\alpha_K) := (\Norm_{L/K}(\alpha_L^{})) = 
{\mathfrak a}_K^{\ell-1} \cdot {\mathfrak l}_0^{\,{\Omega_\ell}}.
\end{equation}

But ${\mathfrak a}_K^{\ell - 1}$ is principal, whence ${\mathfrak l}_0^{\Omega_\ell}$ 
principal.

\smallskip
The following property elucidates the ``mysterious link'' giving an information 
that we can ``project'' on each $\varphi$-component and obtain the annihilation 
of the $\varphi$-class of ${\mathfrak l}_0$ by the $\varphi$-component of ${\Omega_\ell}$:

\begin{lemma}\label{nontrivial}
Except a finite number of primes $\ell$, the ideal
${\mathfrak L}_0^{\Omega_\ell}$ of \eqref{resolvent} gives a non trivial relation,
in the meaning that ${\Omega_\ell}$ in \eqref{omega}
is not of the form $\lambda \cdot \Nu_{L/L_0}$, 
$\lambda \geq 0$, giving ${\mathfrak l}_0^{\Omega_\ell} = (\ell)^\lambda$ in \eqref{resolvent0}.
\end{lemma}

\begin{proof}
Assume that ${\Omega_\ell} = \lambda \cdot \Nu_{L/L_0}$; the character of
${\mathfrak L}^{\Omega_\ell}_0 = (\pi_\ell^\lambda)$, as $\Z[G_K]$-module,
is the unit one and any non-trivial 
$\varphi$-component $\alpha_{L,\varphi}^{}$ of $\alpha_{L}^{}$ is
prime to $\ell$, thus congruent, modulo any ${\mathfrak L} \mid \ell$, to 
$\rho_{\mathfrak l}^{} \in \Z$, $\rho_{\mathfrak l}^{} \not\equiv 0 \pmod \ell$ 
(residue degrees~$1$ in $L/\Q$). 
Since ${\mathfrak L}^s = {\mathfrak L}$, we obtain $\eta_{L,\varphi^{}}^{} 
=\alpha_{L,\varphi}^{s-1} \equiv 1 \pmod {\mathfrak L}$;
but $\eta_{K,\varphi^{}}^{} \equiv \eta_{L,\varphi^{}}^{} \pmod {\pi_\ell}$
leads to $\eta_{K,\varphi^{}}^{} \equiv 1 \pmod {\mathfrak l}$, for all 
${\mathfrak l} \mid \ell$, giving 
$\eta_{K,\varphi^{}}^{} \equiv 1 \pmod \ell$ (absurd for almost all $\ell$).
\end{proof}

Reducing modulo $\Nu_{L/L_0}$, one may get ${\Omega_\ell} \ne 0$,
``minimal'' in an obvious sense, with $r_\sigma \geq 0$ but not all zero.
Consider $\ffrac{\alpha_L^\sigma}{\pi_\ell^{r_{\sigma}}}$
modulo ${\mathfrak L}_0$ and the conjugations $\alpha_L^s = \alpha_L \cdot \eta_L$
and $\ffrac{\pi_\ell^s}{\pi_\ell} = 
\frac{1-\zeta_\ell^{\g_\ell^{}}}{1-\zeta_\ell} \equiv \g_\ell^{} \pmod{\pi_\ell}$ 
(where $\g_\ell^{}$ is a primitive root modulo $\ell$ such that $\zeta_\ell^s =: 
\zeta_\ell^{\g_\ell^{}}$); one gets:
\[\Big(\ffrac{\alpha_L^\sigma} {\pi_\ell^{r_\sigma}} \Big)^s =
\ffrac{\alpha_L^{s \sigma}}{\pi_\ell^{s r_\sigma}} \equiv
\ffrac{\eta_L^\sigma \alpha_L^{\sigma}}{(\g_\ell \pi_\ell)^{r_\sigma}}
\equiv
\ffrac{\eta_L^\sigma} {\g_\ell^{r_\sigma}} \cdot 
\ffrac{\alpha_L^{\sigma}} {\pi_\ell^{r_\sigma}} \pmod{{\mathfrak L}_0}, \]
whence:
\begin{equation}\label{congruences}
\g_\ell^{r_\sigma} \equiv \eta_L^{\sigma}
\equiv \eta_K^{\sigma} \pmod {{\mathfrak l}_0}.
\end{equation}

Notice that if $r_{\sigma} = 0$ for all $\sigma \in G_K$, the above process is empty.
So we have obtained a non-trivial relation between the classes of the conjugates 
of ${\mathfrak l}_0$; for instance, 
if $\eta_{K,\varphi^{}}^{} = \varepsilon_{K,\varphi^{}}^{p^h}$, one gets 
$r_\sigma^{} \equiv 0 \pmod {p^h}$, whence a property of annihilation
of the $\varphi$-class group.
Recall that $\alpha_L^{}$ is given by an explicit Hilbert resolvent 
allowing explicit computations.

\begin{remark}
(i) In the literature, the properties of the $\alpha_L$'s give rise to an homomorphism
$\BF_{\!K}^{}/\BF_{\!K}^{p^N} \to \Z/p^N\Z\, [G_K]$,  of $\Z_p[G_K]$-modules, 
allowing reasoning for the $\varphi$-components. To get more information, 
one varies $\ell$, using Chebotarev's Theorem and Nakayama's Lemma. 
Then the problem of the order of the $\CH_\varphi$'s needs the knowledge 
of the whole analytic formula of Theorem \ref{theoIII2} (see the 
details in \cite[\S\,15.2, 15.3]{Was}, from Thaine's Theorem). 

\smallskip
(ii) We will return elsewhere to the links with genus theory given by the following
fixed-points exact sequence (obtained from the invariant class of ${\mathfrak A}_L$,
${\mathfrak A}_L^{1-s} =: (\alpha_L) \mapsto \Norm_{L/K}(\alpha_L) =: \varepsilon_K$):
\[1 \to \cl_L(\BI_L^{\langle s \rangle}) \otimes \Z_p \too \CH_L^{\langle s \rangle} \too 
\CE_K \cap \Norm_{L/K}(L^\times) / \Norm_{L/K}(\CE_L) \to 1\]
and (in the present context) the Chevalley--Herbrand formula 
\cite[pp.\,402-406]{Che} in $L/K$:
\[\order \CH_L^{\langle s \rangle} = \order \CH_K \cdot 
\ffrac{p^{e\, ([K : \Q]-1)}}{(\CE_K : \CE_K \cap \Norm_{L/K}(L^\times))} \]
and similar formulas in the sub-extensions 
of $L/K$ (noting that the exact sequence and Chevalley--Herbrand's formula 
may be written in terms of $\varphi$-objects without any difficulty; cf. \cite{Gra13,Gra15}).
The reason of such a link with genus theory is the fact that, assuming $\CF_M=\CE_M$ 
for the subfield $M$ of $L$ of degree $p$ over $K$ we know that $\Norm_{L/M}(\CF_L) 
= \CF_M = \CE_M$, so that the above exact sequence in $L/M$ reduces to 
$\CH_L^{\langle s^p \rangle} = \cl_L(\BI_L^{\langle s^p \rangle}) \otimes \Z_p$
and $\order \CH_L^{\langle s^p \rangle} = \order \CH_M \cdot p^{e\, ([K : \Q]-1)}$.

\smallskip
(iii) Any ``$\,\G$-family of numbers $\eta$\,'' satisfying, in cyclic extensions $L/K$,
relations of the form $\Norm_{L/K}(\eta_L^{}) = \eta_K^{1 - {\rm Frob}_{L/K}(\ell)}$ and
$\eta_L^{} \equiv \eta_K^{} \pmod {\prod_{{\mathfrak l} \mid \ell}  {\mathfrak l}}$,
for suitable primes $\ell$, is called an ``Euler system'' \cite{Kol, PR} and gives 
rise to similar reasonings in many domains.

\smallskip
(iv) Equations of the general form $\Norm_{L/K}(y) = \Norm_{L/K}({\mathfrak B})$,
giving $(y) = {\mathfrak B}\cdot {\mathfrak A}^{s-1}$, are fundamental in various 
questions, as Greenberg's conjecture, in a genus theory framework (see \cite[\S\,3, 
Algorithm]{Gra7}). Such equations are due to some $x \in K^\times$,  
local norm in $L/K$ at the $\ell$-places, such that $(x) = \Norm_{L/K}({\mathfrak B})$,
giving the relation $x=\Norm_{L/K}(y)$, for some unknown $y$
(Hasse's norm theorem in $L/K$).
In various papers (as \cite[\S\,7.1]{Gra12}) we have discussed these random 
aspects by computing some ideals ${\mathfrak A}$, so that we may 
conjecture the following more precise property (see Schemas \ref{figIX}, 
\ref{figX}, Lemma \ref{nontrivial}, Relations \eqref{resolvent}--\eqref{congruences}).
\end{remark}

\begin{conjecture}\label{omegaconj}
Let $K$ be a real abelian field of conductor $f$, of $p$-class group of exponent less 
than $2 p^N$ and let $\eta_K := \Norm_{\Q(\mu_{f}^{})/K} \big(1-\zeta_{f} \big)$.
Consider primes $\ell \equiv 1 \pmod {2 p^N}$, totally split in $K$; 
let ${\mathfrak l}_0 \mid \ell$ in $K$ and let $\g_\ell$ be a primitive root modulo $\ell$. 
When $\ell$ varies, $\eta_K^{}$ provides infinitely many
elements ${\Omega_\ell} = \sum_{\sigma \in G_K} r_\sigma \cdot \sigma^{-1}$, with
$\eta_K^{\sigma}\equiv \g_\ell^{r_\sigma} \pmod {{\mathfrak l}_0}$, such that
the ideal generated by these relations yields annihilators of the $\varphi$-components 
$\CH^\ar_\varphi$ as $\Z_p[G_K]$-modules and possibly their structure.
\end{conjecture}

The program, written in Appendix \ref{Omega}, for cyclic cubic fields,
computes the invariants $\psi(\Omega_\ell) \in \Z[j]$ 
only with the knowledge of $\eta_K^{}$ and gives tables of results.

\smallskip
These numerical experiments are particularly remarkable and confirm
that the $\Omega_\ell$'s define an universal $\Omega_K$ which
replaces, in the real case, the Stickelberger element of the imaginary case.
For this, we notice that the embeddings (injectivity from 
\cite[Theorem III.4.4]{Gra6}) of $\CF_K$ and $\CE_K$ in the 
direct product $\prod_{v \nmid p}(F_v^\times \otimes \Z_p)$ (see  
Schemas \ref{figIX}, \ref{figX}) govern the congruences \eqref{congruences} 
giving the relations $\Omega_\ell$ involving only $\BF_K$, without any memory
of the arithmetic of the auxiliary fields $\Q(\mu_\ell^{})$. 
Then, the Schmidt--Chevalley theorem (local--global principle for powers, e.g.,
\cite[\S\,6.3, Theorem II.6.3.3]{Gra6}) claims that there are infinitely many primes 
$\ell$ (totally split in $K$) giving the ``good'' $\Omega_K$.

\smallskip
From Lemma \ref{mono} giving standard structure of $\CE_\varphi$ 
and  $\CF_\varphi$, it is then obvious that one obtains equalities of the 
$\varphi$-invariants $m_\varphi^\ar$ of $\CE_\varphi/\CF_\varphi$ and $\CH_\varphi$
in the semi-simple case. 

\smallskip
Are there improvements of these techniques 
being able to distinguish, for instance, the structures
$\Z_p[\mu_{g_\chi^{}}]/{\mathfrak p}_{\varphi} \oplus 
\Z_p[\mu_{g_\chi^{}}]/{\mathfrak p}_{\varphi}$ and 
$\Z_p[\mu_{g_\chi^{}}]/{\mathfrak p}_{\varphi}^2$\,?

\begin{remark}\label{remafond}
After the writing of this paper, we have considered the phenomenon of capitulation of 
classes to give another approach of the Finite AMC in any real case (semi-simple or not).
We develop, in these articles \cite{Gra13,Gra15}, new promising links between: 
(i) the Chevalley--Herbrand formula giving the number of ``ambiguous 
classes'' in $p$-exten\-sions $L/K$, $L \subset K(\mu_\ell^{})$, for  
auxiliary primes $\ell \equiv 1\!\! \pmod {\!2p^N}$ inert in $K$;
(ii) the phenomenon of capitulation of $\CH_K$ in $L$;
(iii) the real Finite AMC $\order \CH^\ar_\varphi = (\CE_{K_\chi} : 
\wh \CE_{K_\chi}\! \cdot \CF_{K_\chi})_{\varphi_0^{}}$ 
for all $\varphi \mid \chi$. 

\smallskip
We prove that the real Finite AMC is trivially fulfilled as soon as $\CH_K$ capitulates 
in $L$ and conjecture that there exist infinitely many such primes $\ell$ leading to capitulation. 

\smallskip
Computations with PARI programs support this new philosophy of the Finite AMC
and justifies, once again, the relevance of the analytic definitions, especially in the non
semi-simple case.
\end{remark}

\section{Invariants (Algebraic, Arithmetic, Analytic)}\label{MainConj}
We fix an irreducible rational character $\chi \in \CX = \CX^+ \cup \CX^-$
and we apply the previous results to the $\Z_p[\mu_{g^{}_\chi}]$-modules 
$\CH^\alg_\varphi$, $\CH^\ar_\varphi$ and
$\CT^\ar_\varphi = \CT^\alg_\varphi =: \CT_\varphi$, 
for any $\varphi \mid \chi$, $\varphi \in \Phi^+ \cup \Phi^-$ ($\varphi \in \Phi^+$
for $\CT_\varphi$).

\subsection{Algebraic and Arithmetic Invariants 
\texorpdfstring{$m^\alg {\sst(\CM)}$, $m^\ar {\sst(\CM)}$}{Lg}}
\label{invariants}
Write simply that $\CH^\alg_\varphi$, $\CH^\ar_\varphi$ and $\CT_\varphi$ are 
finite $\Z_p[\mu_{g^{}_\chi}]$-modules whatever $\varphi$;
let ${\mathfrak p}_\varphi$ be the maximal ideal of $\Z_p[\mu_{g^{}_\chi}]$:
\begin{equation*}
\left\{\begin{aligned}
\CH^\alg_\varphi& \simeq \hbox{$\prod_{i \geq 1}$}\
\Z_p[\mu_{g^{}_\chi}] \big / {\mathfrak p}_\varphi^{\,n^\alg_{\varphi, i}{\sst(\CH)}}, \\
\CH^\ar_\varphi & \simeq \hbox{$\prod_{i \geq 1}$}\
\Z_p[\mu_{g^{}_\chi}] \big / {\mathfrak p}_\varphi^{\,n^\ar_{\varphi, i}{\sst(\CH)}}, \\   
\CT_\varphi &\simeq \hbox{$\prod_{i \geq 1}$} \
\Z_p[\mu_{g^{}_\chi}] \big / {\mathfrak p}_\varphi^{\, n^\ar_{\varphi, i}(\CT)}, 
\end{aligned}\right.
\end{equation*}
where the $n_{\varphi, i}$ are decreasing integers up to $0$. Put:
\begin{equation*}
\left\{\begin{aligned}
 m^\alg_\varphi {\sst(\CH)} & := \hbox{$\sum_{i \geq 1}$}\, n^\alg_{\varphi, i} {\sst(\CH)}, \ \  \ \ \ \ 
 m^\alg_\chi  {\sst(\CH)}  := \hbox{$\sum_{\varphi \mid \chi}$}\, m^\alg_\varphi {\sst(\CH)}, \\
 m^\ar_\varphi {\sst(\CH)} & :=\, \hbox{$\sum_{i \geq 1}$}\, n^\ar_{\varphi, i} {\sst(\CH)}, \ \  \ \ \ \ 
 m^\ar_\chi {\sst(\CH)}  := \hbox{$\sum_{\varphi \mid \chi}$}\, m^\ar_\varphi {\sst(\CH)}, \\
 m^\ar_\varphi {\sst(\CT)} & := \,\hbox{$\sum_{i \geq 1}$}\, n^\ar_{\varphi, i} {\sst(\CT)},  \ \ \ \ \ \ \ 
 m^\ar_\chi {\sst(\CT)}  := \hbox{$\sum_{\varphi \mid \chi}$} \, m^\ar_\varphi {\sst(\CT)}.
\end{aligned}\right.
\end{equation*}

Whence the order formulas:
\[\order \CH^\alg_\varphi = p^{\varphi(1) \, m^\alg_\varphi {\sst(\CH)}}, \ \ \  
\order \CH^\ar_\varphi = p^{\varphi(1) \, m^\ar_\varphi {\sst(\CH)}}, \ \ \  
\order \CT_\varphi = p^{\varphi(1) \, m^\ar_\varphi {\sst(\CT)}}. \]

\subsection{Analytic Invariants \texorpdfstring{$m^\an {\sst(\CM)}$}{Lg}}\label{analytics}
We define, in view of the statement of the Finite AMC, the following 
Analytic Invariants $m^\an_\varphi$, from the expressions given with rational characters, 
where $\val_p(\bullet)$ denotes the usual $p$-adic valuation; the purpose is
to satisfy the necessary relations implied by Theorems \ref{theoI5}, \ref{theoI2}
about arithmetic components:
\[\sm_{\varphi \mid \chi} m^\ar_\varphi {\sst(\CM)} = 
\sm_{\varphi \mid \chi} m^\an_\varphi {\sst(\CM)}, \]
for any family $\CM \in \{\CH, \CT\}$ and $\chi \in \CX$
(cf. Theorems \ref{theoII2}, \ref{theoIII2}, \ref{theotorsion2}).

\subsubsection {Case \texorpdfstring{$\varphi \in \Phi^-$}{Lg} for class groups}
Then, Algebraic and Arithmetic Invariants coincide. The definitions given in
\cite{Gra, Gra0} were:

\smallskip
(i) Case $p \ne 2$ (proven by Solomon \cite[Theorem II.1]{Sol}).

\smallskip
\quad (i$'$) $K_\chi$ is not of the form $\Q(\mu_{p^n})$, $n \geq 1$; then:

\smallskip
$\bullet$ $m^\an_\varphi {\sst(\CH^-)} :=\val_p 
\Big(\prd_{\psi \mid \varphi} \big(-\hbox{$\frac{1}{2}$} \BB_1(\psi^{-1}) \big) \Big)$,

\quad (i$''$) $K_\chi = \Q(\mu_{p^n})$, $n \geq 1$; let $\psi = \omega^\lambda \cdot \psi_p^{}$,
$\psi_p^{}$ of order $p^{n-1}$ (where $\omega$ is the Teichm\"uller character); then:

\smallskip
$\bullet$ $m^\an_\varphi {\sst(\CH^-)} :=\val_p  
\Big(\prd_{\psi \mid \varphi} \big(- \hbox{$\frac{1}{2}$} \BB_1(\psi^{-1}) \big) \Big)$, 
if $\lambda \ne 1$,

$\bullet$ $m^\an_\varphi {\sst(\CH^-)} := 0$, if $\lambda = 1$.

\smallskip
(ii) Case $p = 2$ (proven by Greither \cite[Theorem B]{Grei},
when $g^{}_\chi$ is not a $2$-power and $f_\chi$ odd).

\smallskip
\quad (ii$'$) $g^{}_\chi$ is not a $2$-power; then:

\smallskip
$\bullet$ $m^\an_\varphi {\sst(\CH^-)} :=\val_2 
\Big (\prd_{\psi \mid \varphi} \big(- \hbox{$\frac{1}{2}$}\BB_1(\psi^{-1}) \big) \Big)$.

\quad (ii$''$) $g^{}_\chi$ is a $2$-power; then:

\smallskip
$\bullet$ $m^\an_\varphi {\sst(\CH^-)} :=\val_2
\Big (\prd_{\psi \mid \varphi} \big(- \hbox{$\frac{1}{2}$} \BB_1(\psi^{-1}) \big) \Big) + 1$, 
if $K_\chi \ne \Q(\mu_4^{})$,

$\bullet$ $m^\an_\varphi {\sst(\CH^-)} := 0$, if $K_\chi = \Q(\mu_4^{})$.

\subsubsection{Case \texorpdfstring{$\varphi \in \Phi^+$, $\varphi \ne 1$}{Lg}, for class groups}\label{wphi}

From Definition \ref{defIII3} and Theorem \ref{theoIII2}, we consider any real cyclic field $K$, 
where we recall that: 

\smallskip\noindent
$\wh \BE_K := \langle \, \BE_k \, \rangle^{}_{k \varsubsetneqq K}$,
$\BF_K := \BC_K \cap \BE_K$, $\CE_K := \BE_K \otimes \Z_p$,
$\wh \CE_K := \wh \BE_K \otimes \Z_p$, $\CF_K := \BF_K \otimes \Z_p$, and
$\wt \CE_\chi  := \CE_{K_\chi}/\wh \CE_{K_\chi} \!\! \cdot \! \CF_{\!K_\chi}$, for which 
$\wt \CE_\chi  = \plus_{\varphi \mid \chi} \wt \CE_\varphi$, where
$\wt \CE_\varphi = \{\wt x \in \wt \CE_\chi,\  
 \wt x^{P_\varphi(\sigma_\chi)} = 1\}$.

\smallskip
Consider the relation $\order \CH^\ar_\chi  = w_\chi \! \cdot 
\big (\CE_{K_\chi}\! : \wh \CE_{K_\chi}\! \! \cdot \CF_{K_\chi} \big) =
w_\chi \cdot \prod_{\varphi \mid \chi} \order \wt \CE_\varphi$ 
of Theorem \ref{theoIII2}; we remark that $w_\chi = p$ occurs only
when $g^{}_\chi$ is a $p$-power, in which case $p$ is totally ramified in $\Q(\mu_{g^{}_\chi})$
and $\varphi = \chi$ (which defines $w_\varphi := w_\chi$). 
So, we may define $m^\an_\varphi {\sst(\CH^+)}$ and $w_\varphi$ as follows
from $\wt \CE_\varphi \simeq 
\Z_p[\mu_{g^{}_\chi}] \big/ {\mathfrak p}_\varphi^{m^\an_\varphi  {\sst(\CH^+)}}$,
$m^\an_\varphi {\sst(\CH^+)} \geq 0$:

\smallskip
(i) Case $g^{}_\chi$ non prime power. Then $w_\varphi = 1$ and:

\smallskip
$\bullet$ $m^\an_\varphi {\sst(\CH^+)} := \val_p (\order \wt\CE_\varphi)$.

\medskip
(ii) Case $g^{}_\chi = p^n$, $p \ne 2$ prime, $n \geq 1$:

\smallskip
\qquad (ii$'$) Case $f_\chi = \ell^k$, $\ell$ prime, $k \geq 1$. Then $w_\varphi = 1$ and :

\smallskip
$\bullet$ $m^\an_\varphi {\sst(\CH^+)} := \val_p (\order \wt\CE_\varphi)$,

\smallskip
\qquad (ii$''$) Case $f_\chi$ non prime power. Then $w_\varphi = p$ and

\smallskip
$\bullet$ $m^\an_\varphi {\sst(\CH^+)} := \val_p (\order \wt\CE_\varphi) + 1$.

\medskip
(iii) Case $g^{}_\chi =2^n$, $n \geq 1$:

\smallskip
\qquad (iii$'$) Case $f_\chi = \ell^k$, $\ell$ prime, $k \geq 1$. Then $w_\varphi = 1$ and:

\smallskip
$\bullet$ $m^\an_\varphi {\sst(\CH^+)} := \val_p (\order \wt\CE_\varphi)$,

\medskip
\qquad (iii$''$) Case $f_\chi$ non prime power. Then $w_\varphi \in\{1, 2\}$ and:

\smallskip
$\bullet$ $m^\an_\varphi {\sst(\CH^+)} \in\{\val_p (\order \wt\CE_\varphi), 
\val_p (\order \wt\CE_\varphi) + 1\}$.

\subsubsection{Case \texorpdfstring{$\varphi \in \Phi^+$}{Lg} for 
\texorpdfstring{$p$}{Lg}-torsion groups}

From Theorem \ref{theotorsion2}, we define $m^\an_\varphi {\sst(\CT)}$
as follows (proven by Greither \cite[Theorem C]{Grei},
when $g^{}_\chi$ is not a $2$-power):

\medskip
(i) Case where $g^{}_\chi$ and $f_\chi$ are not $p$-powers. Then:

\smallskip
$\bullet$ $m^\an_\varphi {\sst(\CT)} := \val_p 
\Big(\prod_{\psi \mid \varphi} \hbox{$\frac{1}{2}$}\,\BL_p (1,\psi) \Big)$.

(ii) Case where $g^{}_\chi \ne 1$ and $f_\chi$ are $p$-powers. Then:

\smallskip
$\bullet$ $m^\an_\varphi {\sst(\CT)} := \val_p 
\Big(\prod_{\psi \mid \varphi} \hbox{$\frac{1}{2}$} \,\BL_p (1,\psi) \Big) + 1$.

\subsection{Finite Abelian Main Conjecture}

The conjecture we gave in \cite{Gra, Gra0}, especially in the non semi-simple 
case, where simply equality of Arithmetic and Analytic $\varphi$-Invariants.
The main justification of such equalities comes from the easy Theorem \ref{chiformula} 
with the arithmetic definitions of \S\,\ref{invariants}, the analytic definitions 
of \S\,\ref{analytics} and the arithmetic expressions of the $\chi$-components
that we recall:

\smallskip
(i) Theorem \ref{theoII2}: $\BH_\chi^\ar = 2^{\alpha_\chi} \cdot w_\chi \cdot 
\prod_{\psi \mid \chi}\big(- \hbox{$\frac{1}{2}$} \BB_1(\psi^{-1}) \big)$, for $\chi \in \CX^-$, 

\smallskip
(ii) Theorem \ref{theotorsion2}: $\order \CT_\chi = w_\chi^{\,\cyc} \cdot 
\prod_{\psi \mid \chi} \hbox{$\frac{1}{2}$} \,\BL_p (1,\psi)$, for $\chi \in \CX^+$,

\smallskip
(iii) Theorem \ref{theoIII2}: $\order \BH^\ar_\chi = 
w_\chi \cdot (\BE_{K_\chi} : \wh \BE_{K_\chi} \! \cdot \BF_{K_\chi})$, for $\chi \in \CX^+$;

\smallskip\noindent
they satisfy, for any family $\CM \in \{\CH^-,\, \CH^+,\, \CT\}$, the equalities: 

\medskip
$\bullet$ $\sum_{\varphi \mid \chi} m^\ar_\varphi {\sst(\CM)} = 
\sum_{\varphi \mid \chi} m^\an_\varphi {\sst(\CM)},\ \hbox{for all $\chi\in \CX$}$, 

\medskip\noindent
taking into account the decomposition $\CM^\ar_\chi \!= 
\oplus_{\varphi \mid \chi} \CM^\ar_\varphi$ (Theorem \ref{theoI2bis}).

\smallskip
Moreover, the annihilation properties of Theorems \ref{theoII5}, \ref{theoII6},
\ref{annihilation2}, \ref{annTphi}, enforce the conjecture as well as reflection 
theorems that were given, after the Leopoldt's Spiegelungsatz, in \cite{Gra5} or 
\cite[Theorem II.5.4.5]{Gra6} giving a more suitable comparison, for instance 
between $\CH_\varphi$ and $\CT_{\omega \varphi^{-1}}$, $\varphi \in \Phi^-$, 
where $\omega$ is the Teichm\"uller character. 
See also \cite{Or1, Or2} for similar informations and complements.

\begin{conjecture}\label{mainconj}
For any $p$-adic irreducible character $\varphi  \in \Phi$, we have:
\begin{equation*}
\left\{\begin{aligned}
m^\ar_\varphi {\sst(\CH)} &= m^\an_\varphi {\sst(\CH)} \ 
(\varphi  \in \Phi^+ \cup \Phi^-), \\
m^\ar_\varphi {\sst(\CT)} &= m^\an_\varphi {\sst(\CT)}\ \ \ (\varphi  \in \Phi^+). 
\end{aligned}\right.
\end{equation*}
\end{conjecture}

\begin{remark}\label{remgreither}
Let $K/\Q$ with a maximal $p$-sub-extension
$K/K_0$ cyclic of degree $p^n$, $n \geq 1$, and let $K_i$,
$K_0 \subseteq K_i \subset K$, be such that $[K_i : K_0] = p^i$. 
Let $\psi_0 \in \Psi_{K_0}$ and let $\psi_p \in \Psi_K$
of order $p^n$; we put $\psi_i = \psi_0 \cdot \psi_p^{p^{n-i}} \in \Psi_{K_i}$ 
and we consider the $p$-adic characters $\varphi_i$ above $\psi_i$, 
$0 \leq i \leq n$.

\smallskip
The Main Conjecture proven by Greither in \cite[Theorem 4.14, 
Corollary 4.15]{Grei}, using Sinnott's cyclotomic units, deals with 
the semi-simple context defined by $\varphi_0^{}$ above $\psi_0$ (it is 
indeed that of the relations \eqref{semicontext} which do not give each 
$\order \CH^\ar_{\varphi_i}$ compared with $\order \wt \CE_{\varphi_i}$).

\smallskip
In other words, in his pioneering work, Greither proves the relation 
$\sm_{i=0}^n m^\ar_{\varphi_i} {\sst(\CH^+)} = 
\sm_{i=0}^n m^\an_{\varphi_i}  {\sst(\CH^+)}$, for each 
$\varphi_0^{} \in \Phi_{K_0}$, instead of our conjecture 
$m^\ar_{\varphi_i} {\sst(\CH^+)} = m^\an_{\varphi_i}  {\sst(\CH^+)}$ 
for all $i \in \{0,1, \cdots, n\}$. However see many improvements 
by Greither--Ku\v cera in \cite{GK0, GK1} and some of their other papers.
\end{remark}

\begin{remark} \label{casalg}
It remains the problem of $\order \CH^\alg_\chi$ and
$\order \CH^\alg_\varphi$, for which no analytic formula does exist
in the non semi-simple real case. For instance, in Example \ref{Ex1}
with $p=3$, $K$ is the compositum of $k_0=\Q(\sqrt{4409})$ with the 
degree $9$ field of conductor $19$,
$\chi_i = \varphi_i$ ($i \in \{0,1,2\}$) is the character of the field $k_i$ of degree
$2\cdot 3^i$; then one gets $\CH^\alg_{\chi_i^{}} \simeq \Z/3\Z$ while 
$\CH^\ar_{\chi_i^{}} = 1$, as predicted by the conjecture and checked numerically.
In Example \ref{Ex2}, one finds $\CH^\alg_{\chi_1^{}} \simeq (\Z/3\Z)^3$ 
while $\CH^\ar_{\chi_1^{}} \simeq (\Z/3\Z)^2$.

Of course, the formula $\order \CH^\ar_{\chi_0^{}} \cdot \order \CH^\ar_{\chi_1^{}} 
\cdot \order \CH^\ar_{\chi_2^{}} = \order \CH^\alg_K$ does not hold for the algebraic 
definition of class groups.

\smallskip
This phenomenon is due to the capitulation of $p$-classes in $p$-exten\-sions and we 
have given in \cite[Conjecture 4.1]{Gra11} a general conjecture justified by means of 
many computations.
\end{remark}

\subsection{Finite Iwasawa's theory in cyclic \texorpdfstring{$p$}{Lg}-extensions}
For more details and an application to classical Iwasawa's theory for
the cyclotomic $\Z_p$-extensions, see \cite[Chap.\,IV]{Gra} (the real 
case being in the spirit of Greenberg's conjecture \cite{Gree3}); 
nevertheless, \textit {the results hold in arbitrary totally ramified cyclic 
$p$-extensions} of an abelian field, as follows depending of a base 
field real or imaginary: 

\subsubsection{Real case}
Let $\psi \mid \varphi \mid \chi  \in \CX^+$ and set $\psi = \psi_0 \cdot \psi_p$, 
where $\psi_0$ is of order $g_0^{}$, prime to $p$, and $\psi_p$ of $p$-power 
order; then, $G_\chi = G_0 \oplus H$ in an obvious meaning.
We consider, temporarily, the semi-simple idempotents
$e_{\varphi_0^{}} := \frac{1}{g_0} \sm_{\sigma \in G_0}
\varphi_0^{} (\sigma^{-1})\, \sigma$, for $\varphi_0^{}$ above $\psi_0$.
We have:
\[\wt \CE_\chi := \CE_{K_\chi}/\wh \CE_{K_\chi}\!\! \cdot 
\CF_{\!K_\chi} = \plus_{\varphi \mid \chi} \wt \CE_\varphi =
\plus_{\varphi \mid \chi}( \wt \CE_\chi)_{\varphi_0^{}} , \] 
with $\wt \CE_\varphi = \wt \CE_\chi^{\,e_{\varphi_0^{}}}$;
we note that $\wh \CE_{K_\chi}^{\,e_{\varphi_0^{}}}\simeq \CE_{\varphi'}$ 
and $\wt \CE_\varphi \simeq \CE_{K_\chi}^{e_{\varphi_0^{}}} /  \CE_{\varphi'}
\cdot \CF_{\!K_\chi}^{e_{\varphi_0^{}}}$, where $\varphi'$ is above 
$\psi_0\! \cdot\! \psi_p^p$ and $\chi'$ above $\varphi'$.
This yields $(\CE_{K_\chi}/\CE_{K_{\chi'}})_\varphi
\simeq \Z_p[\mu_{g^{}_\chi}]$ (\cite[Lemma IV.1]{Gra})
and the following principle taking place in the layers of any 
$p$-tower $K_N/K_0$, of degree $p^N$ over an abelian 
field $K_0$, totally ramified at a set of finite places of 
$K_0$ \cite[Proposition IV.1]{Gra}:

\begin{theorem}\label{Iwreal}
Let $\chi \in \CX^+$ be such that $g^{}_\chi = g_0^{} \cdot p^n$, $p \nmid g_0^{}$,
$n\geq 2$. Let $\chi'$,  $\chi''$ be such that
$[K_{\chi} : K_{\chi'}]=[K_{\chi'} : K_{\chi''}]= p$. To simplify, set $K := K_{\chi}$,
$K' := K_{\chi'}$, $K'' := K_{\chi''}$ and assume that $\Norm_{K/K'} (\CF_{\!K}) = \CF_{\!K'}$
(see Lemma \ref{lemII10} giving the ramification conditions).
Let ${\mathfrak p}_\varphi$ be the maximal ideal of $\Z_p[\mu_{g^{}_\chi}]$; put
$(\CF_K / \CF_K  \cap \CE_{K'} )_\varphi \simeq {\mathfrak p}_\varphi^A,\ \,A \geq 0$ 
and, in the isomorphism $(\CE_{K'} /\CE_{K''})_{\varphi'} \simeq \Z_p[\mu_{g^{}_\chi/p}]$, put:
\begin{equation*}
\begin{aligned}
(\CF_{\!K'} / \CF_{\!K'} \cap \CE_{K''})_{\varphi'} \simeq 
{\mathfrak p}_{\varphi'}^{a} \simeq {\mathfrak p}_\varphi^{p a}, \ a \geq 0 ,& \\
(\Norm_{K/K'} (\CE_K) / \Norm_{K/K'} (\CE_K) \cap  \CE_{K''})_{\varphi'}
\simeq  {\mathfrak p}_{\varphi'}^{b} \simeq {\mathfrak p}_\varphi^{p\,b},\  b \geq 0.& 
\end{aligned}
\end{equation*}

(i) If $a < p^{n-2}\,(p-1)$, then $A=a-b$.

\smallskip
(ii) If $a \geq p^{n-2}\,(p-1)$, then $A \geq p^{n-2}\,(p-1) - b$.
\end{theorem}

This allows to prove again Iwasawa's formula in the case $\mu=0$
\cite[Theorems IV.1, IV.2, Remark IV.4]{Gra} and gives an 
analytic algorithm to study the $p$-class groups in the first layers. 

\smallskip
Let $k =: k_0$ be real of prime-to-$p$ degree $g$ 
and let $k^\cyc = \bigcup_{n \geq 0} k_n$ be its 
cyclotomic $\Z_p$-extension. The condition $\mu=0$
of Iwasawa's theory is here equivalent to the existence of $n_0 \gg 0$ 
(corresponding to a character $\chi_{n_0}$ of order $g\, p^{n_0}$) such that,
for each $\varphi_{n_0}$-component, $a_{n_0-1} < p^{n_0-2}\,(p-1)$ 
(case (i) of Theorem \ref{Iwreal}); then the sequence 
$\order \CH_{\chi_n^{}}$ becomes constant giving the $\lambda$-invariant
and the relations $\CE_{k_{n-1}} = \Norm_{k_n/k_{n-1}}(\CE_{k_n}) \cdot 
\CE_{k_{n-2}}$, for all $n \gg 0$; then $p^\lambda = 
(\CE_{k_n} : \wh \CE_{k_n} \!\!\cdot\CF_{\!k_n})$ for $n \gg 0$. 
More precisely:
\[p^{\lambda_\varphi} = 
\order (\CE_{k_n} / \CE_{k_{n-1}} \! \cdot \CF_{\!k_n})_{\varphi_0^{}}, \  n \gg 0. \] 

This methodology does exist in terms of $p$-adic $\BL$-functions 
for abelian fields (see, e.g., \cite[Chapitre\,V]{Gra3}).

\smallskip
Recall that Greenberg's conjecture \cite{Gree3} for a totally real base field
(i.e., $\lambda = \mu =0$) is equivalent to the property that the norms
$\Norm_{k_m/k_n} : \CH_{k_m} \to \CH_{k_n}$, $m \geq n \gg 0$ are 
isomorphisms (see other equivalent conditions in \cite[Corollary 3.4]{Gra9}).
Whence the result:

\begin{theorem}
Let $k$ be a real abelian field of prime-to-$p$ degree. 
Greenberg's conjecture is equivalent to 
$(\CE_{k_n} : \wh \CE_{k_n}\! \!\cdot \CF_{\!k_n}) = constant$, for all $n \gg 0$, where
$\wh \CE_{k_n}$ is the subgroup of $\CE_{k_n}$ generated by the units
of the strict subfields and $\CF_{\!k_n}$ is the group 
of Leopoldt cyclotomic units (Definitions \ref{E0}\,(i), \ref{defIII3}).
\end{theorem}

\subsubsection{Imaginary case}

This part is related to relative $p$-class groups for $p \ne 2$ 
\cite[Proposition IV.2, Th\'eor\`eme IV.2]{Gra}:

\begin{theorem}
Let $\chi \in \CX^-$ be such that $g^{}_\chi = g_0^{} \cdot p^n$, 
$p \nmid g_0^{}$, $n\geq 2$. Let $\chi'$ be such that $[K_\chi : K_{\chi'}]=p$.
Set $K := K_{\chi}$, $K' := K_{\chi'}$ and assume that the Stickelberger elements 
$\BB_K$, $\BB_{K'}$ are $p$-integers in $\Q[G_K]$ and that 
$\Norm_{K/K'} (\BB_K) = \BB_{K'}$ (see Lemma \ref{lemII10}). Put:
\begin{equation*}
\begin{aligned}
\BB_1(\psi^{-1}) \Z_p[\mu_{g^{}_\chi}] & = {\mathfrak p}_\varphi^A, \  A \geq 0, \\
\BB_1(\psi^{-p}) \Z_p[\mu_{g^{}_\chi/p}] & = {\mathfrak p}_\varphi^{p\,a}, \  a \geq 0.
\end{aligned}
\end{equation*}

(i) If $a < p^{n-2}\,(p-1)$, then $A=a$.

\smallskip
(ii) If $a \geq p^{n-2}\,(p-1)$, then $A \geq p^{n-2}\,(p-1)$.
\end{theorem}

\begin{remark} 
The integers $A$ and $a$ are the Analytic Invariants 
$m^\an_\varphi {\sst(\CH^-)}$ and $m^\an_{\varphi'} {\sst(\CH^-)}$,
respectively, defined \S\,\ref{analytics}.
From \cite[Remark IV.4]{Gra}, the Iwasawa  $\mu$-invariant is 
zero as soon as there exists $n_0^{} \gg 0$ such that
the case (i) of the theorem is satisfied for all $\varphi$ 
of $K_{n_0^{}}$.
In a $\Z_p$-extension $\wt k/k$, this condition implies that the 
$p$-rank of the $\CH^\ar_{k_n}$'s is bounded (a known result 
of Iwasawa's theory \cite[Proposition 13.23]{Was}).
\end{remark}

\section{Illustrations of the Finite AMC with cubic fields}
\subsection{Introduction}
For $\chi \in \CX^+$ and $\wt \CE_\chi := \CE_{K_\chi} / \wh \CE_{K_\chi}\! \!\cdot \CF_{\!K_\chi}$,
we have $\order \CH^\ar_\chi = w_\chi \cdot \order \wt \CE_\chi$ (Theorem \ref{theoIII2}), 
and for any $\varphi \mid \chi$ we have (conjecturally):
\begin{equation*}
\order \CH^\ar_\varphi = w_\varphi \cdot \order \wt \CE_\varphi, \, 
w_\varphi \in \{1, p\}, \ \hbox{where} \ 
\wt \CE_\varphi = \{\wt x \in \wt \CE_\chi, \  \wt x^{P_\varphi (\sigma_\chi)} = 1\}.
\end{equation*}

In another way, we have:
\begin{equation*}
\left\{\begin{aligned}
\wt \CE_\varphi & \simeq \Z_p[\mu_{g^{}_\chi}] \big / 
{\mathfrak p}_\varphi^{\,m^\an_\varphi  {\sst(\CH)}} \!\!, \ \  m^\an_\varphi {\sst(\CH)} \geq 0, \\
\CH^\ar_\varphi & \simeq \plus_{i=1}^{r_\varphi} \Z_p[\mu_{g^{}_\chi}] \big / 
{\mathfrak p}_\varphi^{\,m^\ar_{\varphi, i} (\CH)}, \ \ r_\varphi \geq 0,
\ \ m^\ar_{\varphi, i} {\sst(\CH)} \geq 0,
\end{aligned}\right.
\end{equation*}

\noindent
and $m^\an_\varphi {\sst(\CH)} := \sum_{i=1}^{r_\varphi^{}} m^\ar_{\varphi, i} {\sst(\CH)}$
to be compared with $m^\ar_\varphi{\sst(\CH)}$.

\smallskip
We intend to see more precisely what happens for these analytic and arithmetic 
invariants since the above equality defining $m^\an_\varphi {\sst(\CH)}$ can 
be fulfilled in various ways (indeed, $\wt \CE_\varphi$ is monogenic and $\CH_\varphi$
may have arbitrary structure). 

\smallskip
We will examine 
the case of the cyclic cubic fields $K = K_\chi$ for primes $p \equiv 1 \pmod 3$ giving two 
$p$-adic characters $\varphi \mid \chi$; in that case, $\wh \CE_K=1$ and
$\order \CH^\ar_\varphi = (\CE_K : \CF_K)$. 

For example, for $p=7$, the possible structures, for the $\Z[j]$-module $\BE_K/\BF_K$, 
are of the form $\Z[j] \big /\big [(-2+j)^{m_1^{}} \cdot (3+j)^{m_2^{}} \cdot {\mathfrak a} \big]$, 
($m_1, m_2 \geq 0$ and ${\mathfrak a}$ prime to $7$), giving the two $\varphi_i$-components 
$\Z_7/7^{m_1^{}} \Z_7$ and $\Z_7/7^{m_2^{}} \Z_7$
(from $\big[\Z[j]/(-2+j)^{m_1^{}}\big] \otimes \Z_7$ and 
$\big[\Z[j]/(3+j)^{m_2^{}} \otimes \Z_7\big]$), for the $\wt \CE_\varphi$'s.

\subsection{Description of the computations}\label{nfgaloisconj}

The PARI program computing all the cyclic cubic 
fields is that given in \cite[\S\,6.1]{Gra9}.

\smallskip
A crucial fact, without which the checking of the $\varphi$-components 
of the $G_K$-modules $\CE_K/\CF_K$ and $\CH_K$ could be misleading, 
is the definition of a generator $\sigma$ of $G_K$ giving the correct
conjugation, both for the fundamental units, the cyclotomic ones and the 
elements of the class group (see more comments at the beginning of  
Appendix \ref{appA}).

\smallskip
It is not too difficult to find, from ${\sf K.fu}$ giving a $\Z$-basis of $\BE_K$, 
a ``Minkowski unit'' $\varepsilon$ and its conjugate $\varepsilon^\sigma$
such that $\langle \varepsilon, \varepsilon^\sigma \rangle_\Z^{} = \BE_K$,
up to a prime-to-$p$ index; indeed, for
the evaluation of $\varepsilon(x)$ and $\varepsilon(g(x))$, at
a root $\rho \in \R$ of $P$, we only have a set $\{\rho_1,\rho_2,\rho_3\}$
given in a random order by ${\sf polroot(P)}$. Any change of root gives 
an inconsequential permutation $(\varepsilon, \varepsilon^\sigma) 
\mapsto (\varepsilon^\tau, \varepsilon^{\tau \sigma})$, for some $\tau \in G_K$. 

\smallskip
For security, we test ${\sf Reg^{}_1/Reg=1}$ where ${\sf Reg^{}_1}$ 
is the regulator of the units $\varepsilon(\rho)$ and $\varepsilon(g(\rho))$,
computed with the root $\rho$, and where ${\sf Reg=K.reg}$ is the true 
regulator given by PARI.

\smallskip
Then we must write the Leopoldt cyclotomic unit $\eta$ of $K$ of conductor $f$ 
(Definition \ref{defIII3}) under the form $\eta = \varepsilon^{\alpha + \beta \, \sigma}$, 
$\alpha, \beta \in \Z$, which is easy as soon as we have $\eta$ and $\eta^\sigma$.
But $\eta$ is computed by means of the analytic expression of 
$\vert \BC \vert = \prd_{a \in [1, f/2 [, \, \sigma_a{}_{\vert_K}=1}
\vert \zeta_{2f}^a-\zeta_{2f}^{-a} \vert$, as product of the 
$\vert \zeta_{2f}^a-\zeta_{2f}^{-a} \vert$ for the prime-to-$f$ integers $a < f/2$ 
such that the Artin symbol $\sigma_a = \big( \frac{\Q(\mu_f^{})/\Q}{a} \big)$ 
is in $\Gal(\Q(\mu_f^{})/K)$ (which is tested using a prime $q_a \equiv a \pmod f$ giving 
$\sigma_a{}_{\vert_K}=1$ if and only if $q_a$ splits in $K$). 

\smallskip
If $f$ is prime, $\zeta_{2f}-\zeta_{2f}^{-1}$ generates the prime ideal above $f$; 
thus: 
\[\pi := \Norm_{\Q(\mu_f^{})/K}(\zeta_{2f}-\zeta_{2f}^{-1}) = \pm \BC^2\]
with $\pi^3 = f \cdot \eta'$, $\eta' \in \BE_K$, whence 
$\pi^{3\,(1-\sigma)}=\eta'^{1-\sigma} = \eta^{6} := (\BC^{1-\sigma})^6$
(Proposition \ref{div}); the program computes $3 \log(\BC) - \frac{1}{2} \log(f)
= \frac{1}{2} \log(\eta')$, so that, to compute $\eta$ from 
$\eta^3 = \sqrt {\eta'}^{1-\sigma}$, we must divide the regulator ${\sf RegC}$ by $3$ 
and multiply $\alpha + j\,\beta$ by $\frac{1-j}{3}$ in that case where $w_\chi=1$.

\smallskip
If $f$ is composite, we have $\eta = \BC$ obtained via the half-system and 
the class number is the product of the index of units
by $w_\chi=3$, so this appear in the results (e.g., for the first 
example $f=13 \cdot 97$, $P=x^3+x^2-420x -1728$,
${\sf class group=[21]}$ and ${\sf Index\  [E_K:C_K]=7}$, 
but $\alpha + j\,\beta = -3-2j$ of norm $7$; for $f=3^2 \cdot 307$, 
$P=x^3-921 x-10745$, ${\sf class group=[21,3]}$ and ${\sf Index\  [E_K:C_K]=21}$, 
but $\alpha + j\,\beta = -5- j$ of norm $21$). 

\smallskip
To define the correct conjugation, $\zeta_{2f} \mapsto 
\zeta_{2f}^\sigma =: \zeta_{2f}^q$, for some prime $q$,
we use the fundamental property of Frobenius automorphisms giving
$y^{{\rm Frob}(q)} \equiv y^q \pmod q$, for any $q$-integer $y$ of $K$, if  
$q$ is inert in $K/\Q$; using $x^\sigma = g(x)$, we test the congruence 
$g(x) - x^q \pmod q$ to decide if $\sigma = {\rm Frob}(q)$ or ${\rm Frob}(q)^2$, 
in which case $\zeta_{2f}^\sigma = \zeta_{2f}^q$ or $\zeta_{2f}^{q^2}$, giving easily 
the conjugate $\eta^\sigma$.

\smallskip
The program and the numerical results are given in Appendix \ref{mainP}.

\section*{Conclusion} 
Standard probabilistic approaches may confirm (or not) the classical 
Cohen--Lenstra--Malle--Martinet heuristics on $p$-class groups, especially
in the non semi-simple case. Indeed, heuristics on the order of the whole 
$p$-class group of $K$ are given by that of the components $\CH^\ar_\varphi$'s 
which must be compatible with that obtained for the 
$(\CE_K / \wh \CE_K \!\cdot \CF_K)_{\varphi_0^{}}$'s;
a remarkable fact being that the structures are independent, but with
$(\CE_K / \wh \CE_K \!\cdot \CF_K)_{\varphi_0^{}}$ monogenic and 
$\CH^\ar_\varphi$ arbitrary as $\Z_p[\mu_{g_\chi^{}}]$-module, which means that heuristics 
on the \textit{structure} of $\CH^\ar_\varphi$ is another probabilistic problem
which clearly depends on that of the filtration studied in \cite{Gra666}
and accessible to probabilities in the spirit of Koymans--Pagano \cite{KoPa} 
and Smith \cite{Smi} techniques.  

\smallskip
Then, the main problem remains \textit {a proof of the Finite AMC in the non semi-simple 
real case} using the statement with Arithmetic $\varphi$-objects, especially a proof that 
for all abelian real field $K$, with a cyclic maximal $p$-sub-extension, we have,
for all $\varphi \in \Phi_K$ and $g_\chi$ non $p$-power (cf. \S\,\ref{wphi}):
\[\order \CH^\ar_\varphi = \order 
(\CE_K / \wh \CE_K \!\cdot \CF_K)_{\varphi_0^{}} ,\ \varphi =\varphi_0^{} \varphi_p.\]
where:
\[(\CE_K / \wh \CE_K \!\cdot \CF_K)_{\varphi_0^{}} = \{\wt \varepsilon \in 
(\CE_K / \wh \CE_K \!\cdot \CF_K),\, \wt \varepsilon^{P_\varphi(\sigma_\chi)} = 1\}
= (\CE_K / \wh \CE_K \!\cdot \CF_K)^{e_{\varphi_0^{}}}. \]

As we have explained in Remark \ref{remafond}, new tools using auxiliary 
cyclotomic extensions $K(\mu_\ell^{})$ and capitulation of $\CH_K$ in these extensions 
\textit {proves} the Finite Real Abelian Main Conjecture; unfortunately, the 
capitulation conjecture is not yet proved, but is very attractive since it governs 
several other arithmetic properties and we believe in this a lot.

\newpage

\appendix

\section{Numerical examples -- PARI programs}\label{appA}

As the referee pointed out to us, explicit computations in Galois fields
$K$ need to define an embedding of $\ov \Q$ in $\C$, especially with
PARI \cite{P}; so, let's recall that PARI works in $\Z[x]/(P)$ for irreducible
monic polynomials $P$ defining $K$ and gives a list
${\sf G=nfgaloisconj(P)}$, $\sigma \in G_K$ being defined by means of
$x \mapsto s(x)$, where $s(x) \in \Q[x]$ defines a (mysterious)
conjugate, but ${\sf nfgaloisapply(K,G[i],G[j])}$
(where ${\sf s_i=G[i], s_j=G[j]}$) computes $s_i s_j$, and so on. 

\smallskip
Similarly, ${\sf nfgaloisapply(K,G[i],E[j])}$ computes the corresponding
conjugate of the unit ${\sf E[j]}$.

\smallskip
For instance, for ${\sf P=x^3 - x^2 - 30*x - 27}$ ($K$ of conductor ${\sf 7*13}$),
PARI gives ${\sf G=[x, -1/3*x^2 + 1/3*x + 7, 1/3*x^2 - 4/3*x - 6]}$.

\smallskip
In other words, if one chooses a root ${\sf \rho}$ of $P$ (in the list 
${\sf polroots(P)}$), this defines an embedding and the evaluations 
$x \mapsto \rho$ in ${\sf G}$ allow suitable computations which, of course, 
depend numerically of ${\sf \rho}$.

\smallskip
Then, Leopoldt definitions work in $\Q(\zeta_f ) \subset \C$ by means of the 
choice of $\zeta_f := \exp \big(\ffrac{2 i \pi}{f}\big)$, generating the subfield $K$. 
This is problematic when one also defines $K$ via PARI since it is ugly to express 
$x$, formal root of $P$, in terms of roots of unity; so, in the programs, conjugates
of cyclotomic units are computed from the $\zeta_f$'s, and conjugates $\zeta_f^g$, 
while the units of $K$ are computed via the instruction ${\sf K.fu}$ and we must find
the correspondence of the two systems, which may be rough as we have 
explained \S\,\ref{nfgaloisconj}.
It is what we do in the forthcoming explicit examples when we say, 
for instance, that the ${\sf s1}$-conjugate of a cyclotomic unit ${\sf Eta}$ 
is ${\sf Eta\,\hat{}\,{s1}=945628377316488.87204143}$, and so on.
This explains that running the programs may give, for the user, results
different from ours, without any worries.

\subsection{Exceptional congruences}\label{P1}

The program verifies the exceptional congruence described
in Proposition \ref{AACK}, for the conductors $f$ up to $10^4$:
${}$
\ft\begin{verbatim}
{for(m=5,10^4,if(core(m)!=m,next);if(Mod(m,9)!=-3,next);
f=quaddisc(m);PP=x^2-f;PM=x^2+f/3;KP=bnfinit(PP,1);
KM=bnfinit(PM,1);hP=KP.no;hM=KM.no;E=lift(KP.fu[1]);
t=abs(polcoeff(E,0));u=abs(polcoeff(E,1));X=hP*t*u+hM;print
("f=",f," t=",t," u=",u," h=",hP," h'=",hM," htu+h'=",lift(Mod(X,3))))}    

f=24    t=5     u=1     h=1   h'=1   htu+h'=0
f=60    t=4     u=1/2   h=2   h'=2   htu+h'=0
f=33    t=23    u=4     h=1   h'=1   htu+h'=0
f=168   t=13    u=1     h=2   h'=4   htu+h'=0
f=204   t=50    u=7/2   h=2   h'=4   htu+h'=0
f=69    t=25/2  u=3/2   h=1   h'=3   htu+h'=0
(...)
\end{verbatim}\ns

\subsection{Numerical examples about the gap 
\texorpdfstring{$\CH^\ar_\chi$}{Lg} v.s.  \texorpdfstring{$\CH^\alg_\chi$}{Lg}}\label{ex12}
Let $k = \Q(\sqrt m)$ be a real quadratic field and let $K$ be the compositum 
of $k$ with a cyclic extension $L$ of $\Q$ of $p$-power degree; the field $K$ is 
of the form $K_\chi$ for $\chi \in \CX^+$ which is also irreducible $p$-adic.
We have given in \cite{Gra11} many examples of capitulations of $\CH_k$
in $K$, giving $\CH^\ar_\chi \varsubsetneqq \CH^\alg_\chi$.

\subsubsection{General PARI program}
One must precise the prime $p>2$, the minimal required 
$p$-rank ${\sf rpmin}$ of $\BH_k$, the length ${\sf N}$ of the sub-tower 
of $k(\mu_\ell^{})/k$
considered and the interval for ${\sf m}$ (the program uses primes $\ell$ 
(in ${\sf ell}$) congruent to 1 modulo $2p^N$, up to ${\sf Bell}$); the class group 
(resp. the $p$-class group) is computed in ${\sf Ck}$ (resp. ${\sf Ckp}$).
To compute $\J_{K/k}(\CH_k)$, we represent the $p$-classes of $k$ by
prime ideals ${\mathfrak q} \mid q$ inert in $K/k$.

\smallskip
\ft\begin{verbatim}
{p=3;rpmin=1;N=2;bm=2;Bm=10^4;Bell=10^4;for(m=bm,Bm,if(core(m)!=m,next);
P=x^2-m;k=bnfinit(P,1);Ck=k.clgp;r=matsize(Ck[2])[2];Ckp=List;Ekp=List;
rp=0;for(i=1,r,ei=Ck[2][i];vi=valuation(ei,p);if(vi>0,rp=rp+1;
ai=idealpow(k,Ck[3][i],ei/p^vi);listput(Ckp,ai,rp);
listput(Ekp,p^vi,rp)));if(rp<rpmin,next);L0=List;
for(i=1,rp,listput(L0,0,i));forprime(ell=2,Bell,
if(Mod(ell-1,2*p^N)!=0 || Mod(m,ell)==0,next);
Lq=List;for(i=1,rp,A=Ckp[i];forprime(q=2,10^5,if(q==ell,next);
if(kronecker(m,q)!=1 || Mod((ell-1)/znorder(Mod(q,ell)),p)==0,next);
F=idealfactor(k,q);qi=component(F,1)[1];cij=qi;for(j=1,Ekp[i]-1,
cij=idealmul(k,cij,A);if(Mod(j,p)==0,next);
if(List(bnfisprincipal(k,cij)[1])==L0,listput(Lq,q,i);break(2)))));
print("____");print();print("m=",m," ell=",ell," Lq=",Lq);
for(n=0,N,R=polcompositum(P,polsubcyclo(ell,p^n))[1];K=bnfinit(R,1);
print();print("C",n,"=",K.cyc);for(i=1,rp,Fi=idealfactor(K,Lq[i]);
Qi=component(Fi,1)[1];print(bnfisprincipal(K,Qi)[1])))))}
\end{verbatim}\ns

\smallskip
We shall consider the base field $k = \Q(\sqrt{4409})$
(i.e., $m=4409$ in the program) with $\ell = 19$, then $\ell = 1747$.

\subsubsection{Example 1}\label{Ex1}
Let $L$ be the degree $9$ subfield of 
$\Q(\mu_{19}^{})$; for convenience, put $k_0 := k$, $k_1 := L_1 k_0$ 
(resp. $k_2 := L_2 k_0$), where $L_1$ (resp. $L_2$) is the degree $3$
(resp.  $9$) subfield of $\Q(\mu_{19}^{})$.
The prime $2$ splits in $k_0$, is inert in $k_2/k_0$ and 
such that ${\mathfrak Q}_0 \mid 2$ in $k_0$ generates
$\CH_{k_0}$ (cyclic of order $9$); considering the extensions 
${\mathfrak Q}_i= \J_{k_i/k_0} ({\mathfrak Q}_0)$ of ${\mathfrak Q}_0$ in $k_i$, 
we test its order in $\CH_{k_i}$, $i=1,2$ (we are going to 
see that $\CH_{k_i} \simeq \Z/9\Z$ for all $i$, which is supported by the fact that 
$\Norm_{k_2/k_0}({\mathfrak Q}_2) = {\mathfrak Q}_0^9$ but 
$\Norm_{k_2/k_0}(\CH_{k_2}) =\CH_{k_0}$ since $k_2/k_0$ is totally 
ramified at $19$):

\smallskip
\ft\begin{verbatim}
    C0=[9]  [4]               C1=[9]  [6]               C2=[9]  [0]
\end{verbatim}\ns

\smallskip\noindent
where more precisely, ${\sf C0=[9]}$ denotes the class group of $k_0$ and, 
using the instruction ${\sf bnfisprincipal}$, ${\sf [4]}$ means that the class of 
${\mathfrak Q}_0 \mid 2$ is $h_0^4$, where $h_0$ is the generator (of order $9$) 
given in ${\sf kn.cyc}$ by PARI; then ${\sf C1=[9]}$, ${\sf [6]}$, is similar for 
$k_1$ in which we see a partial capitulation since the class of 
${\mathfrak Q}_1 = \J_{k_1/k_0} ({\mathfrak Q}_0)$ becomes of order $3$. 
Finally, ${\sf C2=[9]}$, ${\sf [0]}$ shows the complete capitulation in $k_2$; the 
$18$ large integers below are the coefficients, over the PARI integral basis, of a generator of 
${\mathfrak Q}_2$ in~$k_2$:

\smallskip
\ft\begin{verbatim}
[[0],[-270476874595642910,323533824277028894,-236208800298303000,
       119737461690335806,-255607858779215282,-198423813102857420,
       410588865020870414,-110028179006577678,-449600797918214026,
       -4906665437527948,10274048566854232,4319852458093887,
       13258715755947394,-6817941144899095,-15448507867705832,
       2623003974789062,-3264916449440532,-16606126998680345]]
\end{verbatim}\ns

We use obvious notations for the characters defining the fields $k_i$, $i=0,1,2$.
Since arithmetic norms are surjective (here, they are isomorphisms), the 
above computations prove that:
\[\Nu_{k_2/k_1}(\CH_{k_2}) = \J_{k_2/k_1} \circ \Norm_{k_2/k_1}(\CH_{k_2})
= \J_{k_2/k_1}(\CH_{k_1}) \simeq \Z/3\Z, \] 
since $\Norm_{k_2/k_1} \circ \J_{k_2/k_1}(\CH_{k_1}) = \CH_{k_1}^3$,
or simply $\J_{k_2/k_1}(\CH_{k_1}) = \CH_{k_2}^3$
(partial capitulation of $\CH_{k_1} \simeq \Z/9\Z$). 
Whence: 
\begin{equation*}
\left\{\begin{aligned}
\CH^\ar_{\chi^{}_2} & = \{x \in \CH_{k_2},\ \Norm_{k_2/k_1}(x)=1 \} = 1, \\
\CH^\alg_{\chi^{}_2} & = \{x \in \CH_{k_2},\ x^{P_{\chi^{}_2}(\sigma_{\chi^{}_2})} =1 \} \\
& = \{x \in \CH_{k_2},\  \Nu_{k_2/k_1}(x)=1\} = \CH_{k_2}^3 \simeq \Z/3\Z.
\end{aligned}\right.
\end{equation*}

We have $P_{\chi^{}_2}(\sigma_{\chi^{}_2}) = 
\sigma_{\chi^{}_2}^6+\sigma_{\chi^{}_2}^3+1 = \Nu_{k_2/k_1}$ 
(since $L$ is principal, the norms $\Nu_{k_i/L_i}$ does not intervene in the definition of
the $\CH^\alg_{\chi^{}_i}$'s). 

\smallskip
Similarly, we have: 
\[\Nu_{k_1/k_0}(\CH_{k_1})
= \J_{k_1/k_0} \circ \Norm_{k_1/k_0}(\CH_{k_1})= 
\J_{k_1/k_0}(\CH_{k_0}) \simeq \Z/3\Z\]
(partial capitulation of $\CH_{k_0} \simeq \Z/9\Z$); whence:
\begin{equation*}
\left\{\begin{aligned}
\CH^\ar_{\chi^{}_1} & = \{x \in \CH_{k_1},\ \Norm_{k_1/k_0}(x)=1\} =1, \\
\CH^\alg_{\chi^{}_1} & = \{x \in \CH_{k_1},\  \Nu_{k_1/k_0}(x)=1\} = \CH_{k_1}^3 \simeq \Z/3\Z.
\end{aligned}\right.
\end{equation*}

Thus, the formula of Theorem \ref{theoI5} giving:
\[\order \CH_{k_2} = 
\order \CH^\ar_{\chi_0^{}} \cdot \order \CH^\ar_{\chi_1^{}} \cdot \order \CH^\ar_{\chi_2^{}}\]
is of the form $\order \CH_{k_2} = 9 \times 1 \times 1$, then $\order \CH_{k_1} = 9 \times 1$
since $\CH^\ar_{\chi_0^{}} = \CH_{k_0^{}}$.

\smallskip
These formulas are not fulfilled in the algebraic sense, because:
\[\order \CH^\alg_{\chi^{}_0} \cdot\order \CH^\alg_{\chi^{}_1} = 9 \times 3 = 3^3
  \ \hbox{and}\ 
\order \CH^\alg_{\chi^{}_0} \cdot\order \CH^\alg_{\chi^{}_1} \cdot 
\order \CH^\alg_{\chi^{}_2} = 9 \times 3 \times 3 =3^4. \]

Now we intend to compute $\order \CH^\ar_{\chi^{}_1} = 
\order ( \CE_{k_1} / \wh \CE_{k_1} \cdot \CF_{\!k_1})$
(analytic formula of Theorem \ref{theoIII2}); in the general definition,
$\CF_{\!K}$ denotes the Leopoldt group of cyclotomic units of $K$,
$\wh \CE_K$ the group of units generated by the units of the strict subfields of $K$. 

\smallskip
We give numerical values of the units ${\sf \mid\! e0 \!\mid}$ of $k_0$, ${\sf \mid\! ei \!\mid}$ of
$L_1$, ${\sf \mid \!Ej\! \mid}$ of $k_1$, and their logarithms; 
they are, respectively (standard PARI programs):

\smallskip
\ft\begin{verbatim}
Units                                        Logarithms
e0=664.00150602068057486397714386165380      6.49828441757729630972016
e1=0.2851424818297853643941198735306274     -1.25476628739511494204754
e2=4.5070186440929762986607999237156780      1.50563588039686576534798
E1=0.2851424818297853643941198735306274     -1.25476628739511494204754
E2=0.2218761622631909342666800501850506     -1.50563588039686576534798
E3=664.00150602068057486397714386165380      6.49828441757729630972016
E4=945628377316488.87204143428389231544     34.4828707719825581974318
E5=0.0025736519075274654929993463127951     -5.96242941301396593243487

Cyclotomic units:
{f=19*4409;z=exp(I*Pi/f);g1=lift(Mod(74956,f)^2);g2=lift(Mod(4410,f)^3);
frob=1;for(s=1,6,frob=lift(Mod(3*frob,f));Eta=1;for(k=1,(4409-1)/2,
for(j=1,(19-1)/3,as=lift(Mod(g1^k*g2^j*frob,f));if(as>f/2,next);
Eta=Eta*(z^as-z^-as)));print("Eta^s",s,"=",Eta," ",log(abs(Eta))))}

Eta^s1=945628377316488.87204143428      34.4828707719825581974318471
Eta^s2=2433718277092.6834663091300      28.5204413589685922649969695
Eta^s3=0.0025736519075274654929993      -5.96242941301396593243487762
Eta^s4=1.0574978754738804652063 E-15   -34.4828707719825581974318471
Eta^s5=4.1089390231091111982824 E-13   -28.5204413589685922649969695
Eta^s6=388.55293409150677930552135       5.96242941301396593243487762
\end{verbatim}\ns

One obtains easily the following relations:

\ft\begin{verbatim}
E1=e1,   E2=e2^-1,   E3=e0,   E_4^2=Eta^s,   E5^2=Eta^-1,
Eta^{s^3+1}=1,   Eta^{s^2-s+1}=1, giving: Eta^(s^2)=E4^2.E5^2.
\end{verbatim}\ns

Then, one gets $(\CE_{k_1} : \wh \CE_{k_1} \cdot \CF_{\!k_1}) =
(\CE_{k_1}  : \CE_{k_0} \cdot \CE_{L_1} \cdot \CF_{\!k_1}) = 1$ 
as expected since $\CH^\ar_{\chi^{}_1}=1$.
Moreover, we see that the conjugates of the cyclotomic units are not independent
(due, from Lemma \ref{lemII10}, to norm relations in $k_i/k_0$ and $k_i/L_i$ since
$19$ splits in $k_0$ and $4409$ splits in the $L_i$'s), but, with our point of view, 
this does not matter since $\wh \CE_{k_1}$ is of $\Z_3$-rank $3$ and $\CF_{\!k_1}$ is of 
$\Z_3$-rank $2$.
Indeed, these relations lead to some difficulties in $\chi$-formulas of the literature
\textit {using larger groups of cyclotomic units} like Sinnott's cyclotomic units
(see Remark \ref{e0}). 

\smallskip
To be complete, compute the classical index of
$\CF_{\!k_0} =: \langle \eta_0^{} \rangle$ in $\CE_{k_0}$:

\ft\begin{verbatim}
{f=4409;z=exp(I*Pi/f);Eta0=1;g=znprimroot(f)^2;for(k=1,(f-1)/2,
a=lift(g^k);if(a>f/2,next);Eta0=Eta0*(z^a-z^-a)/(z^(3*a)-z^-(3*a)));
print("Eta0=",Eta0," log(Eta0)=",log(abs(Eta0)))} 
Eta0=3.985459685929 E-26     log(Eta0)=-58.484559758195
\end{verbatim}\ns

\smallskip\noindent
giving immediately ${\sf log(Eta0) = -9*log(e0)}$ from the above computation of 
${\sf log(e0)}$; whence $\order \CH^\ar_{\chi^{}_0} = 
(\CE_{k_0} : \wh \CE_{k_0} \cdot \CF_{\!k_0}) = (\CE_{k_0} : \CF_{\!k_0}) = 9$;  
obviously,  $9$ is the annihilator of $\CE_{k_0} / \CF_{\!k_0}$ 
and $\CH^\ar_{\chi^{}_0}$ (Conjecture~\ref{annihilationthm}).

\smallskip
The verification of $(\CE_{k_2} : \wh \CE_{k_2} \cdot \CF_{\!k_2}) = 1$ is analogous since
$\CF_{\!k_2}$ is of $\Z_3$-rank $8$ ($\Norm_{k_2/k_1}(\CF_{\!k_2}) = \CF_{\!k_1}$,
$\Norm_{k_2/k_0}(\CF_{\!k_2}) = 1$, $\Norm_{k_2/L_2}(\CF_{\!k_2}) = 1$).

\subsubsection{Example 2}\label{Ex2}
Consider the same framework, replacing $19$ by the prime $1747$; one obtains the data
showing, as before with ${\mathfrak Q}_0 \mid 2$, a partial capitulation of $\CH_{k_0}$ in 
$k_1$ (but $\CH_{k_1}$ is not cyclic):

\ft\begin{verbatim}
C0=[9]   [4]           C1=[9,3,3]   [6,0,0]
\end{verbatim}\ns
One verifies that the ideal ${\sf Q_1}$, extending ${\sf Q_0}$ in $k_1$, is 
non-principal and such that its class is ${\sf h_1^6\, h_2^0\, h_3^0}$ on the 
PARI basis ${\sf \{h_1,\, h_2,\, h_3\}}$:

\smallskip
\ft\begin{verbatim}
bnfisprincipal(K,[2, [-1,0,0,1,0,0],1,3,[0,0,0,1,0,0]]) = [[6,0,0]
\end{verbatim}\ns

\noindent
but its $6$-power gives as expected the principality and a generator:

\smallskip
\ft\begin{verbatim}
bnfisprincipal(K,[64,0,0,21,0,0;0,64,0,0,0,42;0,0,64,0,21,0;0,0,0,1,0,0;
                                                0,0,0,0,1,0;0,0,0,0,0,1])
=[[0,0,0],[8217190756304871153969213,526028282779527429138218,
            -687786029075595676594134,251301709772155482917577,
            -21032376402967976888126,-15609327127430752932511]]
\end{verbatim}\ns

\smallskip
The kernel of the arithmetic norm is isomorphic to $\Z/3\Z \times \Z/3\Z$, thus:
\begin{equation*}
\left\{\begin{aligned}
\CH^\ar_{\chi^{}_1} & = \{x \in \CH_{k_1},\ \Norm_{k_1/k_0}(x)=1\} \simeq \Z/3\Z \times \Z/3\Z, \\
\CH^\alg_{\chi^{}_1} & = \{x \in \CH_{k_1},\ \, \Nu_{\,k_1/k_0}(x)=1\}  \simeq 
\Z/3\Z \times \Z/3\Z \times \Z/3\Z.
\end{aligned}\right.
\end{equation*}

\noindent
since the transfer map applies 
$\CH^\ar_{\chi^{}_0} \simeq \Z/9\Z$ onto $\langle h_1^6 \rangle$.

\smallskip
Formula of Theorem \ref{theoI5} is of the form
$\order \CH_{k_1} = \order \CH^\ar_{\chi^{}_0} \cdot \order \CH^\ar_{\chi^{}_1} = 9 \times 9$, 
since we have $\CH^\ar_{\chi^{}_0} = \CH_{k_0}$ of order $9$; 
of course a same formula with the $\CH^\alg$'s does not exist since 
$\order \CH^\alg_{\chi^{}_0} \cdot \order \CH^\alg_{\chi^{}_1} = 9 \times 27$.

\subsubsection{Varying $\ell \equiv 1 \pmod 9$}\label{Ex3} The program gives the 
following other results, for $k = \Q(\sqrt{4409})$, varying only ${\sf ell}$, where 
${\sf q}$ is the prime split in $k_0=k$ and inert in $k_2$:

\smallskip
\ft\begin{verbatim}
ell=37  q=2    C0=[9]  [4]    C1=[18]   [6]      C2=[18]    [0]
ell=73  q=2    C0=[9]  [4]    C1=[9]    [6]      C2=[171]   [0]
ell=109 q=5    C0=[9]  [1]    C1=[9]    [6]      C2=[9]     [0]
ell=127 q=23   C0=[9]  [4]    C1=[9]    [6]      C2=[9]     [0]
ell=163 q=2    C0=[9]  [4]    C1=[54]   [12]     C2=[54]    [18]
ell=181 q=2    C0=[9]  [4]    C1=[27]   [12]     C2=[81]    [63]
ell=199 q=2    C0=[9]  [4]    C1=[9,3]  [6,0]    C2=[27,3]  [9,0]
\end{verbatim}\ns

\smallskip\noindent
The image of $\CH_{k_0}$ in $k_1$ is of order $3$, except for
$\ell \in \{163, 181\}$; then $\CH_{k_0}$ capitulates in $k_2$, except for 
$\ell \in \{163, 181, 199\}$. One verifies that formula of Theorem \ref{theoI5} holds
with the $\order \CH^\ar_{k_i}$ but not for the $\order \CH^\alg_{k_i}$.

\subsection{Computation of \texorpdfstring{$\order \BH_\chi$}{Lg} for 
\texorpdfstring{$K = \Q(\mu_{47}^{})$}{Lg}}\label{ex47}

Let $K := K_\chi$ be the field $\Q(\mu_{47}^{})$, of degree $g^{}_\chi = 46$. 
From Theorem \ref{theoII2}, we have $\order \BH_\chi = 
2^{\alpha_\chi} \cdot w_\chi \cdot \prod_{\psi \mid \chi}
\big (- \frac{1}{2} \BB_1(\psi^{-1}) \big)$ with in that case $\alpha_\chi=0$ and $w_\chi = 47$
and where by definition: 
\[- \ffrac{1}{2} \BB_1(\psi^{-1}) 
= - \ffrac{1}{2} \sm_{a=1}^{46} \big({\ffrac{a}{47} - \ffrac{1}{2} \big)\psi^{-1}(\sigma_a)}
=  - \ffrac{1}{2} \sm_{a=1}^{46} {\ffrac{a}{47}\,\psi^{-1}(\sigma_a)}. \]
Let's compute $\order \BH_\chi = 47 \cdot 
\Norm_{\Q(\mu_{46}^{})/\Q}\big (- \ffrac{1}{2} \sm_{a=1}^{46} 
{\ffrac{a}{47}\,\psi^{-1}(\sigma_a)} \big )$:

\smallskip
\ft\begin{verbatim}
{P=polcyclo(46);g=lift(znprimroot(47));A=0;for(n=0,45,
a=lift(Mod(g,47)^n);A=A+x^n*(1/47*a-1/2));B=Mod(-1/2*A,P);
print("47*Norm(B)=",47*norm(B))}
47*Norm(B)=139
\end{verbatim}\ns

Note that $-\frac{47}{2} \BB_1(\psi^{-1})$ is, writing $x=\zeta_{46}$, the PARI integer:

\smallskip
\ft\begin{verbatim}
4*x^21+25*x^20+9*x^19+26*x^18-19*x^17+11*x^16-22*x^15
+x^14-24*x^13+10*x^12+6*x^11+16*x^10-21*x^9+20*x^8
+8*x^7+7*x^6-4*x^5+14*x^4-12*x^3+3*x^2+14*x+27
\end{verbatim}\ns

Whence $\order \BH_\chi = 139$ and $\BH_\chi \simeq \Z[\mu_{46}^{}]/{\mathfrak p}_{139}$.
Since $\Lambda_\chi = 47$, the ideal ${\mathfrak A}_K$ is $\big(\sigma_a-a, 47 \big)$, 
with for instance $a=5$ (Lemma \ref{lemII7}), and ${\mathfrak A}_K \cdot \ffrac{1}{2} \BB_K$ 
annihilates $\BH_\chi$; since the image of ${\mathfrak A}_K \cdot \ffrac{1}{2} \BB_K$ 
is the ideal $\big( \ffrac{1}{2} \BB_1(\psi^{-1}) \big) = {\mathfrak p}_{139}$, 
\textit {the annihilator} of $\BH_\chi$ is ${\mathfrak p}_{139}$. 
But this ideal is not principal in $\Q(\mu_{46}^{})$ (from \cite{Gra00}):

\smallskip
\ft\begin{verbatim}
{L=bnfinit(polcyclo(46));F=idealfactor(L,139);
print(bnfisprincipal(L,component(F,1)[1])[1])}
[2]
\end{verbatim}\ns

\noindent
showing that its class is the square of the PARI generating class.
More precisely, the class group of $\Q(\mu_{46}^{}) = \Q(\mu_{23}^{})$
is equal to $3$; then any ${\mathfrak q}_{47} \mid 47$ or 
${\mathfrak q}_{139} \mid 139$ generates this class group.

\subsection{Computation of annihilators of torsion groups 
\texorpdfstring{$\CT_K$}{Lg}}\label{annihilators}

Consider, for $p=7$, the cubic field $K$ of conductor $f=2557$ defined by
the polynomial $P=x^3 + x^2 - 852\,x + 9281$; then (using the main 
program of Appendix \ref{mainP}), one obtains:
\[\hbox{$\CH_K \simeq  \Z[j]/ (1-2 j) \Z[j]\ $  and $\ \CE_K/\CF_K 
\simeq \Z[j]/ (1-2 j) \Z[j] $}, \] 
where $(1-2 j) \Z[j]$ is a prime ${\mathfrak p}$ dividing $7$, and
$\CT_K \simeq \Z/7^2\Z \oplus \Z/7\Z$ .

\smallskip
The following program (only valid for prime conductors $f$) 
computes the annihilator $\BA_{K}(c)$ of $\CT_K$;
it defines the classes $\sigma^k \cdot \Gal(\Q(\mu_{f p^N}^{})/K)$, $k=0,1,2$,
of Artin symbols, giving $\BA_{K}(c) = A_0+A_1\sigma +A_2 \sigma^2$, then $\beta :=
A_0-A_2 + (A_1-A_2)\,j$, yielding $(\beta) = {\mathfrak p}_1^u \cdot {\mathfrak p}_2^v$
in $\Z[j]$ (up to a prime-to-$p$ ideal):

\smallskip
\ft\begin{verbatim}
{p=7;f=2557;N=4;pN=p^N;fpN=f*pN;c=lift(znprimroot(f));cm=Mod(c,fpN)^-1;
g=znprimroot(f);lg=lift(Mod((1-lift(g))/f,pN));g=Mod(lift(g)+lg*f,fpN);
g3=g^3;G=znprimroot(pN);lG=lift(Mod((1-lift(G))/pN,f));
G=Mod(lift(G)+lG*pN,fpN);A0=0;A1=0;A2=0;for(k=1,(f-1)/3,
for(j=1,p^(N-1)*(p-1),A=g3^k*G^j;gA=g*A;ggA=g^2*A;
a=lift(A);aa=lift(A*cm);la=(aa*c-a)/fpN;A0=A0+la*Mod(a,pN)^-1;
a=lift(gA);aa=lift(gA*cm);la=(aa*c-a)/fpN;A1=A1+la*Mod(a,pN)^-1;
a=lift(ggA);aa=lift(ggA*cm);la=(aa*c-a)/fpN;A2=A2+la*Mod(a,pN)^-1));
print(A0," ",A1," ",A2)}

Mod(184, 2401)   Mod(1526, 2401)   Mod(643, 2401)
\end{verbatim}\ns

\smallskip
Modulo $7^4$, $A_0=184$, $A_1=1526$ and $A_2=643$; this yields 
the ideal $(1-2 j)^3 = {\mathfrak p}^3$. Necessarily, $\CT_K \simeq
\Z[j]/ {\mathfrak p}^2  \oplus \Z[j]/ {\mathfrak p}$.
We note that the annihilator is ${\mathfrak p}^3$ (and not ${\mathfrak p}^2$) 
although the structure is not $\Z[j]/{\mathfrak p}^3$.

\subsection{Computation of the invariants of 
\texorpdfstring{$\psi(\Omega_\ell)$}{Lg}} \label{Omega}

The program computes, for cyclic cubic fields, the invariants 
$\psi(\Omega_\ell) = r_1-r_2 - (r_1 + 2\,r_2)\cdot j$ only with the 
knowledge of $\eta_K^{}$; taking a primitive root $\g_\ell^{}$ 
modulo $\ell$, the $r_\sigma$'s come from the PARI instructions 
${\sf r=znlog(L[j],g)}$, where the ${\sf L[j]}$ are the rationals $a_\sigma$ 
such that $\eta_K^\sigma \equiv a_\sigma \pmod {{\mathfrak l}_0}$ 
in $K$ (we use the results of Appendix \ref{examples}\,(c) 
to compute $\eta_K^{} = \varepsilon_K^{\alpha+\beta\,\sigma}$
and $\BH_K$). The line ${\sf Orders\ of\ components\ of\ cl(Lell)}$
of the form $(p^u, p^v, \cdots)$ means that the components of the $p$-class 
of ${\mathfrak l}_0$ (on the PARI system of generators of $\CH_K$),
are of orders $p^u$, $p^v,\,\cdots$; one sees that the annihilator 
$\Omega_\ell$ is independent on these orders, but it is clear that, using 
Chebotarev's theorem, any set of components may be obtained.

\ft\begin{verbatim}
{p=7;n=3;P=x^3+x^2-884540*x-393129;alpha=-112;beta=-70;
Q=y^2+y+1;k=bnfinit(Q);J=Mod(y,Q);pi=idealfactor(k,p);
pi1=component(pi,1)[1];pi2=component(pi,1)[2];
K=bnfinit(P,1);G=nfgaloisconj(P);CK=K.cyc;d=matsize(CK)[2];
CKp=List;for(i=1,d,h=p^valuation(CK[i],p);listput(CKp,h,i));
print("P=",P," p-class group=",CKp);
E=K.fu;E1=E[1];E2=nfgaloisapply(K,G[2],E[1]);
F1=E1^alpha*E2^beta;F2=nfgaloisapply(K,G[2],F1);
F1=lift(F1);F2=lift(F2);forprime(ell=1,5*10^5,
if(Mod(ell,p^n)!=1 || matsize(factor(P+O(ell)))[1]!=3,next);
g=znprimroot(ell);Lell=component(idealfactor(K,ell),1)[1];
F10=Mod(polcoeff(F1,0),ell);F11=Mod(polcoeff(F1,1),ell);
F12=Mod(polcoeff(F1,2),ell);Eta1=lift(F12*x^2+F11*x+F10);
F20=Mod(polcoeff(F2,0),ell);F21=Mod(polcoeff(F2,1),ell);
F22=Mod(polcoeff(F2,2),ell);Eta2=lift(F22*x^2+F21*x+F20);
Leta=List;listput(Leta,Eta1,1);listput(Leta,Eta2,2);L=List;
for(i=1,2,A=Mod(Leta[i],P);for(a=1,ell-1,v=idealval(K,A-a,Lell);
if(v>0,listput(L,a,i))));Lr=List;for(i=1,2,r=znlog(L[i],g);
listput(Lr,r));print();print("ell=",ell," Omega=",Lr);
X=Lr[1]-Lr[2]+(-Lr[1]-2*Lr[2])*J;
w1=idealval(k,X,pi1);w2=idealval(k,X,pi2);
Y=alpha+beta*J;W1=idealval(k,Y,pi1);W2=idealval(k,Y,pi2);print
("Cyclotomic invariants=",W1,",",W2," Omega invariants=",w1,",",w2);
Exp=List;Order=bnfisprincipal(K,Lell)[1];for(i=1,d,
tp=valuation(CK[i],p);if(Order[i]==0,Or=1);if(Order[i]!=0,
t=valuation(Order[i],p);Or=p^(tp-t));listput(Exp,Or));
print("Orders of components of cl(Lell)=",Exp))}
\end{verbatim}\ns

\smallskip
For $P=x^3+x^2-884540*x-393129$ (conductor 
$f = 2653621$, $\alpha=-112$, $\beta=-70$, the $\varphi$-components
of the $7$-class group $\CH_K$ are  
$\CH_{\varphi^{}_1} \simeq \Z_7[j]/{\mathfrak p}_{\varphi^{}_1}$ and 
$\CH_{\varphi^{}_2} \simeq \Z_7[j]/{\mathfrak p}_{\varphi^{}_2}^3$;
we have $\wt \CE_{\varphi^{}_1} \simeq 
\Z_7[j]/{\mathfrak p}_{\varphi^{}_1}$ and $\wt \CE_{\varphi^{}_2} 
\simeq \Z_7[j]/{\mathfrak p}_{\varphi^{}_2}^3$.

\ft\begin{verbatim}
P=x^3+x^2-884540*x-393129 p-class group=List([343,7])
conductor f=2653621

ell=1373 Omega=List([1162, 1246])
Cyclotomic invariants=1,3 Omega invariants=1,3
Orders of components of cl(Lell)=List([343, 7])

ell=7547 Omega=List([6888, 1526])
Cyclotomic invariants=1,3 Omega invariants=1,3
Orders of components of cl(Lell)=List([343, 7])

ell=8233 Omega=List([6496, 742])
Cyclotomic invariants=1,3 Omega invariants=1,3
Orders of components of cl(Lell)=List([49, 7])

ell=18523 Omega=List([11830, 12586])
Cyclotomic invariants=1,3 Omega invariants=1,3
Orders of components of cl(Lell)=List([343, 1])

ell=22639 Omega=List([4004, 13104])
Cyclotomic invariants=1,3 Omega invariants=1,3
Orders of components of cl(Lell)=List([343, 7])

ell=30871 Omega=List([27734, 5390])
Cyclotomic invariants=1,3 Omega invariants=2,3
Orders of components of cl(Lell)=List([343, 1])

ell=39103 Omega=List([32018, 35812])
Cyclotomic invariants=1,3 Omega invariants=1,3
Orders of components of cl(Lell)=List([49, 7])

ell=42533 Omega=List([1330, 17262])
Cyclotomic invariants=1,3 Omega invariants=1,3
Orders of components of cl(Lell)=List([343, 7])

ell=54881 Omega=List([44366, 18662])
Cyclotomic invariants=1,3 Omega invariants=1,3
Orders of components of cl(Lell)=List([49, 7])

ell=58997 Omega=List([5236, 21938])
Cyclotomic invariants=1,3 Omega invariants=1,3
Orders of components of cl(Lell)=List([343, 7])

ell=72031 Omega=List([24276, 51884])
Cyclotomic invariants=1,3 Omega invariants=1,3
Orders of components of cl(Lell)=List([343, 7])

ell=76147 Omega=List([17066, 25606])
Cyclotomic invariants=1,3 Omega invariants=1,3
Orders of components of cl(Lell)=List([343, 7])

ell=80263 Omega=List([22036, 79352])
Cyclotomic invariants=1,3 Omega invariants=1,3
Orders of components of cl(Lell)=List([343, 7])

ell=93983 Omega=List([69174, 5558])
Cyclotomic invariants=1,3 Omega invariants=1,3
Orders of components of cl(Lell)=List([343, 7])
\end{verbatim}\ns

\smallskip
For $P=x^3-4792107\,x+4022175142$ (conductor 
$f = 3^2 \cdot 1597369$, $\alpha=-7$, $\beta=-21$, the $\varphi$-components
of the $7$-class group $\CH_K$ are  
$\CH_{\varphi^{}_1} \simeq \Z_7[j]/{\mathfrak p}_{\varphi^{}_1} \oplus 
\Z_7[j]/{\mathfrak p}_{\varphi^{}_1}$ and 
$\CH_{\varphi^{}_2} \simeq \Z_7[j]/{\mathfrak p}_{\varphi^{}_2}$;
nevertheless, we have $\wt \CE_{\varphi^{}_1} \simeq 
\Z_7[j]/{\mathfrak p}_{\varphi^{}_1}^2$ (non-isomorphic to $\CH_{\varphi^{}_1}$) 
and $\wt \CE_{\varphi^{}_2} \simeq \Z_7[j]/{\mathfrak p}_{\varphi^{}_2}$.

\smallskip
But almost all $\Omega_\ell$ give the expected response ${\sf (2, 1)}$
whatever the order of the $p$-class of ${\mathfrak l}_0 \mid \ell$:

\smallskip
\ft\begin{verbatim}
P=x^3 - 4792107*x + 4022175142 p-class group=List([7,7,7])
conductor f=9*1597369

ell=1373 Omega=List([917, 1267])
Cyclotomic invariants=2,1 Omega invariants=2,1
Orders of components of cl(Lell)=List([7, 7, 7])

ell=8233 Omega=List([1141, 3535])
Cyclotomic invariants=2,1 Omega invariants=2,1
Orders of components of cl(Lell)=List([7, 1, 7])

ell=49393 Omega=List([41069, 39277])
Cyclotomic invariants=2,1 Omega invariants=2,1
Orders of components of cl(Lell)=List([1, 7, 1])

ell=54881 Omega=List([14357, 31311])
Cyclotomic invariants=2,1 Omega invariants=2,2
Orders of components of cl(Lell)=List([7, 7, 7])

ell=63799 Omega=List([53977, 53767])
Cyclotomic invariants=2,1 Omega invariants=2,1
Orders of components of cl(Lell)=List([7, 7, 7])

ell=76147 Omega=List([44912, 73514])
Cyclotomic invariants=2,1 Omega invariants=2,1
Orders of components of cl(Lell)=List([1, 7, 7])

ell=80263 Omega=List([20328, 16387])
Cyclotomic invariants=2,1 Omega invariants=3,1
Orders of components of cl(Lell)=List([1, 7, 7])
               (...)
ell=329281 Omega=List([311136, 189770])
Cyclotomic invariants=2,1 Omega invariants=2,1
Orders of components of cl(Lell)=List([7, 7, 7])

ell=331339 Omega=List([157696, 276465])
Cyclotomic invariants=2,1 Omega invariants=2,1
Orders of components of cl(Lell)=List([7, 7, 7])

ell=343687 Omega=List([174391, 82173])
Cyclotomic invariants=2,1 Omega invariants=2,2
Orders of components of cl(Lell)=List([7, 7, 7])

ell=363581 Omega=List([204974, 276584])
Cyclotomic invariants=2,1 Omega invariants=2,1
Orders of components of cl(Lell)=List([7, 7, 7])

ell=384847 Omega=List([254100, 68887])
Cyclotomic invariants=2,1 Omega invariants=2,1
Orders of components of cl(Lell)=List([7, 7, 7])

ell=396509 Omega=List([114947, 1540])
Cyclotomic invariants=2,1 Omega invariants=2,1
Orders of components of cl(Lell)=List([7, 7, 7])

ell=403369 Omega=List([11361, 206458])
Cyclotomic invariants=2,1 Omega invariants=2,1
Orders of components of cl(Lell)=List([7, 7, 7])

ell=408857 Omega=List([364287, 259343])
Cyclotomic invariants=2,1 Omega invariants=5,1
Orders of components of cl(Lell)=List([7, 7, 1])

ell=415717 Omega=List([239225, 363657])
Cyclotomic invariants=2,1 Omega invariants=2,1
Orders of components of cl(Lell)=List([7, 1, 7])

ell=417089 Omega=List([327908, 33957])
Cyclotomic invariants=2,1 Omega invariants=3,4
Orders of components of cl(Lell)=List([1, 7, 7])

ell=419147 Omega=List([17059, 339451])
Cyclotomic invariants=2,1 Omega invariants=2,1
Orders of components of cl(Lell)=List([1, 1, 1])

ell=426007 Omega=List([161434, 215859])
Cyclotomic invariants=2,1 Omega invariants=2,1
Orders of components of cl(Lell)=List([7, 7, 7])

ell=456877 Omega=List([361697, 10010])
Cyclotomic invariants=2,1 Omega invariants=3,1
Orders of components of cl(Lell)=List([7, 7, 7])
\end{verbatim}\ns

For $\ell = 419147$ (first example where 
any prime ideal ${\mathfrak l} \mid \ell$ is principal): 

\ft\begin{verbatim}
bnfisprincipal(K,Lell)=[0,0,0],[1311001361541054679,35057663364174,
                                                        1019317530188062]
\end{verbatim}\ns

\noindent
but the invariants of $\Omega_\ell$ are still $(2,1)$ giving 
$\order \CH_{\varphi_1} = 7^2$ and $\order \CH_{\varphi_2} = 7$.

\subsection{Illustrations of the Finite AMC}

We intend to illustrate the Finite AMC with cyclic cubic fields and $p \equiv 1 \pmod 3$
giving two $p$-adic characters (of course, it is now a Theorem and we shall 
speak of the ``Finite AMT''); then statistics may have some interest.

\subsubsection{The general PARI program}\label{mainP}
The program is the following and we explain, with some examples, how to
use the numerical results checking the Finite AMT; ${\sf hmin=p^{vp}}$ 
means that the program only computes fields with $p$-class groups 
${\sf CKp}$ of order at least ${\sf p^{vp}}$; then ${\sf bf, Bf}$ define 
an interval for the conductors ${\sf f}$ of the cyclic cubic field.
Other indications are given in the text of the program: 
 
\smallskip
\ft\begin{verbatim}
\p 50
{p=7; \\ Take any prime p congruent to 1 modulo 3
bf=2;Bf=10^6;hmin=p^2;
\\ Arithmetic of Q(j), j^2+j+1=0:
S=y^2+y+1;kappa=bnfinit(S);Y=idealfactor(kappa,p);
 \\ Decomposition (p)=P1*P2 in Z[j]:
P1=component(Y,1)[1];P2=component(Y,1)[2];
\\ Iteration over the conductors f in [bf,Bf]:
for(f=bf,Bf,vf=valuation(f,3);if(vf!=0 & vf!=2,next);
F=f/3^vf;if(core(F)!=F,next);F=factor(F);Div=component(F,1);
d=matsize(F)[1];for(j=1,d,D=Div[j];if(Mod(D,3)!=1,break));
\\ Computation of solutions a and b such that f=(a^2+27*b^2)/4:
\\ Iteration over b, then over a:
for(b=1,sqrt(4*f/27),if(vf==2 & Mod(b,3)==0,next);A=4*f-27*b^2;
if(issquare(A,&a)==1,
\\ computation of the corresponding defining polynomial P:
if(vf==0,if(Mod(a,3)==1,a=-a);P=x^3+x^2+(1-f)/3*x+(f*(a-3)+1)/27);
if(vf==2,if(Mod(a,9)==3,a=-a);P=x^3-f/3*x-f*a/27);
K=bnfinit(P,1); \\ PARI definition of the cubic field K
\\ Test on the p-class number #CKp regarding hmin:
if(Mod(K.no,hmin)==0,print();
G=nfgaloisconj(P); \\ Definition of the Galois group G
\\ Frob = Artin symbol defining the PARI generator sigma=G[2]:
forprime(q=2,10^4,if(Mod(f,q)==0,next);
Pq=factor(P+O(q));if(matsize(Pq)[1]==1,Frob=q;break));X=x^Frob-G[2];
if(valuation(norm(Mod(X,P)),Frob)==0,Frob=lift(Mod(Frob^2,f)));
E=K.fu;Reg=K.reg; \\ Group of units, Regulator
\\ We certify that a suitable PARI unit is a Z[G]-generator of E_K:
E1=lift(E[1]);E2=lift(nfgaloisapply(K,G[2],E[1]));
Root=polroots(P);Rho=real(Root[1]); \\  Selecting a root of P
e1= abs(polcoeff(E1,0)+polcoeff(E1,1)*Rho+polcoeff(E1,2)*Rho^2);
e2= abs(polcoeff(E2,0)+polcoeff(E2,1)*Rho+polcoeff(E2,2)*Rho^2);
l1=log(e1);l2=log(e2);Reg1=l1^2+l1*l2+l2^2;quot=Reg1/Reg;
print(quot); \\ This quotient must be equal to 1
\\ Computation of the cyclotomic units C1,C2=sigma(C1):
z=exp(I*Pi/f);C1=1;C2=1;
\\ Case of a prime conductor f using (Z/fZ)^* cyclic):
if(isprime(f)==1,g=znprimroot(f)^3;
\\ Description of a half-system:
for(k=1,(f-1)/6,gk=lift(g^k);sgk=lift(Mod(gk*Frob,f));
C1=C1*(z^gk-z^-gk);C2=C2*(z^sgk-z^-sgk));
 \\ Logarithms of C1,C2:
L1=3*log(abs(C1))-log(f)/2;L2=3*log(abs(C2))-log(f)/2;
\\ computation of the cyclotomic regulator and of the index Quot=(E:F):
RegC=L1^2+L1*L2+L2^2;Quot=1/3*RegC/Reg); \\ Division by 3 of RegC
\\ Case of a composite conductor:
if(isprime(f)==0,for(aa=1,(f-1)/2,if(gcd(aa,f)!=1,next);
\\ Search of a prime qa congruent to a modulo f, split in K:
qa=aa;while(isprime(qa)==0,qa=qa+f);
if(matsize(idealfactor(K,qa))[1]==1,next);
\\ The Artin symbol of aa fixes K:
C1=C1*(z^aa-z^-aa);C2=C2*(z^(Frob*aa)-z^-(Frob*aa)));
L1=log(abs(C1));L2=log(abs(C2)); \\ Logarithms of C1,C2
\\ computation of the cyclotomic regulator and the index Quot=(E:F):
RegC=L1^2+L1*L2+L2^2;Quot=RegC/Reg);
\\ printing of the basic data of K:
print("P=",P," f=",f,"=",factor(f)," (a,b)=","(",a,",",b")",
" class group=",K.cyc,"  sigma=",Frob);print("Index [E_K:C_K]=",Quot); 
\\ Annihilator alpha+sigma.beta of the quotient E/C:
alpha=((log(e1)+log(e2))*L1+log(e2)*L2)/Reg;
beta=(log(e2)*L1-log(e1)*L2)/Reg;
 \\ In the prime case one multiply alpha+j.beta by (1-j)/3:
if(isprime(f)==1,
alpha0=(alpha+beta)/3;
beta0=(-alpha+2*beta)/3;alpha=alpha0;beta=beta0);
\\ Writing of alpha and beta as reals for checking:
print("(alpha,beta)=","(",alpha,", ",beta,")");
\\ Computation of alpha and beta as integers:
alpha=sign(alpha)*floor(abs(alpha)+10^-6);
beta=sign(beta)*floor(abs(beta)+10^-6);
\\ Class group (r = global rank;rp = p-rang;expo = exposant of CKp)
\\ vp = valuations of CKp, ve = valuation of the exponent expo of CKp:
CK=K.clgp;r=matsize(CK[2])[2];CKp=List;EKp=List;rp=0;vp=0;ve=0;
for(i=1,r,ei=CK[2][i];vi=valuation(ei,p);
if(vi>0,rp=rp+1;vp=vp+vi;ve=max(ve,vi));expo=p^ve;
\\ The rp following ideals Ai generate the p-class group CKp:
Ai=idealpow(K,CK[3][i],ei/p^vi);listput(CKp,Ai,i);listput(EKp,p^vi,i));
\\ Matrices h and sh of Ai and sAi on the PARI basis of CK
L0=List;for(i=1,r,listput(L0,0,i));LH=List;LsH=List;
for(i=1,rp,Ai=CKp[i];h=bnfisprincipal(K,Ai)[1];
sAi=nfgaloisapply(K,G[2],Ai);sh=bnfisprincipal(K,sAi)[1];
print("h=",h,", ","sigma(h)=",sh);listput(LH,h,i);listput(LsH,sh,i));
\\ Determination of the Pi-valuations of (alpha+j.beta), i=1,2:
Z=Mod(alpha+y*beta,S);w1=idealval(kappa,Z,P1);w2=idealval(kappa,Z,P2);
print(w1,"  ",w2,"  P1 and P2-valuations for alpha+j*beta");
\\ Galois structure of CKp; computation of the phi-components:
if(rp==1,
u=lift(LsH[1][1]*Mod(LH[1][1],expo)^-1);
YY=Mod(y-u,S);v1=idealval(kappa,YY,P1);v2=idealval(kappa,YY,P2);
v1=min(v1,ve);v2=min(v2,ve);
print(v1,"  ",v2,"  P1 and P2-valuations for H"));
if(rp==2,
\\ Computation of ci(mod expo) such that Pi=(ci+j),i=1,2:
Sp=lift(factor(S+O(p^ve)));Sp1=component(Sp,1)[1];Sp2=component(Sp,1)[2];
c1=polcoeff(Sp1,0);c2=polcoeff(Sp2,0);
\\ Coefficients of LH[1],LsH[1],LH[2],LsH[2], on the PARI basis of CK
H1=LH[1];A1=H1[1];B1=H1[2];sH1=LsH[1];C1=sH1[1];D1=sH1[2];
H2=LH[2];A2=H2[1];B2=H2[2];sH2=LsH[2];C2=sH2[1];D2=sH2[2];
\\ Computation of the determinants of the relations:
Delta1=((C1+c1*A1)*(D2+c1*B2)-(D1+c1*B1)*(C2+c1*A2));
Delta1=lift(Mod(Delta1,expo));
Delta2=((C1+c2*A1)*(D2+c2*B2)-(D1+c2*B1)*(C2+c2*A2));
Delta2=lift(Mod(Delta2,expo));
print(Delta1,"  ",Delta2,"   Determinants: Delta1,Delta2");
\\ Computation of the relations defining the phi-components:
r11x=C1+c1*A1;r11y=C2+c1*A2;r12x=D1+c1*B1;r12y=D2+c1*B2;
r11x=lift(Mod(r11x,expo));r11y=lift(Mod(r11y,expo));
r12x=lift(Mod(r12x,expo));r12y=lift(Mod(r12y,expo));
r21x=C1+c2*A1;r21y=C2+c2*A2;r22x=D1+c2*B1;r22y=D2+c2*B2;
r21x=lift(Mod(r21x,expo));r21y=lift(Mod(r21y,expo));
r22x=lift(Mod(r22x,expo));r22y=lift(Mod(r22y,expo));
print("R11=",r11x,"*X+",r11y,"*Y","  R12=",r12x,"*X+",r12y,"*Y");
print("R21=",r21x,"*X+",r21y,"*Y","  R22=",r22x,"*X+",r22y,"*Y"));
\\ Structure of the torsion group Tp of p-ramification:
n=6; \\ Choose any n, large enough, such that p^(n+1) annihilates Tp:
LTp=List;Kpn=bnrinit(K,p^n);Hpn=Kpn.cyc;
dim=component(matsize(Hpn),2);for(k=2,dim,c=component(Hpn,k);
if(Mod(c,p)==0,listput(LTp,p^valuation(c,p),k)));
print("Structure of the ",p,"-torsion group: ",LTp)))))}
\end{verbatim}\ns

\subsubsection{Numerical examples}\label{examples}
Since the approximations are in general very good (with precision ${\sf \!\backslash p\ 50}$), 
we have suppressed useless decimals in the results. But for some conductors, 
the precision ${\sf \!\backslash p\ 100}$ may be necessary, because of a fundamental 
unit close to $0$ (e.g., $f=21193, \,30223$).
For $f=42667$, ${\sf \!\backslash p\ 100}$ does not compute correctly and 
${\sf \!\backslash p\ 150}$ gives a nice result for $\alpha$ and $\beta$; 
but we see that, for this example:

\smallskip
\ft${\sf e_1=3062171948818717694.348000505806}$\ns \ \ 
and \ \  \ft${\sf e_2=1.221295564694 E-69.}$\ns

\medskip\noindent
\textbf{Galois structure of \texorpdfstring{$\CE_K/\CF_K$}{Lg}.}
Let $\varepsilon$ be the $\Z[G_K]$-generator of $\BE_K$ and let $\eta$ that 
of the sub\-group $\BF_K$ of Leopoldt's cyclotomic units; thus we have 
$\eta = \varepsilon^{\alpha + \beta\,\sigma}$ and obtain the isomorphism:
\[\BE_K/\BF_K \simeq \Z[j]/(\alpha + j\,\beta)\Z[j], \] 
where $j$ is a root of ${\sf S := y^2+y+1}$. 

\smallskip
In all the sequel, from a factorization $p=(r_1+ j \, r'_1)\cdot(r_2+ j \, r'_2)$ 
giving the ideal product $(p)={\mathfrak p}_1 {\mathfrak p}_2$ in $\Z[j]$, we 
associate, for the exponent $p^e$, the two annihilators $c_i+\sigma$ such that 
$(c_i+j) = {\mathfrak p}_i^e$ (up to a prime-to-$p$ ideal); this preserves 
the definition of the $\varphi^{}_1$ and $\varphi^{}_2$-components.

\smallskip
For instance, for $p=7$, ${\mathfrak p}_1 := (-2+j)\Z[j]$ and 
${\mathfrak p}_2 := (3+j)\Z[j]$; writing $(\alpha + j\,\beta) =: 
{\mathfrak p}_1^u\cdot {\mathfrak p}_2^v \cdot {\mathfrak a}$,
${\mathfrak a}$ prime to $7$, we get immediately the two 
$\varphi$-components of $\wt \CE_K=\CE_K/\CF_K$ (e.g., if $e=2$,
the two annihilators are $19+j$ and $-18+j$, respectively; for 
$p=13$, we get $23+j$ and $-22+j$). 

\medskip\noindent
\textbf{Galois structure of \texorpdfstring{$\CH_K$}{Lg}.}
Recall that ${\sf bnfisprincipal(K,A)[1]}$
gives the matrix of components of the class of ${\sf A}$ on the basis 
$\{h_1, \ldots, h_{r}\}$ given by ${\sf K.clgp}$ (in ${\sf CK}$) and the fact 
that ${\sf 0}$ at the place $i$ means that the corresponding component of 
${\sf cl(A)}$ on $h_i$ is trivial.

\smallskip
We first replace the generators of $\BH_K$ by generators ${\sf Ai}$
of $\CH_K$ (where $r_p \leq r$ is the $p$-rank).
The Galois action on the ${\sf Ai}$ is computed using the instructions
(where ${\sf G[2]}$ gives the $\sigma$-conjugate, ${\sf G[1]}$ being the identity):
\ft\begin{verbatim}
h=bnfisprincipal(K,Ai)[1];sAi=nfgaloisapply(K,G[2],Ai);   
sh=bnfisprincipal(K,sAi)[1]}; 
\end{verbatim}\ns

\noindent
so the Galois structure of $\CH_K$ becomes linear algebra from the matrices 
given by the program, via the relations:
\[\hbox{$h=\prod_{i=1}^{r_p} h_i^{a_i}$ (in ${\sf h}$)\ \ \ \&\ \ \
$h^\sigma=\prod_{i=1}^{r_p} h_i^{b_i}$ (in ${\sf sh}$).} \]

(a) \textbf{Case of $7$-rank $r_7=1$}.
This case is obvious, writing $h=h_1^a$, $h^\sigma=h_1^b$;
we put $P_{\varphi^{}_1} \equiv c_1+y \pmod{7^e}$ and 
$P_{\varphi^{}_2} \equiv c_2+y \pmod{7^e}$, where $7^e$ is the exponent of $\CH_K$; 
we obtain $h^{c_1+\sigma}=h_1^{c_1 a+b}$ and $h^{c_2+\sigma}=h_1^{c_2 a+b}$; so
$\CH_K=\CH_{\varphi^{}_1}$ (resp. $\CH_{\varphi^{}_2}$) if and only if 
$c_1 a+b\equiv 0 \pmod {7^e}$ (resp. $c_2 a+b \equiv 0 \pmod {7^e}$). 
In fact one computes $-a^* b+j$, where $a^*$ is inverse of $a$ 
modulo $7^e$, and write $(-a^* b+j) = {\mathfrak p}_i^u$ for the suitable 
$i \in \{1,2\}$. 

\smallskip
The Galois actions are to be read in columns; for instance, the valuations
in the two lines:
\begin{equation*}
\begin{aligned}
\sf v\ \ \ \, & 0 \ \ \ \ \   \sf P1\  and\  P2-valuations\  for \  alpha+j*beta \\
\sf v\ \ \ \, & 0 \ \ \ \ \   \sf P1\  and\  P2-valuations\  for \ \ \ H
\end{aligned}
\end{equation*}

\noindent
give the structures $\Z[j]/{\mathfrak p}_1^v \cdot {\mathfrak p}_2^0$ 
for ``$\CM = \wt \CE = \CE/\CF$ and $\CH$\,'', respectively, whence 
$\CM_{\varphi_1^{}} \simeq \Z[j]/{\mathfrak p}_1^v$, 
$\CM_{\varphi_2^{}}=1$, and so on. First examples:

\medskip
\ft\begin{verbatim}
P=x^3+x^2-104*x+371  f=313=Mat([313,1])  (a,b)=(35,1)  
Class group=[7]  sigma=4
(alpha,beta)=(-3.000000000,-2.000000000)  Index [E_K:C_K]=7.000000000
h=[1],  sigma(h)=[2]
1  0   P1 and P2-valuations for alpha+j*beta
1  0   P1 and P2-valuations for  H
Structure of the 7-torsion group: List([7,7])
\end{verbatim}\ns 

\smallskip
We have $\wt \CE_{\varphi^{}_1} \simeq \CH_{\varphi^{}_1} \simeq 
(\Z[j]/{\mathfrak p}_1) \otimes \Z_7 \simeq \Z/7\Z$ and the conjugation $h^\sigma = h^2$,
giving the annihilator $(-2+j) = {\mathfrak p}_1$ as expected;
whence the two columns given by the program.
We deduce that $\CT_K = \CH_K \oplus \CR_K$.

\medskip
\ft\begin{verbatim}
P=x^3+x^2-2450*x-1089  f=7351=Mat([7351,1])  (a,b)=(-1,33)  
Class group=[49]  sigma=4
(alpha,beta)=(5.000000000,8.000000000)  Index [E_K:C_K]=49.000000000
h=[1],  sigma(h)=[30]
2  0   P1 and P2-valuations for alpha+j*beta
2  0   P1 and P2-valuations for  H
Structure of the 7-torsion group: List([2401])
\end{verbatim}\ns 

\smallskip
We have $(\alpha + j\,\beta) = (5+8 j)$, thus the annihilator $(19+j)=
{\mathfrak p}_1^2$; then $h^\sigma = h^{30}$ gives (modulo~$7^2$)
the same annihilator. The $\varphi^{}_2$-components are trivial.
Since $\CT_K \simeq \Z/7^4\Z$, $\CR_K = \CT_K^{7^2}$,
$\CH_K \simeq \CT_K/\CR_K \simeq \Z/7^2\Z$.

\medskip
The first field such that $\CH_K \simeq \Z/7^3\Z$ is the following:

\ft\begin{verbatim}
P=x^3+x^2-77006*x-34225  f=231019=Mat([231019,1])  (a,b)=(-1,185) 
Class group=[343]  sigma=4
(alpha,beta)=(19.000000000,18.000000000)  Index [E_K:C_K]=343.000000000
h=[1],  sigma(h)=[18]
0  3   P1 and P2-valuations for alpha+j*beta
0  3   P1 and P2-valuations for  H
Structure of the 7-torsion group: List([343,7])
\end{verbatim}\ns 

\smallskip
The annihilator of $\CH_K$ is $(-18+j) = {\mathfrak p}_2^3$.
The structures are similar with the $\varphi^{}_2$-components since
$(19+18 j) = {\mathfrak p}_2^3$.
In that case, $\CT_K = \CH_K \oplus \CR_K$ with $ \CH_K \simeq \Z/7^3\Z$
and $\CR_K \simeq \Z/7\Z$.

\medskip
(b) \textbf{Case of $7$-rank $r_7=2$} This case depends on the matrices giving:
\[\hbox{${\sf h=[a, b],\  sigma(h)=[c,d]}\ \ \ $ \& $\ \ \ {\sf h'=[a',b'],\   sigma(h')=[c',d']}$;}\]
this means that the corresponding generating classes $h$, $h'$, fulfill the relations (regarding 
the basis $\{h_1, h_2\}$ of the class group) $h = h_1^a \cdot h_2^b$ and 
$h^\sigma = h_1^{c} \cdot h_2^{d}$, then $h' = h_1^{a'} \cdot h_2^{b'}$ and 
$h'^\sigma = h_1^{c'} \cdot h_2^{d'}$.
Thus we compute the conditions $H^{c_i+\sigma}=1$, $i=1,2$, for $H:= h^x\cdot h'^y$;
this gives the relations ${\sf R11}$, ${\sf R21}$ (${\sf R12}$, ${\sf R22}$ are checked 
by security since they must be proportional to the previous ones); 
whence the arrangement of lines when the conjecture holds.
The program computes the corresponding determinants of the relation
(${\sf Determinants\ Delta1\ Delta2}$); this is superfluous but have been computed  
(but not printed) for verification.

\medskip
\ft\begin{verbatim}
P=x^3+x^2-3422*x-1521  f=10267=Mat([10267,1])  (a,b)=(-1,39) 
Class group=[7,7]  sigma=2
(alpha,beta)=(-7.000000000,-7.000000000)  Index [E_K:C_K]=49.000000000
h=[1,0],  sigma(h)=[0,1]
h'=[0,1], sigma(h')=[6,6]
1  1   P1 and P2-valuations for alpha+j*beta
R11=3*X+6*Y  R12=1*X+2*Y
R21=5*X+6*Y  R22=1*X+4*Y
Structure of the 7-torsion group: List([49,7])
\end{verbatim}\ns 

\smallskip
This case means that $\wt \CE_K \simeq \Z[j]/(7)$, giving the
two non trivial $\varphi$-components of order $7$.
The relations, for $\CH_K$, reduce to ${\sf R11}$ and  ${\sf R21}$
Thus $\CH_K=\CH_{\varphi^{}_1} \oplus \CH_{\varphi^{}_2} \simeq \Z/7\Z \oplus \Z/7\Z$,
$\CR_K  = \CT_K^7 \simeq \Z/7\Z$.

\medskip
\ft\begin{verbatim}
P=x^3+x^2-55296*x-1996812  f=165889=[19,1;8731,1]  (a,b)=(-322,144) 
Class group=[294,2,2,2] sigma=25
(alpha,beta)=(-32.000000000,-20.000000000) Index [E_K:C_K]=784.000000000
h=[6,0,0,0], sigma(h)=[108,1,0,0]
0  2   P1 and P2-valuations for alpha+j*beta
0  2   P1 and P2-valuations for  H
Structure of the 7-torsion group: List([49])
\end{verbatim}\ns 

\smallskip
Here $\CR_K=1$ and $\CT_K = \CH_K \simeq 
(\Z[j]/{\mathfrak p}_2^2) \otimes \Z_7 \simeq \Z_7/7^2 \Z_7$.

\medskip
\ft\begin{verbatim}
P=x^3+x^2-453576*x+117425873  f=1360729=Mat([1360729,1])  (a,b)=(2333,1) 
Class group=[98,14]  sigma=2
(alpha,beta)=(42.000000000,28.000000000)  Index [E_K:C_K]=1372.000000000
h=[1,0],  sigma(h)=[44,11]
h'=[0,1], sigma(h')=[7,11]
2  1   P1 and P2-valuations for alpha+j*beta
R11=14*X+7*Y  R12=11*X+30*Y      
R21=26*X+7*Y  R22=11*X+42*Y
Structure of the 7-torsion group: List([49,7,7])
\end{verbatim}\ns 

\smallskip
We have $(\alpha + \beta j) = 2 \cdot 7 (3+2 j)$ giving the annihilator 
${\mathfrak p}_1^2{\mathfrak p}_2$ which is also the annihilator of $\CH_K$.
The structure is $\CT_K = \CH_K \oplus \CR_K$.

\medskip
\ft\begin{verbatim}
P=x^3+x^2-884540*x-393129  f=2653621=Mat([2653621,1])  (a,b)=(-1,627) 
Class group=[686,14]  sigma=2
(alpha,beta)=(-112.00000000,-70.00000000)  Index [E_K:C_K]=9604.00000000
h=[2,0],  sigma(h)=[36,2]
h'=[0,2], sigma(h')=[0,4]
1   3   P1 and P2-valuations for alpha+j*beta
R11=74*X+0*Y  R12=2*X+42*Y      
R21=0*X+0*Y  R22=2*X+311*Y
Structure of the 7-torsion group: List([343,49])
\end{verbatim}\ns 

\smallskip
In that case, $\CT_K \simeq Z/7^3\Z \oplus \Z/7^2\Z$ and $\CR_K \simeq 
(\Z/7^3\Z)^0 \oplus(7\,\Z/7^2\Z)$.

\medskip
(c) \textbf{Larger $7$-ranks}. If the order $7^3$, with $7$-rank $1$ or $2$, is 
rather frequent for the $7$-class group, we find, after several days of computer,
only three examples of $7$-rank $3$ in the interval $f \in [7, 50071423]$;
they are obtained with the conductors $f =14376321, 39368623, 43367263$,
giving interesting structures (use precision ${\sf \backslash  p \ 100}$).
The least cubic field with $7$-rank $3$ is the following:

\smallskip
\ft\begin{verbatim}
P=x^3-4792107*x+4022175142 f=14376321=[3,2;1597369,1] (a,b)=(-7554,128) 
Class group=[21,7,7] sigma=5
(alpha,beta)=(-7.000000000,-21.000000000) Index [E_K:C_K]=343.000000000
h =[3,0,0], sigma(h) =[15,4,0]
h'=[0,1,0], sigma(h')=[3,1,0]
h"=[0,0,1], sigma(h")=[6,5,2]
2   1   P1 and P2-valuations for alpha+j*beta
Structure of the 7-torsion group: List([7,7,7])
\end{verbatim}\ns

\smallskip
Using the information on $\alpha$ and $\beta$, we obtain, for 
$\wt \CE_K = \CE_K/\CF_K$:
\begin{equation*}
\wt \CE_K\simeq (\Z[j]/ 7  {\mathfrak p}_2) \otimes \Z_7  \simeq 
(\Z[j]/{\mathfrak p}_1^2 \, {\mathfrak p}_2) \otimes \Z_7  
\simeq (\Z[j]/{\mathfrak p}_1^2 \oplus \Z[j]/ {\mathfrak p}_2) \otimes \Z_7, 
\end{equation*}
where ${\mathfrak p}_1 = (-2+j)$ and ${\mathfrak p}_2=(3+j)$.
We get the $\varphi$-components:

\medskip
\centerline{$\wt \CE_{\varphi^{}_1}  \simeq  (\Z[j]/{\mathfrak p}_1^2) \otimes \Z_7\simeq \Z/7^2\Z$ and
$\wt \CE_{\varphi^{}_2} \simeq (\Z[j]/{\mathfrak p}_2) \otimes \Z_7\simeq \Z/7\Z$.} 

\medskip
To obtain the two $\varphi$-components of $\CH_K = \CT_K$, we put $H=h^x h'^yh''^z$ 
and we determine the solutions of the two relations$H^{P_{\varphi^{}_i}(\sigma)}=1$, 
$i=1, 2$, that is to say, $H^{-2+\sigma} = 1$ and $H^{3+\sigma} = 1$, 
respectively. 

\smallskip
We then obtain the systems (considered modulo $7$ since the exponent of $\CH_K$ is $7$)
of ranks 1 and 2, respectively:
\begin{equation*}
\left\{\begin{aligned}
2x+3y+6z & =0 \\
4x+6y+5z & =0 \\
\end{aligned}\right. \, \hbox{($H^{-2+\sigma} = 1)$\ \   \&\ \ }
\left\{\begin{aligned}
3x+3y+6z & = 0 \\
4x+4y+5z & =0 \\
z & =0,
\end{aligned}\right. \, \hbox{($H^{3+\sigma} = 1)$.}
\end{equation*}

They are equivalent to:
\[2x+3y+6z  =0 \  (\hbox{$H^{-2+\sigma} \!= 1$})  \ \ \ \&
\ \ \  \big[ x+y=0 \ \  \&\  \ z=0 \big] \  (\hbox{$H^{3+\sigma}\! = 1$}).\]

Which gives, considering the $\F_7$-dimensions given by the systems: 
\[\CH_{\varphi^{}_1} \simeq \big[ (\Z[j]/{\mathfrak p}_1) \otimes \Z_7 \big]
\oplus \big[(\Z[j]/{\mathfrak p}_1) \otimes \Z_7\big]
 \ \ \, \& \ \ \, \CH_{\varphi^{}_2} \simeq (\Z[j]/{\mathfrak p}_2)\otimes \Z_7. \] 
We have indeed equalities for the orders of the $\varphi$-components
relative to $\wt \CE_K$ and $\CH_K$, respectively, but of course with 
different structures of $\Z_7[j]$-modules since 
$\wt \CE_{\varphi^{}_1} \simeq \Z/7^2 \Z$ and 
$\CH_{\varphi^{}_1} \simeq \big [\Z/7 \Z \big]^2$.

\smallskip
The two other examples are similar:

\smallskip
\ft\begin{verbatim}
P=x^3+x^2-13122874*x-7765825411 
f=39368623=[7,1;79,1;71191,1] (a,b)=(-5323,2187) 
class group=[21,21,7] sigma=4
(alpha,beta)=(28.000000000,-7.000000000) Index [E_K:C_K]=1029.000000000
h =[3,0,0],  sigma(h) =[3,9,0]
h'=[0,3,0],  sigma(h')=[18,15,0]
h"=[0,0,1],  sigma(h")=[15,6,4]
1   2  P1 and P2-valuations for alpha+j*beta
Structure of the 7-torsion group: List([7,7,7])

P=x^3+x^2-14455754*x-16977480367 
f=43367263=[43,1;1008541,1] (a,b)=(-10567,1513) 
class group=[273,7,7] sigma=2
(alpha,beta)=(42.000000000,77.000000000) Index [E_K:C_K]=4459.000000000
h =[39,0,0],  sigma(h) =[0,5,1]
h'=[0,1,0],   sigma(h')=[156,6,5]
h"=[0,0,1],   sigma(h")=[0,0,2]
2   1  P1 and P2-valuations for alpha+j*beta
Structure of the 7-torsion group: List([49,7,7])
\end{verbatim}\ns

\medskip
(d) \textbf{Larger primes $p$}. Let's give, without comments, some examples:

\ft\begin{verbatim}
p=13 P=x^3+x^2-15196*x-726047 f=45589=Mat([45589,1]) (a,b)=(-427,1) 
Class group=[169] sigma=2
(alpha,beta)=(15.000000000,8.000000000) Index [E_K:C_K]=169.000000000
h=[1], sigma(h)=[146]
2  0   P1 and P2-valuations for alpha+j*beta
2  0   P1 and P2-valuations for H
Structure of the 13-torsion group: List([169])
\end{verbatim}\ns
\ft\begin{verbatim}
p=13 P=x^3+x^2-238516*x-7579519 f=715549=Mat([715549,1]) (a,b)=(-283,321) 
Class group=[13,13] sigma=2
(alpha,beta)=(7.000000000,-8.000000000) Index [E_K:C_K]=169.000000000
h =[1,0], sigma(h) =[9,0]
h'=[0,1], sigma(h')=[0,9]
0  2   P1 and P2-valuations for alpha+j*beta
R11=0*X+0*Y  R12=0*X+0*Y      
R21=6*X+0*Y  R22=0*X+6*Y
Structure of the 13-torsion group: List([13,13])

p=19 P=x^3-137271*x+45757 f=411813=[3,2;45757,1] (a,b)=(-3,247) 
Class group=[1083] sigma=2
(alpha,beta)=(-21.000000000,-5.000000000) Index [E_K:C_K]=361.000000000
h=[3], sigma(h)=[204]
0  2   P1 and P2-valuations for alpha+j*beta
0  2   P1 and P2-valuations for H
Structure of the 19-torsion group: List([361])
\end{verbatim}\ns
\ft\begin{verbatim}
p=19 P=x^3+x^2-162636*x+25190561 f=487909=[31,1;15739,1] (a,b)=(1397,1) 
Class group=[57,19] sigma=2
(alpha,beta)=(19.00000000,4.19514516 E-69) Index [E_K:C_K]=361.00000000
h =[3,0],  sigma(h) =[51,16]
h'=[0,1],  sigma(h')=[3,1]
1  1   P1 and P2-valuations for alpha+j*beta
R11=18*X+3*Y  R12=16*X+9*Y      
R21=11*X+3*Y  R22=16*X+13*Y
Structure of the 19-torsion group: List([19,19])

p=31 P=x^3+x^2-63804*x+6181931 f=191413=Mat([191413,1]) (a,b)=(875,1) 
class group=[31,31] sigma=4
(alpha,beta)=(31.00000000,-4.10842850 E-69) Index [E_K:C_K]=961.00000000
h=[1,0],  sigma(h) =[30,30]
h'=[0,1], sigma(h')=[1,0]
1  1  P1 and P2-valuations for alpha+j*beta
R11=5*X+1*Y  R12=30*X+6*Y      
R21=25*X+1*Y  R22=30*X+26*Y
Structure of the 31-torsion group: List([31,31])
\end{verbatim}\ns
\ft\begin{verbatim}
p=31 P=x^3+x^2-76004*x-8090239 f=228013=Mat([228013,1]) (a,b)=(-955,1) 
class group=[961] sigma=2
(alpha,beta)=(-11.000000000,-35.000000000) Index [E_K:C_K]=961.000000000
h=[1], sigma(h)=[439]
2  0  P1 and P2-valuations for alpha+j*beta
2  0  P1 and P2-valuations for H
Structure of the 31-torsion group: List([961])
\end{verbatim}\ns

\section*{Acknowledgements} \noindent
I would like to thank very warmly the anonymous Referee for a large 
number of comments and suggestions which have greatly improved the readability 
of this survey and enabled clarifications or corrections.

\end{document}